\documentclass[a4paper,onecolumn,11pt,unpublished]{quantumarticle}
\pdfoutput=1

\usepackage{aux_tikz}

\begin{document}

% To make sure that we all cite things in the same way
\renewcommand{\ref}[1]{{\color{red}[use \texttt{cref} not \texttt{ref}]}}

\title{Cyclic functional causal models beyond unique solvability with a graph separation theorem}

\author{Carla Ferradini$^1$}
\author{Victor Gitton$^1$}
\author{V.\ Vilasini$^{1,2}$}
\affiliation{$^1$Institute for Theoretical Physics, ETH Zurich, 8093 Zürich, Switzerland\\
$^2$Université Grenoble Alpes, Inria, 38000 Grenoble, France}
\maketitle

\begin{abstract}
\small
Functional causal models (fCMs) specify functional dependencies between random variables associated to the vertices of a graph. 
In directed acyclic graphs (DAGs), fCMs are well-understood: a unique probability distribution on the random variables can be easily specified, and a crucial graph-separation result called the $d$-separation theorem allows one to characterize conditional independences between the variables. 
However, fCMs on cyclic graphs pose challenges due to the absence of a systematic way to assign a unique probability distribution to the fCM's variables, the failure of the $d$-separation theorem, and lack of a generalization of this theorem that is applicable to all consistent cyclic fCMs. 
In this work, we develop a causal modeling framework applicable to all cyclic fCMs involving finite-cardinality variables, except inconsistent ones admitting no solutions. 
Our probability rule assigns a unique distribution even to non-uniquely solvable cyclic fCMs and reduces to the known rule for uniquely solvable fCMs. 
We identify a class of fCMs, called averagely uniquely solvable, that we show to be the largest class where the probabilities admit a Markov factorization. 
Furthermore, we introduce a new graph-separation property, $p$-separation, and prove this to be sound and complete for all consistent finite-cardinality cyclic fCMs while recovering the $d$-separation theorem for DAGs. 
These results are obtained by considering classical post-selected teleportation protocols inspired by analogous protocols in quantum information theory. 
We discuss further avenues for exploration, linking in particular problems in cyclic fCMs and in quantum causality.
\end{abstract}

\newpage
\tableofcontents
\newpage 

\section{Introduction}
Providing causal explanations for observable correlations and understanding how causal relations constrain possible observations are fundamental across scientific disciplines. To achieve this goal, the framework of causal models was formulated for rigorously connecting cause-effect relations and observable data between random variables. Originating in the context of classical statistics, this formalism has found diverse applications in data-driven fields such as machine learning, economics, biological systems and medical trials~\cite{Raita_2021,Kleinberg_2011,Pearl_2009,Spirtes_2005,Petersen_2014,Arti_2020,Liu_2021}.

Classical causal modeling frameworks involve a representation through directed graphs, whose vertices are associated with random variables and edges with causal relations, and incorporate causal mechanisms, such as functional dependencies between variables. More specifically, these are called \emph{functional causal models} (fCMs), as well as \emph{structural equation/causal models}.\footnote{We refer to these as classical causal modeling frameworks to distinguish them from quantum causal modeling frameworks (e.g., \cite{Henson_2014, Costa_2016, Barrett_2019}) where causal mechanisms are linked to quantum channels on quantum systems. This paper will focus on classical causal models, while our accompanying paper \cite{Quantum_paper} develops a similar framework for the quantum case.} 
The majority of the causal modeling literature focuses on \emph{acyclic} graphs, where there exists a well-defined probability rule to evaluate correlations over observed variables for all causal mechanisms and crucial graph-theoretic properties, such as the $d$-separation theorem~\cite{Verma1990, Geiger1990, Pearl_2009}, hold. This theorem is a cornerstone for causal discovery and inference methods as it allows to tightly relate observed conditional independences to the connectivity of the causal graph.

However, these results no longer hold for classical causal models on \textit{cyclic} graphs, which are important to consider to effectively describe physical processes involving feedback~\cite{Pearl_1996,Forre_2017,Bongers_2021} and to study informational models for exotic solutions of general relativity involving causal loops~\cite{vanStockum_1938,Godel1949,Matzner_1967, Deutsch1991,Lloyd_2011, Lloyd_2011_2}. Existing cyclic causal modeling frameworks consider the probability distribution associated with a causal model to be well defined only for a restricted class of models where the functional dependences admit unique solutions (uniquely solvable models). Moreover, it is known that for cyclic graphs, even when restricting to the class of uniquely solvable models involving finite-cardinality variables, the $d$-separation theorem fails \cite{Neal_2000}.

Studying the limits of $d$-separation and identifying alternative graph-separation properties that hold in the cyclic case has emerged as an active research area in classical causal modeling and there has been significant progress in understanding this from different (stronger) unique solvability properties~\cite{Forre_2017, Bongers_2021}. Notably, the concept of $\sigma$-separation was introduced in~\cite{Forre_2017} and shown to be sound and complete in a subclass of uniquely solvable models known as modular structural equation models, which have no constraint on the cardinality or discreteness of the variables involved. However, to our knowledge, a general graph-separation property that applies to all (including non-uniquely solvable) consistent causal models is lacking, even in the case of finite-cardinality variables. This is complicated by the fact that existing methods do not uniquely fix the probability distribution of a classical causal model in non-uniquely solvable models. 

In this work, we address both the problems of fixing the observed probability distribution and of identifying a graph-separation property for all consistent classical cyclic causal models involving finite-cardinality variables. Our approach reveals connections between problems in classical functional causal models and various results in the quantum information community, which we hope will serve as an invitation for collaboration between the communities. We summarize the main contributions below.

\begin{enumerate}
    \item {\bf Framework for cyclic classical models:} We develop a framework for classical cyclic causal models (cyclic fCMs) analogous to, yet independent of, a framework for quantum cyclic causal models introduced in our companion paper~\cite{Quantum_paper}. This framework encompasses all classical causal models definable on discrete variables of finite cardinality, excluding only the most pathological ones where the functional dependencies admit no solutions. 
    \item {\bf Probability rule:} We construct a method for uniquely fixing the probability distribution for all causal models in our framework, which recovers the previously known probability rule for uniquely solvable models. In \cite{VilasiniColbeckPRL} a method for computing the probability distribution was suggested for one specific non-uniquely solvable cyclic causal model involved in demonstrating that the relativistic principle of no superluminal communication does not forbid causal loops in $(1+1)$-Minkowski space-time. Our probability rule also recovers this case and provides a generalization of the method suggested in \cite{VilasiniColbeckPRL}. 
    \item {\bf A graph-separation property for cyclic fCMs inspired by quantum teleportation:} We introduce a novel graph-separation property for cyclic fCMs, called $p$-separation. We prove that $p$-separation is sound and complete for all fCMs in our framework, and reduces to $d$-separation in the acyclic case. To our knowledge, this is the first sound and complete graph-separation property known for non-uniquely solvable fCMs over finite-cardinality variables. This is achieved by mapping any cyclic fCM on finite-cardinality variables to an acyclic causal model involving post-selection. This mapping introduces the concept of a classical post-selected teleportation protocol, inspired by the quantum teleportation protocol \cite{Bennett1993} that is a key concept in the quantum information community. 

    \item {\bf On correlation gaps between graph-separation properties and finite vs infinite cardinality variables:} We compare $p$-separation introduced here with $\sigma$-separation~\cite{Forre_2017}, and propose a new direction for future research on characterizing correlation gaps between graph-separation properties: for a given graph $\graphname$, and two graph-separation properties that imply distinct sets of conditional independences on correlations, is it possible to witness this gap through a class of causal models on $\graphname$? In particular, our results indicate that there exists a graph within which no gap between $d$ and $\sigma$ separation can be witnessed for any causal model in our framework. This suggests that infinite-cardinality or continuous variables would be necessary to witness the gap, and indeed there exist such examples of continuous variable models on this graph \cite{Forre_2017, forre_2018}. More generally, this motivates exploration of correlation gaps between finite vs infinite cardinality fCMs on a given cyclic causal structure, analogous to the highly studied problem of certifying correlation gaps between classical and quantum correlations in an acyclic causal structure (of which the fundamental Bell’s theorem~\cite{Bell_1964} of quantum mechanics is an instance). 

\item {\bf Solvability properties:} We analyze solvability properties of cyclic fCMs, linking the number and existence of solutions to conditions on post-selection success probability. We identify a class of models satisfying a property called average unique solvability — a strict superset of uniquely solvable models — and show that this property is necessary and sufficient for recovering the usual probability rule based on Markov factorization. Our findings connect cyclic functional causal models, studied in the statistics community, with classical processes without definite causal order \cite{Baumeler_2016}, originating in the quantum information community. We note striking similarities between the solvability properties of these models and results~\cite{Baumeler_2016_fixedpoints} on the fixed points of classical processes.

\end{enumerate}

For a brief overview of our framework and probability rule, as well as an introduction to our new graph-separation property, $p$-separation, we recommend referring to the example in~\cref{example: probability rule detail} and~\cref{sec: psep and examples}, respectively.

\paragraph{Notation.}
We denote with $\graphname = \graphexpl$ a directed graph where each edge $\edgename = \edgearg{\vertname_1}{\vertname_2} \in \edgeset$ is an ordered pair of two vertices $\vertname_1,\vertname_2\in\vertset$. We shall always assume, for convenience, that the set of vertices $\vertset$ is equipped with a preferred order, such that we can write $\vertset = \{\vertname_1,\dots,\vertname_\vertcount\}$ where $\vertcount = |\vertset|$.
The incoming and outgoing edges to a vertex $\vertname\in\vertset$ are denoted
\begin{subequations}
\begin{align}
    \inedges{\vertname} &= \biglset
    { e\in\edgeset }{ \exists \vertname'\in\vertset \st e = \edgearg{\vertname'}{\vertname} }, \\
    \outedges{\vertname} &= \biglset
    { e\in\edgeset }{ \exists \vertname'\in\vertset \st e = \edgearg{\vertname}{\vertname'} },
\end{align}
\end{subequations}
while the parents and children of a vertex $\vertname\in\vertset$ are denoted
\begin{subequations}
\begin{align}
    \parnodes{\vertname} &= \biglset{ \vertname' \in \vertset }{ (\vertname',\vertname) \in \edgeset}, \\
    \childnodes{\vertname} &= \biglset{ \vertname' \in \vertset }{ (\vertname,\vertname') \in \edgeset }.
\end{align}
\end{subequations}
Given a graph $\graphname=\graphexpl$, we say that a vertex $\vertname\in\vertset$ is exogenous if the set $\parnodes{\vertname}$ is empty.
Otherwise, $\vertname$ is said to be endogenous.
We define the sets
\begin{equation}
    \exnodes = \{\vertname\in\vertset \; : \; \parnodes{\vertname}=\emptyset\} \quad \text{and} \quad  \nexnodes = \vertset \setminus \exnodes,
\end{equation}
of exogenous and endogenous vertices of a graph.
We represent directed graphs using diagrams where vertices are denoted as rectangles $\centertikz{
            \node[onode] (q) at (0,0) {$\vertname$};
            }$ and edges as directed arrows $\centertikz{
        \draw[cleg] (0,0) -- (0.6,0);
        }$.

\section{Functional causal models}
\label{sec:fCMs}
Here we review \emph{functional causal models}, which are also referred to as \textit{structural causal models}.  
While functional causal models are generally defined for variables taking values from a continuous set, here we restrict our definition to functional models over finite-cardinality variables as this is the relevant case for our results. For brevity, we will refer to this subclass of functional models as functional models instead of finite functional models, and denote them as $\fcm$.
\begin{definition}[Finite functional causal model]
    \label{def:functional_CM}
    Given a directed graph $\graphname=\graphexpl$, a finite functional model \textup{($\fcm_{\graphname}$)} is given by associating the following specifications to each vertex $\vertname\in\vertset$:
    \begin{myitem}
        \item A random variable $X_\vertname$ taking values $x_\vertname$ 
        from a non-empty finite set $\outcomemaparg{\vertname}$. We will use a notation where if $\vertset'\subseteq\vertset$ is a non-empty subset of vertices,
        \begin{equation}
            \outcomemaparg{\vertset'}=\prod_{\vertname\in\vertset'} \outcomemaparg{\vertname}
        \end{equation}
        where $\prod$ here denotes the Cartesian product.
        \item An error random variable $\errorrvarg\vertname$ taking values $u_\vertname$ 
        from a finite set $\errormaparg{\vertname}$, distributed as $\probex{\vertname}: \errormaparg{\vertname}\mapsto [0,1]$.
        \item A function 
            $\funcarg{\vertname}: \outcomemaparg{\parnodes{\vertname}}\times \errormaparg{\vertname} \mapsto \outcomemaparg{\vertname}$.
    \end{myitem} 
\end{definition}
One can think of functional models as assigning to each vertex $\vertname\in\vertset$ a value from the set $\outcomemaparg{\vertname}$. The value assignment on a vertex, $\vertname\in\vertset$, depends stochastically on the value assignments of its parents through the functional dependency associated to it, $\funcarg{\vertname}$. The stochastic character of such dependency is given by error variables associated to each vertex.
In other words, the value $\outcome_{\vertname}\in\outcomemaparg{\vertname}$ that we associate to the vertex $\vertname$, depends deterministically on the values $\outcome_{\parnodes{\vertname}} = \{\outcome_{\vertname'}\in\outcomemaparg{\vertname'}\}_{\vertname'\in\parnodes{\vertname}}$ and $u_\vertname\in\errormaparg{\vertname}$ through $\outcome_{\vertname} = \funcarg{\vertname}\big(\outcome_{\parnodes{\vertname}}, u_\vertname\big)$, by averaging on $u_\vertname\in\errormaparg{\vertname}$ we obtain the stochastic dependency on $\outcome_{\parnodes{\vertname}}$ only. 

Given a functional model on a graph, a well-defined probability distribution over the values of the vertices has to be defined.
In the special case of functional models on \textit{acyclic} graphs, the probability rule is given as follows.
Note that this definition is standard in the literature~\cite{Pearl_2009, Spirtes_2005}\footnote{This probability rule for \textit{acyclic} functional models is well-defined also in the infinite-cardinality or continuous variable case.}.

\begin{definition}[Probability distribution of an acyclic functional model] 
    \label{def: distribution_functional_cm}
    \hypertarget{probfacyc}{Consider} a functional model $\fcm_{\graphname}$ on an \textup{acyclic} graph $\graphname=\graphexpl$ and a global observed event $\outcome :=\{\outcome_{\vertname}\in\outcomemaparg{\vertname}\}_{\vertname\in\vertset}$. The probability $\probfacyc\left(\outcome\right)_{\graphname}\in [0,1]$ is defined as 

    \begin{equation}
        \probfacyc\left(\outcome\right)_{\graphname} = \sum_{u}\prod_{\vertname\in\vertset}\probex{\vertname}(u_\vertname) \delta_{\outcome_{\vertname}, \funcarg{\vertname}\left(\outcome_{\parnodes{\vertname}},u_{\vertname}\right)},
    \end{equation}
    where the sum $\sum_{u}$ runs over $u=\{u_\vertname\in\errormaparg{\vertname}\}_{\vertname\in\vertset}$.
\end{definition}
If some vertices are unobserved, the probability distribution over the remaining ones is obtained though marginalising the distribution in~\cref{def: distribution_functional_cm} over the unobserved variables.

Notice that the above probability rule can be equivalently expressed as
 \begin{equation}
 \label{eq: Markov1}       
   \probfacyc\left(\outcome\right)_{\graphname} = \prod_{\vertname\in\vertset} \prob^{(\vertname)}\left(\outcome_\vertname|\parnodes{\outcome_\vertname}\right)_{\graphname},
 \end{equation}  
where $\prob^{(\vertname)}\left(\outcome_\vertname|\parnodes{\outcome_\vertname}\right)_{\graphname}= \sum_{u_\vertname}\probex{\vertname}(u_\vertname) \delta_{\outcome_{\vertname}, \funcarg{\vertname}\left(\outcome_{\parnodes{\vertname}},u_{\vertname}\right)}$ is the conditional probability distribution of the variable $\outcomervarg{\vertname}$ conditioned on its parent variables $\parnodes{\outcomervarg{\vertname}}$ in $\fcm_{\graphname}$.

\paragraph{Remark.}\label{remark:errorrv}\Cref{def:functional_CM} associates to each vertex an error random variable to account for stochastic dependencies.
In many of the examples treated below, we consider deterministic dependencies, which amount to associating functions that do not depend on the error variable. In this case, we omit the error variables from the definition of the functional model and in the probability rule. 
In addition, for an exogenous vertex in $\graphname$, $\vertname\in\exnodes$, such that $\outcomemaparg{\vertname}=\errormaparg{\vertname}$ and $\funcarg{\vertname}(u_{\vertname})=u_\vertname$, the variable associated to $\vertname$, $\outcomervarg{\vertname}$, is distributed as the error variable $\errorrvarg{\vertname}$. 

Thus, in the special case where all endogenous vertices depend deterministically on their parents and all associations to exogenous vertices satisfy $\outcomemaparg{\vertname}=\errormaparg{\vertname}$ and $\funcarg{\vertname}(u_{\vertname})=u_\vertname$, the probability rule can be written as
\begin{equation}
        \probfacyc\left(\outcome\right)_{\graphname} = \prod_{\vertname\in\exnodes}\probex{\vertname}\left(\outcome_{\vertname}\right)  \prod_{w\in\nexnodes} \delta_{\outcome_{w}, \funcarg{w}\left(\outcome_{\parnodes{w}}\right)}.
    \end{equation}

\paragraph{Example of an acyclic functional model.} 

Consider the following directed acyclic graph:
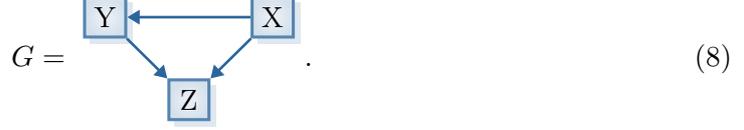
\begin{equation}
\label{eq:example acyclic fcm}
    \graphname=\centertikz{
        \node[onode] (Y) {Z};
        \node[onode] (X) [above left = \chanvspace and \chanvspace of Y] {Y};
        \node[onode] (Z) [above right = \chanvspace and \chanvspace of Y] {X};
        \draw[cleg] (Z) -- (X);
        \draw[cleg] (Z) -- (Y); 
        \draw[cleg] (X) -- (Y);
    }.
\end{equation}
We define a functional causal model for $\graphname$. Let us associate to the vertex $X$ an error $\errorrvarg{X}$ distributed as $\probex{X}$, and a function $\funcarg{X}:\errormaparg{X}\mapsto\outcomemaparg{X}$. Since the vertex $Y$ has parents $\parnodes{Y}=\{X\}$, it has associated an error $\errorrvarg{Y}$ distributed as $\probex{Y}$, and a function $\funcarg{Y}:\outcomemaparg{X}\times\errormaparg{Y}\mapsto\outcomemaparg{Y}$. While the vertex $Z$, which has $\parnodes{Z}=\{X,Y\}$, has associated an error $\errorrvarg{Z}$ distributed as $\probex{Z}$, and a function $\funcarg{Z}:\outcomemaparg{X}\times\outcomemaparg{Y}\times\errormaparg{Z}\mapsto\outcomemaparg{Z}$. 
Thus, the probability of a joint event $(x,y,z)\in \outcomemaparg{X}\times\outcomemaparg{Y}\times\outcomemaparg{Z}$, is evaluated using~\cref{def: distribution_functional_cm}:
\begin{equation}    \probfacyc(x,y,z)_{\graphname}=\sum_{u_X,u_Y,u_Z}\probex{X}(u_X) \probex{Y}(u_Y)\probex{Z}(u_Z)\delta_{x,\funcarg{X}(u_X)}\delta_{y,\funcarg{Y}(x,u_Y)}\delta_{z,\funcarg{Z}(x,y,u_Z)}.
\end{equation}

As described in the remark of~\cref{remark:errorrv}, let us consider the case where $\funcarg{Y}(x,u_Y)=\funcarg{Y}(x)$, $\funcarg{Z}(x,y,u_Z)=\funcarg{Z}(x,y)$, i.e., the variables $Y$ and $Z$ depend deterministically on their parents, and $\funcarg{X}(u_X)=u_X$. Then, the probability rule simplifies to
\begin{equation}
\label{eq:example acyclic fcm probability}
    \probfacyc(x,y,z)_{\graphname}=\probex{X}(x)\delta_{y,\funcarg{Y}(x)}\delta_{z,\funcarg{Z}(x,y)}.
\end{equation}
If we consider the variable $X$ unobserved, we evaluate the probability of a joint event $(y,z)\in\outcomemaparg{Y}\times\outcomemaparg{Z}$ over the remaining observed vertices by marginalising~\cref{eq:example acyclic fcm probability} as
\begin{equation}
     \probfacyc(y,z)_{\graphname}=\sum_{x\in\outcomemaparg{X}}\probfacyc(x,y,z)_{\graphname}=\sum_{x\in\outcomemaparg{X}}\probex{X}(x)\delta_{y,\funcarg{Y}(x)}\delta_{z,\funcarg{Z}(x,y)}.
\end{equation}

\paragraph{Cyclic functional models.}

If the graph is \textit{cyclic}, the literature considers the probability rule to be well-defined only for a subset of models~\cite{Forre_2017, Bongers_2021}. 
These correspond to functional models that are \textit{uniquely solvable}, i.e., for each value assignment of the error variables there always exists a unique solution (a unique value assignment for the variables associated to the vertices) satisfying all the functional dependencies. In this special case, the distribution is defined in the same way as~\cref{def: distribution_functional_cm}. We refer to~\cite{Forre_2017, Bongers_2021} and~\cref{sec:solvability} for more details. 

However, not all cyclic causal models are uniquely solvable, for instance consider the functional model on the cylic graph
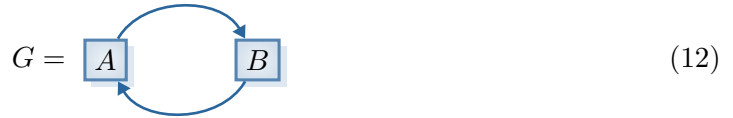
\begin{equation}
    \graphname= \centertikz{
        \node[onode] (A) at (0,0) {$A$};
        \node[onode] (B) at (2,0) {$B$};
        \draw[cleg] (A) to [in=120, out=60] (B);
        \draw[cleg] (B) to[in=300, out=240] (A);
    }
\end{equation}
with the following associations:
we associate to $A$ and $B$ the same finite set, i.e., $\outcomemap=\outcomemaparg{A}=\outcomemaparg{B}$.
The vertex $A$ is equipped with the function $\funcarg{A} : \outcomemaparg{B} \mapsto \outcomemaparg{A}$ such that $\funcarg{A}(b) = b$ for all $b \in \outcomemap$.
The vertex $B$ is equipped with the function $\funcarg{B} : \outcomemaparg{A} \mapsto \outcomemaparg{B}$ such that $\funcarg{B}(a) = a$ for all $a\in\outcomemap$. These functional dependencies are deterministic, i.e., do not depend on error variables, hence we omit them (see the remark above), and fix the values of $A$ and $B$ to be equal.

This functional model is not uniquely solvable (unless the set $\outcomemap$ has trivial cardinality)\footnote{In this case, the error variables do not play any role in the functional dependencies. Thus, the model is uniquely solvable if and only if the deterministic dependencies admit a unique solution}, since $a=b$ is a solution of the model for all $a=b\in\outcomemap$. The literature considers any distribution $\mathcal{P}$ on $A$ and $B$ such that $\mathcal{P}(a=b)=1$, i.e., the two variables are equal with certainty. This corresponds to a distribution of the form $\mathcal{P}(a,b)=p(a)\delta_{a,b}$, for an arbitrary distribution $\{p(a)\}_{a\in\outcomemap}$. Thus, in the non-uniquely solvable case, the distribution is not uniquely defined.

Our first goal for the next sections is to solve this ambiguity and give a method to define a unique probability distribution over the vertices of discrete functional models on any cyclic graph. Based on this, we will develop our new graph-separation property, $p$-separation.

\section{Mapping cyclic to acyclic functional models with post-selection}
\label{sec:fCM_to_afCM}
In order to define probabilities for functional models on arbitrary cyclic graphs, we construct a mapping from cyclic models to acyclic ones with post-selection. The mapping is inspired by a quantum protocol which is well-known in quantum mechanics and information theory, namely the \textit{quantum teleportation protocol} \cite{Bennett1993}. This protocol allows a party to transmit an arbitrary quantum state to another party using only quantum correlations and classical communication. If we introduce post-selection, classical communication is no longer necessary, and in this case, we call the protocol \textit{post-selected teleportation}~\cite{Lloyd_2011, Lloyd_2011_2}.

Here, we construct a classical analogue of the (post-selected) teleportation protocol and define the desired mapping. The mapping and probability rule follow the same ideas of~\cite{Quantum_paper}, where we defined the probability rule for \textit{quantum} cyclic causal models exploiting post-selected teleportation. 

\subsection{Classical post-selected teleportation} 
\label{sec: classical ps tele}
Consider two binary random variables, $A$ and $C$, taking values from the same alphabet $\{0,1\}=\outcomemaparg{A}=\outcomemaparg{C}$, and let $\probex{A}$ be an arbitrary distribution on $A$. Our goal is to construct a protocol which, for any distribution $\probex{A}$, allows us to obtain the same distribution on $C$ as though there was an identity channel between the variables $A$ and $C$ that ``teleports'' $A$'s distribution to $C$.  
Let us consider the following: take $C$ to be uniformly distributed and measure whether $A$ and $C$ are equal, if not apply a bit flip correction to $C$. For clarity let us denote with $C'$ the variable $C$ after the correction (which is trivial if $A$ and $C$ are equal) is applied. We can achieve this as follows. First we represent this protocol using the graph\footnote{This graph mimics the structure of the quantum teleportation protocol \cite{Bennett1993}.}
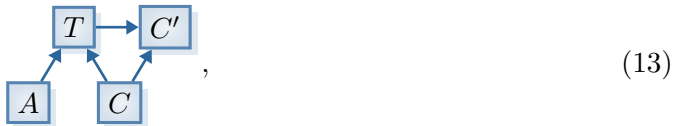
\begin{align}
    \centertikz{
    \begin{scope}[xscale=1.2]
        \node[onode] (A) at (0,0) {$A$};
        \node[onode] (pre) at (1,0) {$C$};
        \node[onode] (post) at (0.5,1) {$\postvertname$};
        \draw[cleg] (pre) -- (post);
        \draw[cleg] (A) -- (post);
        \node[onode] (op) at (1.5,1) {$C'$};
        \draw[cleg] (post) -- (op);
        \draw[cleg] (pre) -- (op);
    \end{scope}
    },
\end{align}
and associate $\probex{A}$ to $A$, $\probex{C}(c)=1/2$ to $C$, the function $\funcarg{\postvertname}(a,c)=\delta_{a,c}$ to $T$ and the function $\funcarg{C'}(c,t) = c\oplus t\oplus 1$ to $C'$, for all $a,c\in\outcomemap$ and $t\in\{0,1\}$. In this example, we consider the error variables as described in the remark of~\cref{remark:errorrv}.
The distribution of $C'$ and $T$ can be evaluated using~\cref{def: distribution_functional_cm} and the fact that $\delta_{t,\delta_{a,c}} = \delta_{t,a\oplus c\oplus 1}$:
\begin{equation}
\begin{split}
    \probfacyc(c',t) &= \sum_{a,c} \probex{A}(a)\probex{C}(c)\delta_{t,a\oplus c\oplus 1} \delta_{c',c\oplus t\oplus 1}\\
    &= \frac{1}{2}\sum_{a} \probex{A}(a)\delta_{a,c'} \sum_{c}\delta_{c'\oplus t,c}\\
    &=\frac{1}{2}\sum_{a} \probex{A}(a)\delta_{a,c'} =\frac{1}{2}\probex{A}(c').
\end{split} 
\end{equation}
Hence, independently of $t$ the distribution on $C'$ equals the initial and unknown distribution on $A$ up to a constant factor. In particular, if we only consider the instances where $A$ and $C$ take equal values, i.e., $t=1$, no correction on $C$ is needed and the function from $C$ to $C'$ acts as an identity channel, $c'=\funcarg{C'}(c,t=1) = c$ for all $c$. Thus, if we post-select on $t=1$, the vertex $C'$ is not necessary and we can represent the post-selected protocol as
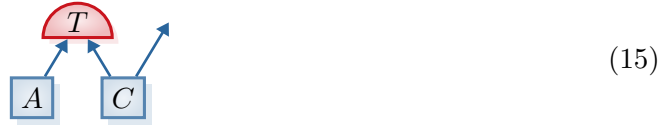
\begin{align}
    \centertikz{
    \begin{scope}[xscale=1.2]
        \node[onode] (A) at (0,0) {$A$};
        \node[onode] (pre) at (1,0) {$C$};
        \node[psnode] (post) at (0.5,1) {$T$};
        \draw[cleg] (pre) -- (post);
        \draw[cleg] (A) -- (post);
        \draw[cleg] (pre) -- (1.5,1);
    \end{scope}
    }
\end{align}
where we denoted with $\centertikz{
            \node[psnode] (q) at (0,0) {};
            }$ the post-selected vertex. 
In this case, we have $\probfacyc(c|t=1)=\probex{A}(c)$ as desired.  Since this holds for any distribution on $\probex{A}$, we can identify this structure with an identity channel:
            \begin{align}
\label{eq:diagram_cps}
    \centertikz
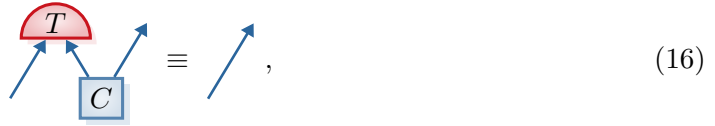
{
    \begin{scope}[xscale=1.2]
        \node[onode] (pre) at (1,0) {$C$};
        \node[psnode] (post) at (0.5,1) {$T$};
        \draw[cleg] (pre) -- (post);
        \draw[cleg] (0,0) -- (post);
        \draw[cleg] (pre) -- (1.5,1);
    \end{scope}
    }
    \equiv
    \centertikz{
    \begin{scope}[xscale=1.2]
        \draw[cleg] (0,0) -- (0.5,1);
    \end{scope}
    },
\end{align}
where the equivalence has to be understood at the level of conditional probabilities, conditioned on the event $t=1$.
Motivated by such example, we define in more generality a classical post-selected teleportation protocol.

\begin{definition}[Classical post-selected teleportation protocol] 
\label{def:classical ps tele}
Let $\outcomemaparg{A}=\outcomemaparg{C}$  and $\outcomemaparg{B}$ be finite sets. A classical post-selected teleportation protocol consists of a triple $(\ctelefunc, \cteleprob_B,\cteleprob_C)$ where $\ctelefunc:\outcomemaparg{A}\times\outcomemaparg{B} \times\outcomemaparg{C} \mapsto \{0,1\}$ is a function and $\cteleprob_i:\outcomemaparg{i}\mapsto [0,1]$ for $i=B,C$ are probability distributions, such that for all probability distribution $\probex{A}:\outcomemaparg{A}\mapsto[0,1]$, it holds that
\begin{equation}
\sum_{\substack{a\in\outcomemaparg{A}\\b\in\outcomemaparg{B}}} \probex{A}(a)\delta_{\ctelefunc(a,b,c), 1} \cteleprob_B(b)\cteleprob_C(c)=\teleprob \probex{A}(c)
\end{equation}
for all $c\in\outcomemaparg{C}$,
where $\teleprob\in(0,1]$ is the success probability of the post-selected teleportation protocol.
\end{definition}
In contrast to the example presented above, general classical post-selected teleportation protocols can be implemented using a non-deterministic function on $A$ and $C$ to post-select (where $B$ plays the role of an error variable for the post-selection vertex $T$). The following results are proven in~\cref{app:proofs map}. 

\begin{restatable}{lemma}{indepprob}
    The success probability of a classical teleportation protocol, $\teleprob$, of~\cref{def:classical ps tele} is independent of the probability distribution $\probex{A}$ being teleported.
\end{restatable}

\begin{restatable}{lemma}{copyprop}
\label{lem: copy property aug}
    Consider a classical teleportation protocol defined by the pair $(\ctelefunc,\cteleprob_B,\cteleprob_C)$. Then, it holds:
    \begin{equation}
    \sum_{\substack{b\in\outcomemaparg{B}}} \delta_{\ctelefunc(a,b,c), 1} \cteleprob_B(b)\cteleprob_C(c)=\teleprob\delta_{a,c}.
    \end{equation}
\end{restatable}
The previous lemma shows that the systems $A$ and $C$ act as copies of each other, given the post-selection.
In a classical theory, variables can be copied and broadcast to multiple parties\footnote{This is not the case for quantum mechanical systems, thus there is no analogue of broadcasting quantum teleportation protocols in~\cite{Quantum_paper}.}. For this reason, one can broadcast the outcome of a classical post-selected teleportation protocol to $n$ copies of $C$ for any finite $n$. 
Diagrammatically, we can represent this as follows
\begin{align}
\label{eq:diagram_cps_broad}
    \centertikz{
    \begin{scope}[xscale=1.2]
        \node[onode] (A) at (0,0) {$A$};
        \node[onode] (pre) at (1,0) {$C$};
        \node[psnode] (post) at (0.5,1) {$T$};
        \draw[cleg] (pre) -- (post);
        \draw[cleg] (A) -- (post);
        \draw[cleg] (pre) -- (1,1);
        \node[rotate=-15] at (1.3, 0.9) {\small$\dots$};
        \draw[cleg] (pre) -- (1.6,0.8);
    \end{scope}
    }
    \equiv
    \centertikz{
    \begin{scope}[xscale=1.2]
        \node[onode] (A) at (0,0) {$A$};
        \draw[cleg] (A) -- (0,1);
        \node[rotate=-20] at (0.3, 0.9) {\small$\dots$};
        \draw[cleg] (A) -- (0.6,0.8);
    \end{scope}
    },
\end{align}
where the equivalence has to be understood at the level of probabilities conditioned on the outcome of $T$ being $t=1$.
Broadcasting a classical post-selected teleportation protocol to $n$ copies of $C$ is naturally connected to a protocol consisting of $n$ copies of the same classical post-selected teleportation protocol. 
Indeed, given $n$ copies of the same protocol we have the following result
\begingroup
      \crefname{equation}{\textup{eq.}}{\textup{Eq.}}
\begin{equation}
\label{eq:from broadcast to copies}
    \centertikz{
        \node[onode] (A) at (0.5,0) {$A$};
    \begin{scope}[xscale=1.2]
        \node[rotate=-20] at (0.55, 1.2) {\small$\dots$};
    \end{scope}
    \begin{scope}[xscale=1.2,shift={(0,2)}, rotate=30]
        \node[onode] (pre) at (1,0) {$C_1$};
        \node[psnode,rotate=30] (post) at (0.5,1) {$\postvertname_1$};
        \draw[cleg] (pre) -- (post);
        \draw[cleg] (pre) -- (1.5,1);
    \end{scope}
    \begin{scope}[xscale=1.2]
         \draw[cleg] (A) -- (post);
    \end{scope}
    \begin{scope}[xscale=1.2,shift={(0.95,0.5)}, rotate=0]
        \node[onode] (pre) at (1,0) {$C_n$};
        \node[psnode] (post) at (0.5,1) {$\postvertname_n$};
        \draw[cleg] (pre) -- (post.320);
        \draw[cleg] (pre) -- (1.5,1);
    \end{scope}
    \begin{scope}[xscale=1.2]
         \draw[cleg] (A) -- (post);
    \end{scope}
    } \stackrel{\textup{\cref{eq:diagram_cps}}}{\equiv}  \centertikz{
    \begin{scope}[xscale=1.2]
        \node[onode] (A) at (0,0) {$A$};
        \draw[cleg] (A) -- (0,1);
        \node[rotate=-20] at (0.3, 0.9) {\small$\dots$};
        \draw[cleg] (A) -- (0.6,0.8);
    \end{scope}    }\stackrel{\textup{\cref{eq:diagram_cps_broad}}}{\equiv} \centertikz{
    \begin{scope}[xscale=1.2]
        \node[onode] (A) at (0,0) {$A$};
        \node[onode] (pre) at (1,0) {$C$};
        \node[psnode] (post) at (0.5,1) {$T$};
        \draw[cleg] (pre) -- (post);
        \draw[cleg] (A) -- (post);
        \draw[cleg] (pre) -- (1,1);
        \node[rotate=-15] at (1.3, 0.9) {\small$\dots$};
        \draw[cleg] (pre) -- (1.6,0.8);
    \end{scope}
    }.
\end{equation}
\endgroup
The equivalence has to be understood at the level of conditional probabilities: explicitly, it holds that
\begin{align} 
    \probfacyc(\{c_k\}_{k=1,\dots,n}|\{t_k=1\}_{k=1,\dots,n}) 
    &= \probex{A}(a = c_1) \delta_{c_2,c_1} \dots \delta_{c_n,c_1} \nonumber\\
    &= \probfacyc(c = c_1| t=1) \delta_{c_2,c_1} \dots \delta_{c_n,c_1}  
\end{align}
where $t_k$ is the variable associated to the $k$-th post-selection vertex $T_k$ on the leftmost diagram of \cref{eq:from broadcast to copies} and $t$ is the variable associated to the post-selection vertex $T$ on the rightmost diagram of \cref{eq:from broadcast to copies}. This holds for an arbitrary distribution $\probex{A}$ on the vertex $A$, taken to be the same in all three diagrams.
Notice that this implies, for all $i=1,\dots, n$,
\begin{equation}
    \probfacyc(c_i=c|\{t_k=1\}_{k=1,\dots,n}) = \probex{A}(a=c) = \probfacyc(c|t=1).
\end{equation}

In principle, there could be many triples $(\ctelefunc,\cteleprob_B,\cteleprob_C)$ which implement a classical post-selected teleportation protocol. Therefore, we define the following as canonical choice. In~\cref{app:proofs map}, we show that it is a valid classical post-selected teleportation protocol with $\teleprob = 1/|\outcomemaparg{A}|$.

\begin{definition}[Uniform prior post-selected teleportation protocol]
\label{def:uniform prior tele}
The canonical choice of classical post-selected teleportation protocol consists of choosing
\begin{equation}
    \ctelefunc(a,b, c)=\ctelefunc(a, c) = \delta_{a, c} \quad \text{and} \quad  \cteleprob_C(c) =\frac{1} {|\outcomemaparg{A}|},
\end{equation}
for all $a,c \in \outcomemaparg{A}$. Since $\ctelefunc$ is independent of $B$, $\outcomemaparg{B}$ and $\cteleprob_B$ are irrelevant. We refer to this choice as \textup{uniform prior} post-selected teleportation protocol.
\end{definition}

\subsection{Family of acyclic functional models from a cyclic functional model}
\label{sec:fCM_to_afCM_map}
Using classical post-selected teleportation protocols, we construct a family of acyclic functional models from a given cyclic functional model.
First, we define the family of acyclic graphs underlying these functional models. 
The proofs of the results in this section are provided in~\cref{app:proofs map}. 

\begin{definition}[Family of acyclic classical teleportation graphs $\graphfamilysn{\graphname}$]
    \label{def:sn_graph_family}
    Given a directed graph $\graphname=\graphexpl$,
    we define an associated family $\graphfamilysn{\graphname}$ of acyclic graphs, where each element $\graphname\sn\in\graphfamilysn{\graphname}$ is obtained from the graph $\graphname$ as follows.
    \begin{myitem}
        \item Choose any subset of vertices $\splitvert{\graphname\sn}\subseteq\vertset$, such that the subgraph $\graphname'=(\vertset',\edgeset')$ of $\graphname$ with $\vertset'=\vertset$ and $\edgeset'=\edgeset\setminus\left(\bigcup_{\vertname\in\splitvert{\graphname\sn}}\outedges{\vertname}\right)$ is acyclic. We refer to $\splitvert{\graphname\sn}$ as the split vertices of $\graphname\sn$. 
        \item The vertices $\vertset\sn$ of $\graphname\sn$ consist of the vertices $\vertset$ of the original graph $\graphname$ together with new vertices $\prevertname_\vertname, \postvertname_\vertname$ for each split vertex $\vertname \in \splitvert{\graphname\sn}$, i.e.,
        \begin{align}
            \vertset\sn = \vertset 
            \cup \{ \prevertname_\vertname \}_{\vertname\in\splitvert{\graphname\sn}} 
            \cup \{ \postvertname_\vertname \}_{\vertname\in\splitvert{\graphname\sn}}.
        \end{align}
        \item The edges $\edgeset\sn$ of $\graphname\sn$ consist of the edges $\edgeset'$ of the subgraph $\graphname'$ together with the following new edges:
        \begin{align}
            \edgeset\sn = \edgeset' 
            \cup \{\edgearg{\vertname}{\postvertname_\vertname}\}_{\vertname\in\splitvert{\graphname\sn}}
            &\cup \{\edgearg{\prevertname_\vertname}{\postvertname_\vertname}\}_{\vertname\in\splitvert{\graphname\sn}} \nonumber\\
            &\cup \{\edgearg{\prevertname_\vertname}{\vertname'}\}_{\vertname\in\splitvert{\graphname\sn}, \vertname' \in \childnodes{\vertname}_{\graphname}},
        \end{align}
        where $\childnodes{\vertname}_{\graphname}$ refers to the children vertices of $\vertname$ in the graph $\graphname$.     
    \end{myitem}
    For each $\vertname\in\splitvert{\graphname\sn}$, we refer to the vertices $\prevertname_\vertname$ and $\postvertname_\vertname$ as pre- and post-selection vertices respectively and depict them with distinct vertex styles  $\centertikz{\node[psnode] {$\postvertname_\vertname$};}$ and $\centertikz{\node[prenode] {$\prevertname_\vertname$};}$, as these will play a special role in our framework.
    This construction makes $\graphname\sn$ identical to $\graphname$ up to replacing all vertices $\vertname\in\splitvert{\graphname\sn}$ with the following structure
    \begin{equation}
\label{eq:split_node}
    \centertikz{
        \node[onode] (v) {$\vertname$};
        \draw[cleg] ($(v)-(0.5,0.8)$) -- (v.240);
        \draw[cleg] ($(v)-(-0.5,0.8)$) -- (v.300);
        \draw[cleg] (v.60) -- ($(v)+(0.5,0.8)$);
        \draw[cleg] (v.120) -- ($(v)+(-0.5,0.8)$);
        \node [above = 0.4*\chanvspace of v] {\indexstyle{\dots}};
        \node [below = 0.4*\chanvspace of v] {\indexstyle{\dots}};
        \draw [thick, decoration={brace},decorate]($(v)-(-0.5,1)$) -- ($(v)-(0.5,1)$) node [pos=0.5,anchor=north] {\indexstyle{\parnodes{\vertname}}};
        \draw [thick, decoration={brace},decorate]($(v)+(-0.5,1)$) -- ($(v)+(0.5,1)$) node [pos=0.5,anchor=south] {\indexstyle{\childnodes{\vertname}}}; 
    } \mapsto
    \centertikz{
        \node[onode] (vin) at (0,0) {$\vertname$};
        \draw[cleg] ($(vin)-(0.5,0.8)$) -- (vin.240);
        \draw[cleg] ($(vin)-(-0.5,0.8)$) -- (vin.300);
         \node[prenode] (vout) at (2.5,0) {$\prevertname_\vertname$};
        \draw[cleg] (vout.40) -- ($(vout)+(0.7,0.8)$);
        \draw[cleg] (vout.100) -- ($(vout)+(-0.3,0.8)$);
        \node [above = 0.4*\chanvspace of vout.60] {\indexstyle{\dots}};
        \node [below = 0.4*\chanvspace of vin] {\indexstyle{\dots}};
        \draw [thick, decoration={brace},decorate]($(vin)-(-0.5,1)$) -- ($(vin)-(0.5,1)$) node [pos=0.5,anchor=north] {\indexstyle{\parnodes{\vertname}_{\graphname}}};
        \draw [thick, decoration={brace},decorate]($(vout)+(-0.3,1)$) -- ($(vout)+(0.7,1)$) node [pos=0.5,anchor=south] {\indexstyle{\childnodes{\vertname}_{\graphname}}}; 
        \node[psnode] (ps) at (1,2) {$\postvertname_\vertname$};
        \draw[cleg] (vout.140) -- (ps.300);
        \draw[cleg] (vin.90) -- (ps.240);
    }.
\end{equation}
Each element of the family $\graphname\sn\in\graphfamilysn{\graphname}$ is called a classical \textup{teleportation graph}. The set of all pre- and post-selection vertices in $\graphname\sn$ are denoted as $\prevertset = \{\prevertname_\vertname\}_{\vertname\in\splitvert{\graphname\sn}}$ and $\psvertset = \{\postvertname_\vertname\}_{\vertname\in\splitvert{\graphname\sn}}$.
\end{definition}

The following lemma, proven in \cref{app:proofs map}, shows that classical teleportation graphs constructed as above are indeed acyclic.

\begin{restatable}[Acyclicity of classical teleportation graphs]{lemma}{acyclicity}
\label{lemma: acyclicity_cl_telegraphs}
Given a directed graph $\graphname$, each classical teleportation graph $\graphname\sn\in\graphfamilysn{\graphname}$ is acyclic.
\end{restatable}

The next step of the mapping of cyclic functional models to acyclic functional models with post-selection involves defining functional models on classical teleportation graphs. 

\begin{definition}[Family of functional models on acyclic classical teleportation graphs]
\label{def: family functional cm}
Given a functional model on a directed graph $\graphname=\graphexpl$, $\fcm_{\graphname}$, we can construct a family of functional models by defining a functional causal model, $\fcm_{\graphname\sn}$, on each acyclic classical teleportation graph $\graphname\sn\in\graphfamilysn{\graphname}$(\cref{def:sn_graph_family}), as follows:
\begin{myitem}
    \item Using the same notation as in~\cref{def:sn_graph_family}, for each split vertex $\vertname\in\splitvert{\graphname\sn}$ of $\graphname\sn$ the functional model $\fcm_{\graphname\sn}$, has the following specifications:
    \begin{enumerate}
        \item To the post-selection vertex $\postvertname_\vertname$, we associate a random variable taking values from $\outcomemaparg{\postvertname_\vertname} = \{0,1\}$, an error variable $\errorrvarg{\postvertname_\vertname}$ with distribution $\probex{\postvertname_\vertname}$ and a function $\funcarg{\postvertname_\vertname}$.
        \item To the pre-selection vertex $\prevertname_\vertname$, we associate a random variable taking values from $\outcomemaparg{\prevertname_\vertname}=\outcomemaparg{\vertname}$, an error variable $\errorrvarg{\prevertname_\vertname}$ with distribution $\probex{\prevertname_\vertname}$ and a function $\funcarg{\prevertname_{\vertname}}(u_{\prevertname_\vertname})=u_{\prevertname_\vertname}$\footnote{As discussed in the remark of~\cref{remark:errorrv}, this is equivalent to directly considering the variable $\prevertname_\vertname$ as having distribution $\probex{\prevertname_\vertname}$.}.

        \item The probability distributions $\probex{\prevertname_\vertname}$, $\probex{\postvertname_\vertname}$ together with the function $\funcarg{\postvertname_\vertname}$ form a classical teleportation protocol (\cref{def:classical ps tele}).
        \end{enumerate}
        \item For each vertex present in both $\graphname$ and $\graphname\sn$, i.e., each vertex except the pre- and post-selection vertices, the assigned finite sets, error variables and functional dependencies are the same in $\fcm_{\graphname\sn}$ and $\fcm_{\graphname}$, where for all $\vertname\in\splitvert{\graphname\sn}$ the corresponding preselection vertex $\prevertname_\vertname$ plays the role of the original vertex $\vertname$ in the inputs of the relevant functions, i.e., $\funcarg{\vertname'}$ for $\vertname'\in\childnodes{\vertname}_{\graphname}$.
\end{myitem}
We will refer to functional models in this family as \textup{classical teleportation functional models}.
\end{definition}

Any classical teleportation model $\fcm_{\graphname\sn}$ constructed according to \cref{def: family functional cm} is defined over an acyclic graph, thus, we can compute its probabilities using the acyclic probability rule of \cref{def: distribution_functional_cm}.
Denoting $\psvertset$ and $\prevertset$ as the set of all pre- and post-selection vertices of $\graphname\sn$, its total set of vertices is $\vertset\cup\prevertset\cup\psvertset$. 
Therefore, by applying the acyclic probability rule of \cref{def: distribution_functional_cm} to $\fcm_{\graphname\sn}$, we can compute
\begin{equation}
    \probfacyc\left(\outcome,r, \{\postoutcome_\postvertname = 1\}_{\postvertname\in\psvertset}\right)_{\graphname\sn},
\end{equation}
 where $\outcome=\{\outcome_\vertname\}_{\vertname\in\vertset}$ and $r=\{r_{\prevertname}\}_{\prevertname\in\prevertset}$. 
 From this probability we can define the probability of successful post-selection.
    
\begin{definition}[Post-selection success probability.]
\label{def: c_success_probability}
    Consider a functional model on a directed graph $\graphname$ and a functional model of the family in~\cref{def: family functional cm} with associated graph $\graphname\sn$.  We define the probability of successful post-selection of the classical teleportation functional model as
    \begin{equation}
        \csuccessprob := 
    \probfacyc\left(\{\postoutcome_\postvertname = 1\}_{\postvertname\in\psvertset}\right)_{\graphname\sn}
    =\sum_{\outcome,r}
    \probfacyc\left(\outcome,r,\{\postoutcome_\postvertname = 1\}_{\postvertname\in\psvertset}\right)_{\graphname\sn},
    \end{equation}
    where the summation is performed over $\outcome=\{\outcome_\vertname\in\outcomemaparg{\vertname}\}_{\vertname\in\vertset}$ and 
    $r = \{r_{\prevertname} \in \outcomemaparg{\prevertname}\}_{\prevertname \in \prevertset}$.
\end{definition}

Since the preselection vertices act as copies of the split vertices, we marginalize over them to obtain
\begin{equation}
    \probfacyc\left(\outcome, \{\postoutcome_\postvertname = 1\}_{\postvertname\in\psvertset}\right)_{\graphname\sn} = \sum_r\probfacyc\left(\outcome,r, \{\postoutcome_\postvertname = 1\}_{\postvertname\in\psvertset}\right)_{\graphname\sn}
\end{equation}
and focus on the conditional probability 

\begin{equation}
    \probfacyc\left(\outcome|\{\postoutcome_\postvertname = 1\}_{\postvertname\in\psvertset}\right)_{\graphname\sn} = \frac{\probfacyc\left(\outcome,\{\postoutcome_\postvertname = 1\}_{\postvertname\in\psvertset}\right)_{\graphname\sn}}{\csuccessprob}
\end{equation}
which will be the probability of interest in our framework.
The following proposition shows that this is independent of the choice of teleportation graph $\graphname\sn\in\graphfamilysn{\graphname}$.

\begin{restatable}[Equivalent probabilities from different classical teleportation graphs]{prop}{differentgraphsequiv}
        \label{lemma: acyclic_prob_same_func}
  Let $\fcm_{\graphname}$ be a functional model on a directed graph $\graphname$.
  Consider any $\graphname_1,\graphname_2 \in \graphfamilysn{\graphname}$, and for $i\in\{1,2\}$, let $\fcm_{\graphname_i}$ be a causal model on the teleportation graph $\graphname_i$ constructed according to \cref{def: family functional cm}. Denoting the set of all post-selection vertices of $\graphname_1,\graphname_2$ as $\psvertset^1$ and $\psvertset^2$ respectively, we have
  \begin{align}
      \probfacyc\left(\outcome \middle| \{\postoutcome_\postvertname = 1\}_{\postvertname\in\psvertset^1}\right)_{\graphname_1} =      \probfacyc\left(\outcome \middle| \{\postoutcome_\postvertname = 1\}_{\postvertname\in\psvertset^2}\right)_{\graphname_2},
  \end{align} 
  for all joint observed events $\outcome = \{\outcome_\vertname \in \outcomemaparg{\vertname}\}_{\vertname\in\vertset}$.
\end{restatable}
The above proposition is proven in~\cref{app:proofs map}.

\subsection{General probability rule for cyclic functional causal models}
\label{sec:probability rule}

In~\cref{sec:fCM_to_afCM_map}, we have constructed a family of acyclic functional models associated with a given, potentially cyclic, functional model~(\cref{def: family functional cm}). 
We have proven that the acyclic distribution conditioned on successful post-selection is independent of which functional model we choose in this family.
Here, we use these results to define a probability rule for the cyclic functional model underlying the family.
The proofs of the results presented in this section can be found in~\cref{app:proofs map}.

\begin{definition}[Probability of a cyclic functional models]
    \label{def: probability distribution cyclic fcm}
    \hypertarget{probf}{Consider} a functional model on a graph $\graphname$. Let $\graphname\sn\in\graphfamilysn{\graphname}$ be a classical teleportation graph with associated functional model $\fcm_{\graphname\sn}$ derived from $\fcm_{\graphname}$ according to~\cref{def: family functional cm} and let $\csuccessprob$ be the probability of successful post-selection as in~\cref{def: c_success_probability}. If $\csuccessprob > 0$, the probability of the joint observed event $x =\{\outcome_{\vertname}\in\outcomemaparg{\vertname}\}_{\vertname\in\vertset}$ in $\fcm_{\graphname}$ is defined as
    \begin{equation}
        \probf(\outcome)_{\graphname}:=  \probfacyc\left(\outcome|\{\postoutcome_\postvertname = 1\}_{\postvertname\in\psvertset}\right)_{\graphname\sn} = \frac{\probfacyc\left(\outcome,\{\postoutcome_\postvertname = 1\}_{\postvertname\in\psvertset}\right)_{\graphname\sn}}{\csuccessprob}.
    \end{equation}
If $\csuccessprob = 0$, we say that the model is inconsistent and the probabilities $\probf(\outcome)_{\graphname}$ are undefined.
\end{definition}

In the above definition of $  \probf$, there are two choices involved. 
One first needs to choose a teleportation graph $\graphname\sn\in\graphfamilysn{\graphname}$, and secondly, once $\graphname\sn$ is fixed, there is still freedom in choosing a specific implementation of classical post-selected teleportation protocol for each pair of pre- and post-selection vertices in $\graphname\sn$. We have already shown in \cref{lemma: acyclic_prob_same_func} that the definition above is independent of the first choice. In the following, we show that it is also independent of the second choice. 

\begin{restatable}[Alternative formulation of the probability rule]{prop}{selfcyclecl} 
\label{prop:connecting channels}
Consider a functional model $\fcm_\graphname$ on a directed graph $\graphname=\graphexpl$. If the model is not inconsistent, an alternative expression for the probability rule in~\cref{def: distribution_functional_cm} is given by
\begin{equation}
    \probf(\outcome)_{\graphname} = \frac{
    \sum_{u}\prod_{\vertname\in\vertset}\probex{\vertname}(u_\vertname) \delta_{\outcome_{\vertname}, \funcarg{\vertname}\left(\outcome_{\parnodes{\vertname}},u_{\vertname}\right)}
    }{
    \sum_{y}\sum_{u}\prod_{\vertname\in\vertset}\probex{\vertname}(u_\vertname) \delta_{y_{\vertname}, \funcarg{\vertname}\left(y_{\parnodes{\vertname}},u_{\vertname}\right)}
    }
\end{equation}
where we have defined the global event $x =\{\outcome_{\vertname}\in\outcomemaparg{\vertname}\}_{\vertname\in\vertset}$, the sum $\sum_{y}$ runs over all $y = \{y_\vertname\in\outcomemaparg\vertname\}_{\vertname\in\vertset}$ and $\sum_{u}$ over $u=\{u_\vertname\in\errormaparg{\vertname}\}_{\vertname\in\vertset}$.
\end{restatable}

The proof of the above proposition is given in~\cref{app:proofs map}. Observing that by definition the success probability of a post-selected teleportation protocol is strictly greater than zero, the following corollary immediately follows from \cref{prop:connecting channels}.

\begin{restatable}[Equivalent probabilities from different implementations]{corollary}{indepimplementation} 
\label{corollary:probs indep of tele implementation classic}
Consider a functional model $\fcm_{\graphname}$ on a directed graph $\graphname$ and any classical teleportation model in the family of classical teleportation functional models derived from $\fcm_{\graphname}$ according to \cref{def: family functional cm}. 
  It holds that whether or not the probabilities $\probf(\outcome)_{\graphname}$ of $\fcm_{\graphname}$ (given by \cref{def: probability distribution cyclic fcm}) are defined and the values they take (if they are defined) do not depend on the choices of the post-selected teleportation protocols (\cref{def:classical ps tele}) one makes in the construction of $\fcm_{\graphname\sn}$ from $\fcm_{\graphname}$ (\cref{def: family functional cm}).
\end{restatable}
\Cref{prop:connecting channels} highlights the way our method picks a unique probability rule even for non-uniquely solvable models. For instance, consider the model introduced in~\cref{sec:fCMs} on the graph
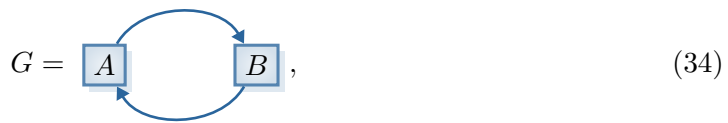
\begin{equation}
    \graphname= \centertikz{
        \node[onode] (A) at (0,0) {$A$};
        \node[onode] (B) at (2,0) {$B$};
        \draw[cleg] (A) to [in=120, out=60] (B);
        \draw[cleg] (B) to[in=300, out=240] (A);
    },
\end{equation}
where $\outcomemap=\outcomemaparg{A}=\outcomemaparg{B}$, $\funcarg{A}(b)=b$ and $\funcarg{B}(a)=a$ for all $a,b\in\outcomemap$. These functional dependencies are deterministic, thus it is not necessary to specify error variables (see the remark of~\cref{remark:errorrv} for more details). As we showed in~\cref{sec:fCMs}, this functional model is not uniquely solvable (unless the set $\outcomemap$ has trivial cardinality) and in literature any distribution $\mathcal{P}$ such that $\mathcal{P}(A=B)=1$ is considered valid. The probability rule~\cref{def: probability distribution cyclic fcm} evaluated using~\cref{prop:connecting channels} gives
\begin{equation}
    \probf(a,b)_{\graphname} = \frac{
      \delta_{a,b}\delta_{b,a}
    }{
    \sum_{a',b'}
      \delta_{a',b'}\delta_{b',a'}
    } = \frac{
      \delta_{a,b}
    }{
    |\outcomemap|
    }.
\end{equation}
Notice that the probability rule of~\cref{def: probability distribution cyclic fcm} associates a uniform distribution to all allowed solutions of the functional model in this example. This is not a consequence of using as classical teleportation protocol the uniform prior one~(\cref{def:uniform prior tele}), as we have shown that the probability rule does not depend on which classical teleportation protocol one chooses (\cref{corollary:probs indep of tele implementation classic}). In principle, other classical teleportation protocols which are associated with a non-uniform prior on the pre-selection vertices could be considered, which would yield the same probability.

\subsection{Examples of cyclic functional causal models}
\paragraph{Detailed example illustrating our framework and probability rule.}
\label{example: probability rule detail}
Here, we illustrate through an example how to perform the mapping of a cyclic functional model to an acyclic one with post-selection, and evaluate the corresponding probability distribution in our framework. Consider the following cyclic graph:
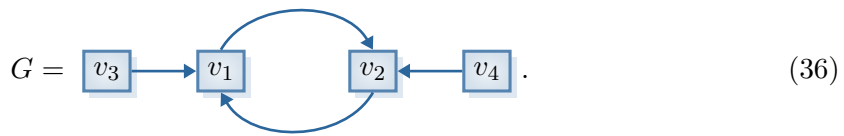
\begin{equation}
\label{eq:example_cycle_classical_overview}
    \graphname=\centertikz{
        \node[onode] (v1) at (-1, 0) {$v_1$};
        \node[onode] (v2) at (1, 0) {$v_2$};
        \node[onode] (v4) at (2.5, 0) {$v_4$};
        \node[onode] (v3) at (-2.5, 0) {$v_3$};
        \draw[cleg] (v1.north) to[in = 120, out = 60] (v2.north);
        \draw[cleg] (v2.south) to[in = 300, out = 240] (v1.south);
        \draw[cleg] (v4) -- (v2);
        \draw[cleg] (v3) -- (v1);
    }.
\end{equation}
A functional model on this graph is defined by associating finite sets $\outcomemaparg{i}$ and error variables $U_i\in\errormaparg{i}$ with distribution $\probex{i}$ to each vertex $\vertname_i$ for $i=1,2,3,4$, and functions $\funcarg{1}(\outcome_2,\outcome_3,u_1)$ to $\vertname_1$, $\funcarg{2}(\outcome_1,\outcome_4,u_2)$ to $\vertname_2$, $\funcarg{3}(u_3)$ to $\vertname_3$ and $\funcarg{4}(u_4)$ to $\vertname_4$ for all $\outcome_i\in\outcomemaparg{i}$ and $i=1,2,3,4$. For simplicity, let the functions associated to endogenous nodes, $\funcarg{1}$ and $\funcarg{2}$, be independent on the error variables and $\funcarg{3}(u_3)=u_3$, $\funcarg{4}(u_4)=u_4$. In this case, we can simplify the probability rule as explained in the remark of~\cref{remark:errorrv}.

This functional model is mapped to an acyclic functional model with post-selection, where we can evaluate probabilities using the acyclic rule in~\cref{def: distribution_functional_cm}. The probabilities of the acyclic model (conditioned on successful post-selection) are used to define the desired distribution of the cyclic functional model. Explicitly, the procedure is given by the following steps for this example:
\begin{myitem}
    \item Consider a subset of vertices $\splitvert{\graphname\sn}$ such that the subgraph $\graphname'$ obtained from $\graphname$ through removing all outgoing edges of vertices in  $\splitvert{\graphname\sn}$ is acyclic, e.g., considering $\splitvert{\graphname\sn} = \{\vertname_2\}$ we have
    \begin{equation}
    \graphname'=\centertikz{
        \node[onode] (v1) at (-1, 0) {$v_1$};
        \node[onode] (v2) at (1, 0) {$v_2$};
        \node[onode] (v4) at (2.5, 0) {$v_4$};
        \node[onode] (v3) at (-2.5, 0) {$v_3$};
        \draw[cleg] (v1.north) to[in = 120, out = 60] (v2.north);
        \node at (0,-0.5) {};
        \draw[cleg] (v4) -- (v2);
        \draw[cleg] (v3) -- (v1);
    }.
\end{equation}
    \item Construct a classical teleportation graph $\graphname\sn$ by adding for each vertex $\vertname\in\splitvert{\graphname\sn}$ the pre- and post-selection vertices of a classical post-selected teleportation protocol (\cref{def:classical ps tele} and \cref{def:sn_graph_family}), i.e.,
    \begin{equation}
    \graphname\sn=
    \centertikz{
        \node[prenode] (pre) at (0,-2.5) {$\prevertname$};
        \node[onode] (v3) at (-1.5, -2.5) {$v_3$};
        \node[onode] (v1) at (-1, -1) {$v_1$};
        \node[onode] (v2) at (1, 1) {$v_2$};
        \node[onode] (v4) at (1.5, -0.5) {$v_4$};
        \node[psnode] (ps) at (0.3,2.5) {$\postvertname$};
        \draw[cleg] (v3.90) -- (v1.240);
        \draw[cleg] (pre.140) -- (v1.300);
        \draw[cleg] (pre.60) -- (ps.240);
        \draw[cleg] (v2.90) -- (ps.300);
        \draw[cleg] (v1.north) -- (v2.240);
        \draw[cleg] (v4.north) -- (v2.300);
    }.
    \end{equation}
    \item Define a functional model on $\graphname\sn$ by keeping the same associations of the original functional model to all vertices and edges that are preserved from $\graphname$ to $\graphname\sn$ and associating a teleportation protocol to the added vertices, e.g., to $\prevertname$ associate the set $\outcomemaparg{2}$ and a uniform probability distribution $\probex{\prevertname}(r)=1/|\outcomemaparg{2}|$ for all $r\in\outcomemaparg{2}$ and to $\postvertname$ associate the set $\{0,1\}$ and function $\funcarg{\postvertname}(x_2,r) = \delta_{x_2,r}$ for all $x_2,r\in\outcomemaparg{2}$.
    \item Evaluate the probability of the functional model on the acyclic graph $\graphname\sn$ 
    \begin{equation}
        \probfacyc(\{x_i\}_{i=1}^4, r, t=1)_{\graphname\sn} =  \frac{ \probex{3}(x_3)  \probex{4}(x_4) \delta_{x_1, f^1(r,x_3)} \delta_{x_2, f^2(x_1,  x_4)} \delta_{x_2,r} }{|\outcomemaparg{2}|},
    \end{equation}
    consider $r$ unobserved, i.e., marginalize $r$,
    \begin{equation}
        \probfacyc(\{x_i\}_{i=1}^4, t=1)_{\graphname\sn} =  \frac{ \probex{3}(x_3)  \probex{4}(x_4) \delta_{x_1, f^1(x_2,x_3)} \delta_{x_2, f^2(x_1,  x_4)}}{|\outcomemaparg{2}|}
    \end{equation}
    and post-select on $t=1$
    \begin{align}
    \label{eq:example_probability}
        \probfacyc(\{x_i\}_{i=1}^4| t=1)_{\graphname\sn} =&   \frac{\probfacyc(\{x_i\}_{i=1}^4, t=1)_{\graphname\sn}}{\probfacyc(t=1)_{\graphname\sn}} \\
        =& \frac{ \probex{3}(x_3)  \probex{4}(x_4) \delta_{x_1, f^1(x_2,x_3)} \delta_{x_2, f^2(x_1,  x_4)}}{\sum_{y_i}\probex{3}(y_3)  \probex{4}(y_4) \delta_{y_1, f^1(y_2,y_3)} \delta_{y_2, f^2(y_1,  y_4)}}
    \end{align}
   
    \item If $\csuccessprob = \probfacyc(t=1)_{\graphname\sn}>0$, define the probability distribution in the functional model on the cyclic graph $\graphname$ as the conditional probability in~\cref{eq:example_probability}, i.e.,
    \begin{equation}
        \probf(x_1,x_2,x_3,x_4)_{\graphname} :=  \probfacyc(\{x_i\}_{i=1}^4| t=1)_{\graphname\sn}.
    \end{equation}
    If $\csuccessprob = 0$, we say that the model is inconsistent and the probabilities are undefined.
\end{myitem}

Now we give a specific functional model on $\graphname$ (\cref{eq:example_cycle_classical_overview}) and evaluate the probability as above. Let us consider binary variables, i.e., $\outcomemaparg{i}=\{0,1\}$, let $\funcarg{1}(\outcome_2,\outcome_3)=\outcome_3\oplus \outcome_2$ and $\funcarg{2}(\outcome_1,\outcome_4)=\outcome_1\oplus \outcome_4$, and consider arbitrary distributions $\probex{3}$ and $\probex{4}$.
This model is not uniquely solvable (see~\cref{sec:solvability}). Indeed, if $\outcome_3\neq\outcome_4$ the model has no solutions, while if $\outcome_3=\outcome_4$ it has $2$ solutions each satisfying $\outcome_2=\outcome_1\oplus\outcome_3$.
From the joint distribution,
\begin{equation}
    \probf(\outcome_1,\outcome_2,\outcome_3,\outcome_4)_\graphname = \frac{\probex{3}(\outcome_3)\probex{4}(\outcome_4)\delta_{\outcome_1,\outcome_2\oplus\outcome_3}\delta_{\outcome_2,\outcome_1\oplus\outcome_4}}{\mathcal{N}}= \frac{\probex{3}(\outcome_3)\probex{4}(\outcome_4)\delta_{\outcome_1,\outcome_2\oplus\outcome_3}\delta_{\outcome_3,\outcome_4}}{\mathcal{N}},
\end{equation}
where 
\begin{align}
\begin{split}    \mathcal{N}=&\sum_{\outcomealt_1,\outcomealt_2,\outcomealt_3,\outcomealt_4}\probex{3}(\outcomealt_3)\probex{4}(\outcomealt_4)\delta_{\outcomealt_1,\outcomealt_2\oplus\outcomealt_3}\delta_{\outcomealt_3,\outcomealt_4}\\
=&\sum_{\outcomealt_1,\outcomealt_2}\probex{3}(0)\probex{4}(0)\delta_{\outcomealt_1,\outcomealt_2}+\probex{3}(1)\probex{4}(1)\delta_{\outcomealt_1,\outcomealt_2\oplus 1} = 2\left(\probex{3}(0)\probex{4}(0)+\probex{3}(1)\probex{4}(1)\right),
\end{split}
\end{align} we can evaluate the conditional probability for the case were $\outcome_3=\outcome_4=0$. Previous methods would suggest that any distributions with $\outcome_1=\outcome_2$ and arbitrary weights on their value being $0$ or $1$ are allowed. Contrary to that, our method uniquely fixes the probability distribution to
\begin{equation}
    \probf(\outcome_1,\outcome_2|0,0)_\graphname = \frac{\probex{3}(0)\probex{4}(0)\delta_{\outcome_1,\outcome_2}}{\probex{3}(0)\probex{4}(0)\sum_{\outcomealt_1,\outcomealt_2}\delta_{\outcomealt_1,\outcomealt_2}}= \frac{\delta_{\outcome_1,\outcome_2}}{2}.
\end{equation}
This is not a consequence of the specific teleportation protocol (\cref{def:uniform prior tele}) we chose, as any other protocol yields the same result (\cref{corollary:probs indep of tele implementation classic}).

\paragraph{Example of~\cite{VilasiniColbeckPRL}.}
We now consider a particular functional model introduced in \cite{VilasiniColbeckPRL} which does not admit unique solutions. The model was introduced to illustrate a result regarding the relativistic principles of no causal loops and no superluminal signalling in Minkowski space-time.
Here, we are only interested in the properties and probability distribution of the underlying functional model and show that these are recovered in our framework. Consider the graph
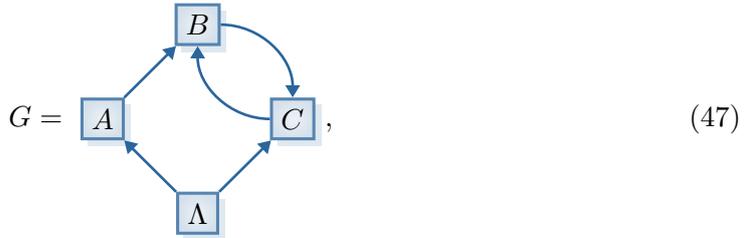
\begin{equation}
    \graphname= \centertikz{
        \node[onode] (A) at (0,0) {$A$};
        \node[onode] (C) at (2.5,0) {$C$};
        \node[onode] (B) at (1.25,1.25) {$B$};
        \node[onode] (L) at (1.25,-1.25) {$\Lambda$};
        \draw[cleg] (A) -- (B);
        \draw[cleg] (C.180) to[out=180, in=270] (B.270);
        \draw[cleg] (B) to[out=0, in=90] (C);
        \draw[cleg] (L) -- (A);
        \draw[cleg] (L) -- (C);
    },
\end{equation}
where all variables are binary and the following deterministic functions are associated to the vertices: $a=\funcarg{A}(\lambda):=\lambda$, $b=\funcarg{B}(a,c):=a\oplus c$ and $c=\funcarg{C}(\lambda,b):=\lambda\oplus b$ for all $a,b,c,\lambda\in\{0,1\}$, and let $\probex{\Lambda}$ be the distribution of the exogenous vertex $\Lambda$ (here we have again used the simplification given in the remark of~\cref{remark:errorrv} to ignore the error variables). Notice that for all values of $\Lambda$ the model has two solutions: indeed, if $\Lambda=0$, the triples $(a=0,b,c=b)$ are solutions for $b\in\{0,1\}$, and if $\Lambda=1$ the triples $(a=1,b,c=b\oplus 1)$ are solutions for $b\in\{0,1\}$.
This is reflected through the joint distribution of this model, evaluated with~\cref{prop:connecting channels}:
\begin{equation}
\label{eq:probability example VC22b}
\begin{split}
    \probf(a,b,c,\lambda=0)_\graphname = \frac{\probex{\Lambda}(0)\delta_{a,0}\delta_{c,b}\delta_{b,a\oplus c}}{\mathcal{N}}= \frac{\probex{\Lambda}(0)\delta_{a,0}\delta_{c,b}}{\mathcal{N}}\\
    \probf(a,b,c,\lambda=1)_\graphname = \frac{\probex{\Lambda}(1)\delta_{a,1}\delta_{c,b\oplus 1}\delta_{b,a\oplus c}}{\mathcal{N}}= \frac{\probex{\Lambda}(1)\delta_{a,1}\delta_{c,b\oplus 1}}{\mathcal{N}}
\end{split}
\end{equation}
where 
\begin{equation}
    \mathcal{N}=\sum_{a,b,c}\probex{\Lambda}(0)\delta_{a,0}\delta_{c,b}+\probex{\Lambda}(1)\delta_{a,1}\delta_{c,b\oplus 1}\delta_{b,a\oplus c}=2\left(\probex{\Lambda}(0)+\probex{\Lambda}(1)\right)=2.
\end{equation}
This shows again how the probability rule treats non-unique solutions.

In~\cite{VilasiniColbeckPRL}, the same example is analysed and a probability rule (which agrees with~\cref{eq:probability example VC22b}) is provided. The construction involves ``splitting'' vertices of the graph, in a way that is analogous to the construction of the elements of $\graphfamilysn\graphname$. Specifically, the construction in~\cite{VilasiniColbeckPRL} is obtained considering the acyclic graph 
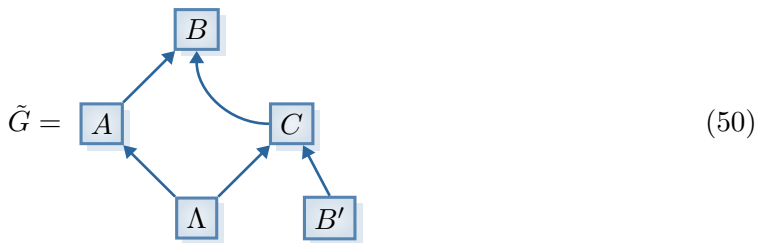
\begin{equation}
\label{eq: example_VVRC}
    \tilde{\graphname}= \centertikz{
        \node[onode] (A) at (0,0) {$A$};
        \node[onode] (C) at (2.5,0) {$C$};
        \node[onode] (B) at (1.25,1.25) {$B$};
        \node[onode] (L) at (1.25,-1.25) {$\Lambda$};
        \node[onode] (pre) at (3,-1.25) {$B'$};
        \draw[cleg] (A) -- (B);
        \draw[cleg] (C.180) to[out=180, in=270] (B.270);
        \draw[cleg] (pre.90) -- (C);
        \draw[cleg] (L) -- (A);
        \draw[cleg] (L) -- (C);
    }
\end{equation}
where the vertex $B$ has been split into $B$ and $B'$ and consider the acyclic distribution $\probfacyc(a,b,c|B'=b)_{\tilde{\graphname}}$ where we post-select on the values of $B'$ and $B$ being equal and renormalise. This method is analogous to the expression derived in~\cref{prop:connecting channels}. Notice that $\tilde{\graphname}$ above is similar to a classical teleportation graph $\graphname\sn$ of $\graphname$ (\cref{eq: example_VVRC}) that would be obtained by taking $\splitvert{\graphname\sn}:=\{B\}$, except that it is missing a post-selection vertex that would be a common child of $B$ and $B'$ (as the post-selection was done implicitly in \cite{VilasiniColbeckPRL} and not represented in the graph).

\section{A sound and complete graph-separation property for cyclic functional models}
\label{sec:pseparation}
As we have seen, in classical causal modeling, cause and effect relations between variables are represented through directed graphs. The power of causal modeling and causal inference come from results that link properties of the observable probability distribution to underlying topology of the graph. Specifically, a central result of this type is the $d$-separation theorem \cite{Verma1990, Geiger1990}, which allows to infer conditional independencies in the probability distribution over observed variables of any \textit{acyclic} functional model from the connectivity of the underlying graph (as captured by a graph-theoretic property, $d$-separation). However, this theorem is known to fail in functional models on cyclic graphs~\cite{Pearl_2009,Neal_2000}.

In the accompanying paper~\cite{Quantum_paper}, we propose a sound and complete graph-separation property, $p$-separation, applicable to all finite dimensional quantum causal models, which embed finite-cardinality fCMs. Here, we present an alternative formulation of $p$-separation that is independent of the quantum formalism, and defined in terms of the classical causal modeling framework we have developed in this paper (\cref{sec:fCM_to_afCM}). The key difference lies in the mapping of cyclic to acyclic causal models with post-selection, which differs in the classical (here) vs the quantum framework (\cite{Quantum_paper}) as classical variables can be perfectly copied while arbitrary quantum states cannot, by virtue of the no-cloning theorem \cite{Wootters1982, Dieks1982}. However, in appendix C of \cite{Quantum_paper}, we prove the equivalence between the two definitions for $p$-separation.

We begin in \cref{sec:dsep} by reviewing the definition of $d$-separation (which is defined on acyclic and cyclic graphs) and the $d$-separation theorem in acyclic graphs.
We then demonstrate in \cref{sec:d-sep_failure} the failure of $d$-separation in cyclic scenarios.
We introduce $p$-separation in \cref{sec: psep and examples}, using the framework of~\cref{sec:fCM_to_afCM}, proving the soundness and completeness of $p$-separation and discussing its properties. 
We describe examples of applications of $p$-separation in \cref{sec:examples psep}.
Finally, we we describe the relations between $p$-separation and $\sigma$-separation in \cref{sec:sigmasep}.

\subsection{Review of $d$-separation}
\label{sec:dsep}
The concept of $d$-separation \cite{Pearl_2009, Spirtes_2005} is a graph-theoretic property defined for any directed graph by looking at the structure of paths between sets of vertices. Let us motivate $d$-separation by considering the following graphs:
\begin{align}
\label{eq: chain}
 \text{Chain: } &\centertikz{  \node[onode] (a) at (-1.5,0) {$A$};  \node[onode] (c) at (0,0) {$C$};  \node[onode] (b) at (1.5,0) {$B$}; \draw[cleg] (a) --(c); \draw[cleg] (c)--(b);}, \\
\label{eq: fork}
 \text{Fork: } &\centertikz{  \node[onode] (a) at (-1.5,0) {$A$};  \node[onode] (c) at (0,0) {$C$};  \node[onode] (b) at (1.5,0) {$B$}; \draw[cleg] (c) --(a); \draw[cleg] (c)--(b);}, \\
\label{eq: collider}
 \text{Collider: } &\centertikz{  \node[onode] (a) at (-1.5,0) {$A$};  \node[onode] (c) at (0,0) {$C$};  \node[onode] (b) at (1.5,0) {$B$}; \draw[cleg] (a) --(c); \draw[cleg] (b)--(c);},
\end{align}

For causal models on the chain and the fork, we would generally expect $A$ and $B$ to get correlated, through the (indirect) causal influence of $A$ on $B$ in the first case and through the common cause $C$ in the second case. However, we would expect that $A$ and $B$ would become independent conditioned on $C$ as $C$ mediates the correlations between $A$ and $B$ in both cases. In the chain and fork, we would say $A$ is $d$-connected to $B$, denoted $A\not\perp^d B$ while $A$ is $d$-separated from $B$ given $C$, denoted $A\perp^d B|C$. For the collider on the other hand, $A$ and $B$ have no prior causes and we would expect them to be uncorrelated and we would say $A\perp^d B$. However, conditioning on their common child $C$ amounts to post-selection and can correlate $A$ and $B$, i.e., they become $d$-connected conditioned on $C$, $A\not\perp^d B|C$. Further, if in the case of the collider, we had $C\rightarrow D$ as shown below, then we would also expect $A$ and $B$ to get correlated given $D$ (as it is in the future of both), and expect $A\not\perp^d B|D$ also for descendants $D$ of a collider $C$. 
\begin{equation}
\label{eq: collider_desc}
 \text{Collider with descendant: }  \centertikz{  \node[onode] (a) at (-1.5,0) {$A$};  \node[onode] (c) at (0,0) {$C$};  \node[onode] (b) at (1.5,0) {$B$};  \node[onode] (d) at (0,1.5) {$D$};\draw[cleg] (a) --(c); \draw[cleg] (b)--(c); \draw[cleg] (c)--(d);}
\end{equation}

Although this intuition comes from thinking of the correlations, the following definitions are purely graph-theoretic. A subsequent theorem links this definition of $d$-separation to conditional independences in a causal model, which recovers the above intuition for these simple examples.

\begin{definition}[Blocked paths]
Let $\graphname$ be a directed graph in which $V_1$, $V_2$ and $V_3$ are disjoint sets of vertices, with $V_1$ and $V_2$ being non-empty.  A path (not necessarily directed) from $V_1$ to $V_2$ is
said to be \emph{blocked} by $V_3$ if it contains, for some vertices $A$ and $B$ in the path, either $A\rightarrow W\rightarrow B$
with $W\in V_3$, $A\leftarrow W\rightarrow B$
with $W\in V_3$ or $A\rightarrow W\leftarrow B$ such that neither the vertex $W$ nor any descendant of $W$ belongs to $V_3$.
\end{definition}

\begin{definition}[$d$-separation]
\label{def: d-sep}
Let $\graphname$ be a directed graph in which $V_1$, $V_2$ and $V_3$ are disjoint
sets of vertices with $V_1$ and $V_2$ non-empty.  $V_1$ and $V_2$ are \emph{$d$-separated} by $V_3$ in
$\graphname$, denoted $(V_1\perp^d V_2|V_3)_{\graphname}$, if for every pair of vertices in $V_1$ and $V_2$ there is no path between them, or if every path from a vertex in $V_1$ to a vertex in
$V_2$ is \emph{blocked} by $V_3$. Otherwise, $V_1$ is said to be \emph{$d$-connected} with $V_2$ given $V_3$, denoted as $(V_1\not\perp^d V_2|V_3)_{\graphname}$.
\end{definition}

The concept of $d$-separation only involves graph properties and is independent on whether a functional model is defined on the graph. However, the $d$-separation theorem relates the graph property to conditional independencies of the probability distribution of a functional model on the graph. For a proof and further details on the theorems see \cite{Verma1990, Geiger1990, Pearl_2009}.

\begin{definition}[Conditional independence]
\label{def:conditional independence}
    Let $V$ be a non-empty finite set, and
    let $\mathcal{P}(x)$ be a joint probability distribution over a set $X =
    \{X_v\}_{v \in V}$ of finite-cardinality random variables, whose values are
    denoted $x = \{x_v\}_{v \in V}$.
    Let $V_1$, $V_2$ and $V_3$ be three disjoint subsets of $V$, with $V_1$ and $V_2$ being non-empty. 
    We denote the corresponding sets of random variables as $X_i = \{ X_v \}_{v\in V_i}$ and the corresponding values as $x_i = \{ x_v \}_{v\in V_i}$ for $i \in \{1,2,3\}$.
    We say that $X_1$ is conditionally independent of $X_2$ given $X_3$ and denote it as $(X_1\indep X_2|X_3)_{\mathcal{P}}$ if, for all $x_1, x_2,x_3$, it holds that $\mathcal{P} (x_1,x_2|x_3)=\mathcal{P}(x_1|x_3)\mathcal{P}(x_2|x_3)$.
\end{definition}

\begin{theorem}[$d$-separation theorem for acyclic graphs]
\label{theorem: dsep theorem}
    Consider a directed acyclic graph $\graphname$ and let $V_1$, $V_2$ and $V_3$ be any three disjoint sets of the vertices of $\graphname$ with $V_1$ and $V_2$ being non-empty. 
    Then, the following holds:
    \begin{itemize}
    \item[]\textbf{\textup{(Soundness)}} For any functional model $\fcm_{\graphname}$ on $\graphname$, we have that $d$-separation between the vertex sets $V_i$ implies conditional independence for the corresponding sets of random variables $X_i:=\{X_{\vertname}\}_{\vertname\in V_i}$ where $i\in\{1,2,3\}$, i.e.,
        \begin{equation}
         (V_1\perp^d V_2|V_3)_{\graphname} \implies (X_1\indep X_2|X_3)_{\probfacyc_{\graphname} }.
        \end{equation} 
        \item[]\textbf{\textup{(Completeness)}} If the $d$-connection $(V_1\not\perp^d V_2|V_3)_{\graphname}$ holds in $\graphname$, then there exists a functional causal model  $\fcm_{\graphname}$ such that ${(X_1\not \indep X_2|X_3)_{\probfacyc_{\graphname}}}$.
    \end{itemize}
    The above conditional (in)dependence statements are relative to the marginal $\probfacyc(x_1,x_2,x_3)_\graphname$ on $X_1\cup X_2\cup X_3$, where $x_i = \{x_\vertname \in \outcomemaparg{\vertname}\}_{\vertname\in V_i}$, of the observed distribution $\probfacyc(\outcome)_{\graphname}$, where $\outcome=\{\outcome_{\vertname}\}_{\vertname\in \vertset}$, in the functional model $\fcm_{\graphname}$. 
\end{theorem}

\subsection{Failure of $d$-separation in cyclic graphs}
\label{sec:d-sep_failure}
The $d$-separation theorem (\cref{theorem: dsep theorem}) and, in particular, the soundness of $d$-separation are known to fail for functional models on cyclic graphs  \cite{Pearl_2009,Neal_2000}. For instance, consider the following cyclic graph and augmented version:
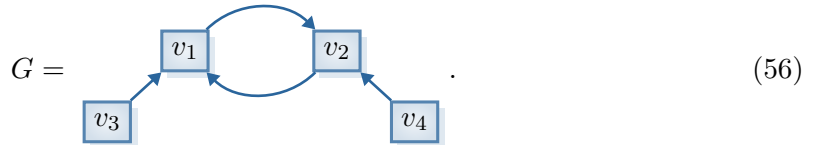
\begin{equation}
\label{eq: dsep_cycle_example}
    \graphname=\centertikz{
    \node[onode] (v1) at (0,0) {$\vertname_1$};
    \node[onode] (v2) at (2,0) {$\vertname_2$};
    \node[onode] (v3) [below left  =\chanvspace of v1] {$\vertname_3$};
    \node[onode] (v4) [below right  =\chanvspace of v2] {$\vertname_4$};
        \draw[cleg] (v1) to [out=45,in=135]  (v2);
    \draw[cleg] (v2) to [out=-135,in=-45] (v1);
 \draw[cleg] (v3) -- (v1);  \draw[cleg] (v4) -- (v2);  
    }.
\end{equation}
We have $(\vertname_3\perp^d \vertname_4)_{\graphname}$, thus, if the soundness of $d$-separation held true for cyclic models we would conclude that for any functional model on it, we must have the independence $(X_3\indep X_4)_{\probf_\graphname}$, where $X_i$ is the random variable associated to $\vertname_i$.
However, it is possible to construct a simple functional model where $\outcome_3$ and $\outcome_4$ must necessarily be correlated (and hence not independent). 

For instance, consider the example in~\cref{example: probability rule detail}. There we assumed that $\outcomemaparg{\vertname_i}=\{0,1\}$ for all $i=1,2,3,4$, i.e., all vertices are associated to binary variables, and considered the following deterministic functional dependencies $\outcome_1=\funcarg{\vertname_1}(\outcome_2,\outcome_3):=\outcome_2\oplus \outcome_3$ and $\outcome_2=\funcarg{\vertname_2}(\outcome_1,\outcome_4):=\outcome_1\oplus \outcome_4$, i.e., both $\outcome_1$ and $\outcome_2$ are given by the sum modulo 2 of their parents. In addition, consider error variables $U_3$ and $U_4$ distributed as $\probex{3}$ and $\probex{4}$ and functions $\funcarg{3}(u_3)=u_3$ and $\funcarg{4}(u_4)=u_4$ associated respectively to $\vertname_3$ and $\vertname_4$.
The probability distribution can be evaluated using~\cref{prop:connecting channels} and we obtain
\begin{align}
    \probf(\outcome_3,\outcome_4)_{\graphname}&=\mathcal{N}^{-1}\sum_{\outcome_1,\outcome_2} \probex{\vertname_3}(\outcome_3) \probex{\vertname_4}(\outcome_4) \delta_{\outcome_1,\outcome_2\oplus\outcome_3} \delta_{\outcome_2,\outcome_1\oplus\outcome_4}\\
    &=\mathcal{N}^{-1}\probex{\vertname_3}(\outcome_3) \probex{\vertname_4}(\outcome_4) \delta_{\outcome_3,\outcome_4},
\end{align}
where we defined $\mathcal{N}=\sum_{\outcome_3,\outcome_4}\probex{\vertname_3}(\outcome_3) \probex{\vertname_4}(\outcome_4) \delta_{\outcome_3,\outcome_4}$ for short (we assume here that $\probex{\vertname_3}$ and $\probex{\vertname_4}$ are such that $\mathcal N \neq 0$, in particular this would be the case if both correspond to the uniform distribution). Thus, $\outcome_3$ and $\outcome_4$ are perfectly correlated, since  $\probf(\outcome_3,\outcome_4)_{\graphname}=0$ whenever $\outcome_3\neq \outcome_4$. 

This example shows that $\outcome_3$ and $\outcome_4$ are correlated in the given fCM even though their associated vertices are $d$-separated in the graph. The reason for this lies on the fact that the loop between $\vertname_1$ and $\vertname_2$ effectively acts as a collider for $\vertname_3$ and $\vertname_4$. 
Conditioning on a collider can $d$-connect two variables which were $d$-separated before the conditioning. 
Here the collider is not explicitly conditioned upon in the original cyclic graph, however, the logical consistency of the model imposes  an effective post-selection on the values in the loop variables. In the next section, we introduce the notion of $p$-separation which generalizes $d$-separation through making the effective conditioning explicit.

In this example, $\outcome_3=u_3$ and $\outcome_4=u_4$, thus whenever $u_3\neq u_4$ also $\outcome_3\neq \outcome_4$ (as these are binary, we have $\outcome_4=\outcome_3\oplus 1$). Then, the two functional dependencies of the model reduce to $\outcome_1=\outcome_2$ and $\outcome_2=\outcome_1\oplus 1$, and there are zero solutions $(\outcome_1,\outcome_2)$ satisfying these for any values of $u_3\neq u_4$. Whenever $u_3\neq u_4$ also $\outcome_3= \outcome_4$, and both the functional dependencies reduce to $\outcome_1=\outcome_2$, which admits the two solutions $(\outcome_1=0,\outcome_2=0)$ and $(\outcome_1=1,\outcome_2=1)$ for any values of $u_3= u_4$. This model is therefore not uniquely solvable (see \cref{sec:solvability} for more precise and general definitions of solvability conditions).

\paragraph{Failure of $d$-separation in uniquely solvable models.} 
There also exist uniquely solvable models where the soundness of $d$-separation fails: \cite{Neal_2000} constructs such an example (albeit on a more complicated graph), which we review below. Consider the graph
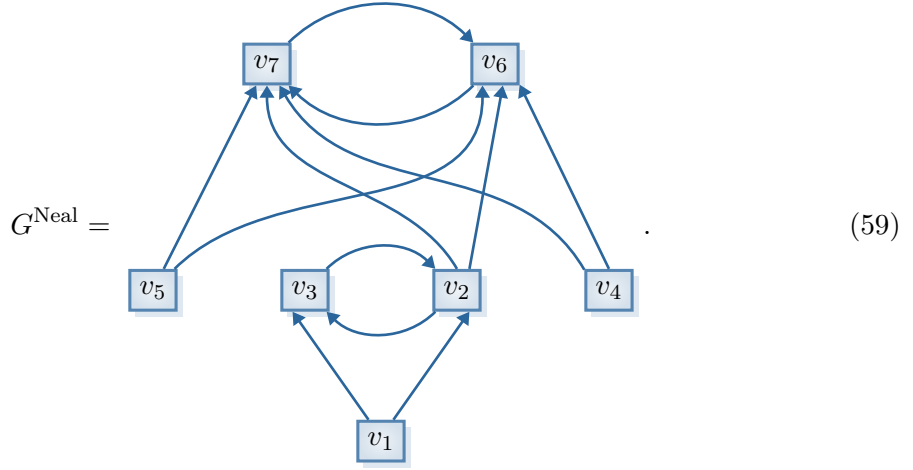
\begin{equation}
\label{eq: dsep_Neal_example}
    \graphname^{\textup{Neal}}=\centertikz{
    \node[onode] (v3) at (0,0) {$\vertname_3$};
    \node[onode] (v2) at (2,0) {$\vertname_2$};
    \node[onode] (v1) at (1,-2) {$\vertname_1$};
    \node[onode] (v4) at (4,0) {$\vertname_4$};
    \node[onode] (v5) at (-2,0) {$\vertname_5$};
    \node[onode] (v7) at (-0.5,3) {$\vertname_7$};
    \node[onode] (v6) at (2.5,3) {$\vertname_6$};
    \draw[cleg] (v3) to [out=45,in=135]  (v2);
    \draw[cleg] (v2) to [out=-135,in=-45] (v3);
    \draw[cleg] (v7) to [out=45,in=135]  (v6);
    \draw[cleg] (v6) to [out=-135,in=-45] (v7);
    
    \draw[cleg] (v1.120) -- (v3.240);  
    \draw[cleg] (v1.60) -- (v2.300); 
    \draw[cleg] (v2.60) -- (v6.290);  
    \draw[cleg] (v2.90) to[out=120,in=270] (v7.270); 
    \draw[cleg] (v4.90) -- (v6.320);  
    \draw[cleg] (v4.140) to[out=120,in=300] (v7.300); 
    \draw[cleg] (v5) to[in=270] (v6.240);  
    \draw[cleg] (v5) -- (v7); 
    }.
\end{equation}
Consider a functional model $\fcm_{\graphname^{\textup{Neal}}}$ which involves only binary variables i.e., $\outcomemaparg{\vertname_i}=\{0,1\}$ for all $i=1,2,3,4,5,6,7$. The exogenous vertices $\vertname_1$, $\vertname_4$ and $\vertname_5$ are given corresponding errors $U_1$, $U_4$ and $U_5$ uniformly distributed and the functional dependences relating the values of each vertex to those of its parents are deterministic and we therefore need not consider their error variables (see the remark in~\cref{remark:errorrv}):
\begin{align}
    \begin{split}
        \outcome_1&=u_1,\\
        \outcome_2&=\outcome_1\oplus \outcome_3,\\
        \outcome_3&=\outcome_1\oplus \outcome_2,\\
        \outcome_4&=u_4,\\
        \outcome_5&=u_5,\\
        \outcome_6&=(\outcome_2\oplus \outcome_4\oplus \outcome_5)\cdot(\outcome_7\oplus 1),\\  
        \outcome_7&=(\outcome_2\oplus \outcome_4\oplus \outcome_5)\cdot\outcome_6.
    \end{split}
\end{align}
To see that the model is uniquely solvable, note that consistency for the functional dependences of $\outcome_6$ and $\outcome_7$ forces $\outcome_2\oplus \outcome_4\oplus \outcome_5=0$, in which case $\outcome_6=\outcome_7=0$. Moreover,  $\outcome_2\oplus \outcome_4\oplus \outcome_5=0$ implies  $\outcome_2= \outcome_4\oplus \outcome_5$ and together with the functional dependence of $\outcome_3$ above gives $\outcome_3=\outcome_1\oplus \outcome_4\oplus \outcome_5$. Therefore, given any valuation of the error variables $u_1$, $u_4$ and $u_5$, there is a unique consistent valuation for the variables assigned to all the vertices: the given values $u_1$, $u_4$ and $u_5$ fix the valuations of the exogenous 
vertices $\outcome_1=u_1$, $\outcome_4=u_4$ and $\outcome_5=u_5$, which in turn uniquely determine $\outcome_2$, $\outcome_3$, $\outcome_6$ and $\outcome_7$ as explained above. Furthermore, we can see that the $d$-separation $(\vertname_4\perp^d\vertname_5|\vertname_2)_{\graphname^{\textup{Neal}}}$ holds here, although $(X_4\not\indep X_5|X_2)_{\probf_{\graphname^{\textup{Neal}}}}$ in the constructed model as $\outcome_2\oplus \outcome_4\oplus \outcome_5=0$ implies that $\outcome_4=\outcome_5$ whenever $\outcome_2=0$. In other words, the soundness of $d$-separation fails in this uniquely solvable functional model.

\subsection{Introducing \textit{p}-separation}
\label{sec: psep and examples}
In the previous section, we have seen that $d$-separation does not correctly capture correlations in cyclic models. Its failure suggested that cycles induce correlations in a way that could be captured through post-selection. Indeed, in \cref{sec:fCM_to_afCM} we provided a mapping for all finite cyclic fCMs into acyclic fCMs with post-selection. Based on this, we define a new graph-separation property $p$-separation for cyclic graphs, where $p$ stands for post-selection.

\begin{definition}[$p$-separation]
\label{def: p-separation}
 Let $\graphname$ be a directed graph and $V_1$, $V_2$ and $V_3$ denote any three disjoint subsets of the vertices of $\graphname$ with $V_1$ and $V_2$ being non-empty. Then, we say that $V_1$ is $p$-separated from $V_2$ given $V_3$ in $\graphname$, denoted $(V_1\perp^p V_2|V_3)_{\graphname}$, if and only if there exists $\graphname\sn\in \graphfamilysn{\graphname}$ (\cref{def:sn_graph_family}) such that $(V_1\perp^d V_2|V_3\cup\psvertset)_{\graphname\sn}$, where $\perp^d$ denotes $d$-separation and $\psvertset$ denotes the set of all post-selection vertices in the teleportation graph $\graphname\sn\in \graphfamilysn{\graphname}$. 
 Otherwise, we say that $V_1$ is $p$-connected to $V_2$ given $V_3$ in $\graphname$, and we denote it $(V_1\not\perp^p V_2|V_3)_{\graphname}$.
 To summarize,
 \begin{equation}
 \begin{aligned}
     \text{$p$-separation: } (V_1\perp^pV_2|V_3)_\graphname \equiva \exists \graphname\sn \in \graphfamilysn{\graphname} \st (V_1\perp^d V_2|V_3 \cup \psvertset)_{\graphname\sn}, \\
     \text{$p$-connection: } (V_1\not\perp^pV_2|V_3)_\graphname \equiva \forall \graphname\sn \in \graphfamilysn{\graphname} \st (V_1\not\perp^d V_2|V_3 \cup \psvertset)_{\graphname\sn}.
 \end{aligned}
 \end{equation}
\end{definition}

The following theorems, whose proofs can be found in~\cref{app:pseparation}, establish soundness and completeness of $p$-separation. These results prove that $p$-separation correctly captures conditional independences in finite-cardinality cyclic functional models.

\begin{restatable}[$p$-separation theorem]{theorem}{pseptheorem}
\label{theorem: psep_theorem}
    Consider a directed graph $\graphname$ and let $V_1$, $V_2$ and $V_3$ be any three disjoint sets of the vertices of $\graphname$ with $V_1$ and $V_2$ being non-empty. 
    Then, the following holds:
    \begin{itemize}
        \item[]\textbf{\textup{(Soundness)}} For any functional model $\fcm_{\graphname}$ on $\graphname$, we have that $p$-separation between the vertex sets $V_i$ implies conditional independence for the corresponding sets of random variables $X_i:=\{X_{\vertname}\}_{\vertname\in V_i}$ where $i\in\{1,2,3\}$, i.e.,
        \begin{equation}
         (V_1\perp^p V_2|V_3)_{\graphname} \implies (X_1\indep X_2|X_3)_{\probf_{\graphname} }.
        \end{equation} 
      
        \item[]\textbf{\textup{(Completeness)}} If the $p$-connection $(V_1\not\perp^p V_2|V_3)_{\graphname}$ holds in $\graphname$, then there exists a functional causal model  $\fcm_{\graphname}$ such that ${(X_1\not \indep X_2|X_3)_{\probf_{\graphname}}}$. 
    \end{itemize}
    The above conditional (in)dependence statements are relative to the marginal $\probf(x_1,x_2,x_3)_\graphname$ on $X_1\cup X_2\cup X_3$, where $x_i = \{x_\vertname \in \outcomemaparg{\vertname}\}_{\vertname\in V_i}$ of the observed distribution $\probf(\outcome)_{\graphname}$, where $\outcome=\{\outcome_{\vertname}\}_{\vertname\in \vertset}$, in the functional model $\fcm_{\graphname}$. 
\end{restatable}

The proof of~\cref{theorem: psep_theorem} is given in~\cref{app:pseparation}.

\subsection{Examples: $p$-separation in acyclic and cyclic graphs} 
\label{sec:examples psep}
\paragraph{Recovering $d$-separation in acyclic graphs.}
Notice that our definition of $p$-separation reduces to $d$-separation whenever the graph $\graphname$ is acyclic. This is because in this case the graph is a representative of its own acyclic family, $\graphname\in \graphfamilysn{\graphname}$ with an empty set of post-selection vertices $\psvertset=\emptyset$, and for $p$-separation, it is sufficient to have one graph from this family where $d$-separation conditioned on $\psvertset$ holds. Thus we immediately have that $d$-separation of $V_1$ from $V_2$ given $V_3$ implies $p$-separation of $V_1$ from $V_2$ given $V_3$ in any directed acyclic graph $\graphname$, according to \cref{def: p-separation}. The other direction, namely that $p$-separation implies $d$-separation is also true, as can be seen from the contrapositive that $d$-connection implies $p$-connection. Indeed, it is easy to see, from the definition of the graph family (\cref{def:sn_graph_family}) that if $V_1$ and $V_2$ are $d$-connected conditioned on $V_3$ in a graph $\graphname$, the same $d$-connection will hold in all teleportation graphs $\graphname\sn\in \graphfamilysn{\graphname}$ once we additionally condition on the post-selection vertices (as these can only act as conditioned colliders). Thus, $V_1$ and $V_2$ are $p$-connected conditioned on $V_3$.

We now highlight the significance in \cref{def: p-separation} of requiring that a $p$-separation relation holds if \emph{there exists} a teleportation graph $\graphname\sn\in \graphfamilysn{\graphname}$ in which a corresponding $d$-separation relation holds, as opposed to requiring that the $p$-separation relation holds if a corresponding $d$-separation relation holds \emph{for all}  teleportation graphs $\graphname\sn\in \graphfamilysn{\graphname}$. 
For this, consider the following simple directed acyclic graph $\graphname$:
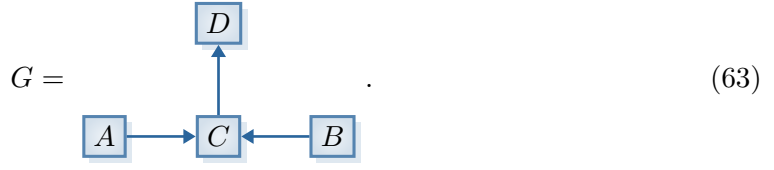
\begin{equation}
\label{eq: collider_desc_eg}
\graphname=  \centertikz{  \node[onode] (a) at (-1.5,0) {$A$};  \node[onode] (c) at (0,0) {$C$};  \node[onode] (b) at (1.5,0) {$B$};  \node[onode] (d) at (0,1.5) {$D$};\draw[cleg] (a) --(c); \draw[cleg] (b)--(c); \draw[cleg] (c)--(d);}.
\end{equation}
Consider the following teleportation graph $\graphname\sn\in \graphfamilysn{\graphname}$ in its acyclic family (\cref{def:sn_graph_family}) obtained by choosing $\splitvert{\graphname\sn}=\{C\}$:
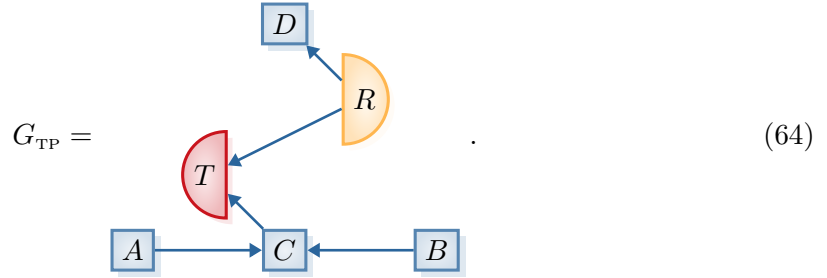
\begin{equation}
\label{eq: collider_desc_PS}
\graphname\sn= \centertikz{
  \node[onode] (A) at (1,0) {$A$};
       
        \node[onode] (C) at (3,0) {$C$};
       
         \node[onode] (B) at (5,0) {$B$};
 \node[psnode, rotate=90] (P3) at (2,1) {\phantom{$\postvertname_{33}$}};
 \node at (1.95,1) {$\postvertname$};
  \node[prenode, rotate=90] (Q3) at (4,2) {\phantom{$\prevertname_{33}$}};
   \node at (4.05,2) {$\prevertname$};
   \node[onode] (D)at (3,3) {$D$};
         
        \draw[cleg] (A) -- (C); 
        \draw[cleg] (B) -- (C);
         \draw[cleg] (C) -- (P3); \draw[cleg] (Q3) -- (P3); \draw[cleg] (Q3) -- (D);
    }.
\end{equation}
Let $\psvertset=\{\postvertname\}$ be the set of post-selection vertices in this graph. Notice that $\graphname\sn':=\graphname$, being acyclic, is also a member of its acyclic family $\graphfamilysn{\graphname}$, and we denote with $\psvertset'=\emptyset$ the (trivial) set of post-selection vertices in this case. 

It holds that $(A\not\perp^d B|\psvertset)_{\graphname\sn}$ but $(A\perp^d B|\psvertset')_{\graphname'\sn}$, since the latter is equivalent to $(A\perp^d B)_\graphname$. However $(A\not\perp^d B|\psvertset)_{\graphname\sn}$ does not imply $p$-connection as there exists an acyclic graph (in this case $\graphname\sn':=\graphname$) in the family having the relevant $d$-separation conditioned on the post-selection vertices. Thus, we have $(A\perp^p B)_\graphname$ in this example, which coincides with $d$-separation.

In other words, our definition only implies a $p$-connection in $\graphname$ when that connection is reflected in \emph{all} the graphs of the graph family $\graphfamilysn{\graphname}$. Indeed, we have shown (see \cref{lemma: acyclic_prob_same_func} and \cref{def: probability distribution cyclic fcm}) that the observed probabilities for the causal model are independent of the representative  of $\graphfamilysn{\graphname}$ chosen in computing them. It follows from these results that for any functional causal model on $\graphname\sn$ induced by a functional model on $\graphname$, the outcomes $a$ and $b$ of the vertices $A$ and $B$ will be conditionally independent even when conditioned on the collider $\psvertset:=\{\postvertname\}$ in \cref{eq: collider_desc_PS}. This is because the post-selection on the unblocking collider $\postvertname$ is fine-tuned in our definition of the induced model (it has to correspond to a post-selected teleportation protocol)\footnote{Specifically, the post-selection here serves to simulate a directed edge from $C$ to $D$ through an identity channel (without post-selecting a particular outcome value on $C$ or $D$), which is why it does not correlate the outcomes of $A$ and $B$.}. Thus, our definition of $p$-separation serves to avoid such fine-tuning and enables us to construct a sound and complete graph separation property for general directed graphs that reduces to $d$-separation in the acyclic case.

\paragraph{$p$-separation in action.} Going back to the example in \cref{sec:d-sep_failure}, we can see that although $d$-separation failed to detect the possibility of correlation between certain vertices, $p$-separation does. In particular, for the cyclic graph in \cref{eq: dsep_cycle_example}, consider the following member of $\graphfamilysn{\graphname}$ obtained by choosing $\splitvert{\graphname\sn}:=\{\vertname_1\}$.

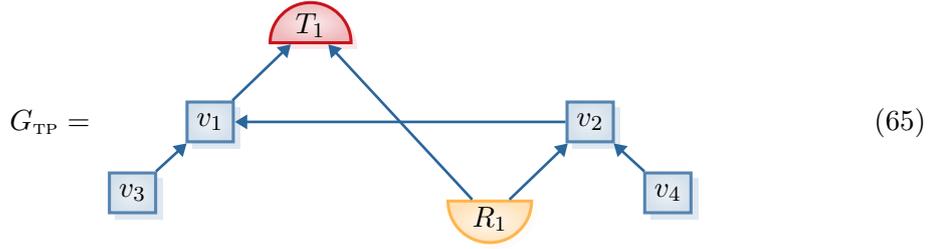
\begin{equation}
 \graphname\sn= \centertikz{
    \node[onode] (v1) at (0,0) {$\vertname_1$};
 \node[psnode] (p1) [above right =2*\chanvspace of v1] {$\postvertname_1$};
  
    \node[onode] (v2) at (5,0) {$\vertname_2$};

     \node[prenode] (q1) [below left =2*\chanvspace of v2] {$\prevertname_1$};

     \node[onode] (v3) [below left  =\chanvspace of v1] {$\vertname_3$};
   \node[onode] (v4) [below right  =\chanvspace of v2] {$\vertname_4$};
        \draw[cleg] (v1) -- (p1);  \draw[cleg] (q1) -- (p1);  \draw[cleg] (q1) -- (v2);  
               \draw[cleg] (v2) -- (v1);

 \draw[cleg] (v3) -- (v1);  \draw[cleg] (v4) -- (v2);  
    }
\end{equation}
We can see that $(\vertname_3\not\perp^d\vertname_4|\psvertset)_{\graphname\sn}$ as $\psvertset:=\{\postvertname_1\}$ acts as an unblocking collider. It is easy to see that every member of $\graphfamilysn{\graphname}$ will have this $d$-connection relative to its $\psvertset$, as each such member will necessarily include at least one of the loop vertices $\vertname_1$, $\vertname_2$ in $\splitvert{\graphname\sn}$ and hence always introduce at least one unblocking collider between $\vertname_3$ and $\vertname_4$. Therefore, by \cref{def: p-separation}, we have that $(\vertname_3\not\perp^p\vertname_4)_\graphname$. This $p$-connection explains the correlation between the outcomes of $\vertname_3$ and $\vertname_4$ that was observed in the functional model constructed in \cref{sec:d-sep_failure}. 

Moreover, note that the same cyclic graph $\graphname$ also has a non-trivial $p$-separation. In particular notice that $(\vertname_3\perp^d \vertname_4|\vertname_1,\vertname_2)_\graphname$. It is straightforward to see by applying \cref{def: p-separation} to the graph $\graphname\sn\in \graphfamilysn{\graphname}$ that we also have $(\vertname_3\perp^p \vertname_4|\vertname_1,\vertname_2)_\graphname$, since conditioning additionally on $\vertname_1$ and $\vertname_2$ blocks all paths between the remaining vertices and the post-selection vertex $\postvertname_1$. The soundness of $p$-separation then implies that the conditional independence $(X_3\indep X_4|X_1, X_2)_{\probf_\graphname}$ will hold for all valid functional causal models one can define on $\graphname$ within our framework.

Finally, applying similar arguments to the cyclic graph $\graphname^{\textup{Neal}}$ (\cref{eq: dsep_Neal_example}) underlying the example of \cite{Neal_2000}, where we had the $d$-separation $(\vertname_4\perp^d\vertname_5|\vertname_2)_{\graphname^{\textup{Neal}}}$, we can see that we would have the $p$-connection $(\vertname_4\not\perp^p\vertname_5|\vertname_2)_{\graphname^{\textup{Neal}}}$. This is because, just as in the simpler example above, any member of $\graphfamilysn{\graphname^{\textup{Neal}}}$ will necessarily involve at least one post-selection vertex within the loop between $\vertname_6$ and $\vertname_7$ and conditioning on this, the exogenous vertices $\vertname_4$ and $\vertname_5$ influencing the loop, will become $d$-connected (and stay that way also when conditioning on $\vertname_2$), i.e., $(\vertname_4\not\perp^d\vertname_5|\{\vertname_2\}\cup\psvertset)_{\graphname^{\textup{Neal}}\sn}$ will hold for every $\graphname^{\textup{Neal}}\sn\in \graphfamilysn{\graphname^{\textup{Neal}}}$. This implies the $p$-connection $(\vertname_4\not\perp^p\vertname_5|\vertname_2)_{\graphname^{\textup{Neal}}}$, which can explain the correlations seen between the outcomes $\outcome_4$ and $\outcome_5$ conditioned on $\outcome_2$ in the uniquely solvable functional model $\fcm_{\graphname^{\textup{Neal}}}$ of that example (reviewed in \cref{sec:d-sep_failure}).

\subsection{Links to $\sigma$-separation}
\label{sec:sigmasep}

In light of the failure of the $d$-separation theorem in cyclic causal models, another generalization of $d$-separation for cyclic graphs, known as $\sigma$-separation, was developed in \cite{Forre_2017}. This property,
$\sigma$-separation was proven to be sound and complete for a class of functional causal models (fCMs) known as modular fCMs or mfCMs~\cite{Forre_2017}\footnote{These are referred to as modular structural equation models, (mSEM), in \cite{Forre_2017}}. 
Modular fCMs correspond to fCMs defined generally on discrete and continuous variables, where the functional dependences possess a stronger version of the unique solvability property. 
Although mfCMs are a well-behaved and important subclass of classical causal models with many desirable properties, there do exist well-defined and consistent functional models outside this class for which the $\sigma$-separation theorem (i.e., the soundness and completeness of the property) no longer applies, and cannot account for the correlations there.
In contrast, the $p$-separation theorem proven here applies to all consistent functional causal models, but with the restriction to finite-cardinality variables. More generally, as shown in the companion paper \cite{Quantum_paper}, the theorem also extends to all consistent cyclic quantum causal models involving finite-dimensional Hilbert spaces. Thus the domains of validity of the $\sigma$- and $p$-separation theorems are quite distinct, neither being a subset of the other (non-uniquely solvable and quantum models are excluded in the former while functional models involving infinite-cardinality discrete variables as well as continuous variables are excluded in the latter). 

Despite the different approaches backing $\sigma$ and $p$-separation, some interesting points of comparison can be made, which suggest new research directions.
We do not discuss the definition of $\sigma$-separation here (for that we refer to \cite{Forre_2017}), but we consider a particular example often used in the previous literature. Consider the following cyclic graph:
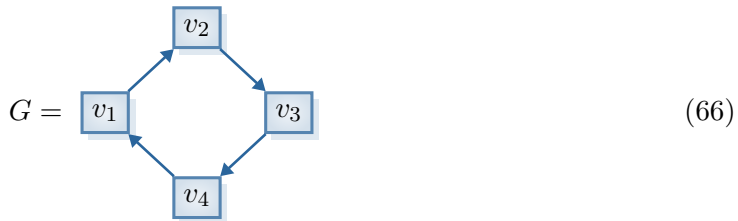
\begin{equation}
\label{eq: example_sigmasep}
    \graphname = \centertikz{
    \node[onode] (v1) {$\vertname_1$};
    \node[onode] (v2) [above right  =1.5*\chanvspace of v1] {$\vertname_2$};
    \node[onode] (v3) [below right  =1.5*\chanvspace of v2] {$\vertname_3$};
    \node[onode] (v4) [below left  =1.5*\chanvspace of v3] {$\vertname_4$};
   \draw[cleg] (v1) -- (v2); \draw[cleg] (v2) -- (v3); \draw[cleg] (v3) -- (v4); \draw[cleg] (v4) -- (v1);
 }
\end{equation}
Applying the definition of $d$-separation, it is easy to confirm that $(\vertname_1\perp^d \vertname_3|\vertname_2,\vertname_4)_\graphname$ and $(\vertname_2\perp^d \vertname_4|\vertname_1,\vertname_3)_\graphname$ hold in this graph. Nevertheless \cite{Forre_2017, forre_2018} shows that there exist functional causal models on this graph, where the variables associated with the vertices are continuous (take real values, $\mathbb{R}$) which can generate correlations between the variables associated with $\vertname_1$ and $\vertname_3$ even when conditioned on the variables of $\vertname_2$ and $\vertname_4$. On the other hand, relative to $\sigma$-separation, denoted as $\perp^\sigma$, $(\vertname_1\not\perp^\sigma \vertname_3|\vertname_2,\vertname_4)_\graphname$ and $(\vertname_2\not\perp^\sigma \vertname_4|\vertname_1,\vertname_3)_\graphname$ which can explain the correlations in this case (where $d$-separation fails to do so). In other words, observing such correlations in this graph would certify a correlation gap between $d$ and $\sigma$ separation. 

Interestingly, $p$-separation in this case agrees with $d$-separation rather than $\sigma$-separation, and we can check that we have the $p$-separations $(\vertname_1\perp^p \vertname_3|\vertname_2,\vertname_4)_\graphname$ and $(\vertname_2\perp^p \vertname_4|\vertname_1,\vertname_3)_\graphname$. This is because, even when introducing post-selection vertices for any number of edges and conditioning on them, the conditioning on the vertices $\vertname_2$ and $\vertname_4$  or $\vertname_1$ and $\vertname_3$ blocks all the paths between the remaining vertices under consideration. As a consequence of \cref{theorem: psep_theorem} about the soundness of $p$-separation, it follows that for every causal model on the above graph where the vertices have non-trivial outcome sets, we will have the conditional independencies $(X_1 \indep X_3|X_2, X_4)_{\probf_{\graphname}}$ and $(X_2 \indep X_4|X_1, X_3)_{\probf_{\graphname}}$. In other words, this implies that a gap between $d$ and $\sigma$ separation cannot be certified in the above graph using correlations arising in any classical causal model that can be described within our framework. Using the results of~\cite{Quantum_paper}, even allowing for quantum causal models on finite dimensional Hilbert spaces would not allow to certify this gap, as the soundness of $p$-separation extends to the quantum case.

A crucial point to note here is that our framework assumes only finite-cardinality random variables. The previously found distinguishing example between $d$- and $\sigma$-separation involves continuous variables, which are not covered by our formalism. 
This observation also highlights that the intuition given by $p$-separation may not immediately generalize to continuous models. Indeed, our method for computing probabilities in cyclic causal models relies on post-selecting on the precise values of certain variables. In the case of continuous variables, this post-selection event is of measure zero, and the associated probability of success would tend to zero\footnote{In the quantum case, this discussion is related to the difficulty in defining an analogue of a maximally entangled state between infinite-dimensional Hilbert spaces while ensuring a bounded norm for that state.}. In \cref{sec: discussion}, we further discuss possible directions for future research based on such comparisons between graph-separation properties and associated correlations.

\section{Solvability and related properties of cyclic functional models} 
\label{sec:solvability}
An interesting question arising in the literature is whether the functional dependencies of a given causal models can be jointly satisfied for a choice of value assignment over error vertices. Explicitly, given a valuation $\{u_{\vertname}\in\outcomemaparg{\vertname}\}_{\vertname\in\vertset}$ over the error variables, one considers whether there exists a set $\{\outcome_{\vertname}\in\outcomemaparg{\vertname}\}_{\vertname\in\vertset}$ of values on the vertices such that 
\begin{equation}
    \outcome_{\vertname} = \funcarg{\vertname}\Big(\outcome_{\parnodes{\vertname}},u_\vertname\Big)   \text{ for all } \vertname\in\vertset
\end{equation} 
holds (then the values entail a  \textit{consistent solution}), and, eventually, whether this set is unique (this entails a \emph{unique solution}). 
For functional models on arbitrary cyclic graphs, the existence of a consistent solution is not guaranteed. However, if we restrict to functional models on \textit{acyclic} graphs, for all valuations of the error variables, a consistent solution that satisfies all functional dependencies always exists and is unique \cite{Pearl_2009}.
In this section, we aim to characterize the existence and number of solutions of a given finite functional causal model, and discuss how this relates to other properties of the model, such as Markovianity.
\subsection{Number of solutions and consistency}
We begin by formally defining the number of solutions of a functional causal model over finite-cardinality variables on an arbitrary graph. Then, we review the notion of unique solvability and introduce a new notion of average unique solvability, proving links to other properties. The proofs relative to results of this section are given in~\cref{app:proofs solv}.

\begin{definition}[Number of solutions]
Consider a functional model on a directed graph $\graphname=\graphexpl$, $\fcm_{\graphname}$, and a set of value assignments on error variables, $u=\{u_{\vertname}\in\errormaparg{\vertname}\}_{\vertname\in\vertset}$. The number of solutions of the model, for the given value assignment on exogenous vertices, is defined as
\begin{equation}
    \numsol{u} = \sum_{\outcome} \prod_{\vertname\in\vertset} \delta_{\funcarg{\vertname}(\outcome_{\parnodes{\vertname}},u_\vertname), \outcome_\vertname},
\end{equation}
where we used the same notation of~\cref{def:functional_CM} and the sum is over the value assignments, $\outcome =\{\outcome_{\vertname}\in\outcomemaparg{\vertname}\}_{\vertname\in\vertset}$. 
\end{definition}

Notice that the number of solutions of a model is defined independently of the prior distributions $\probex{\vertname}$ associated to error variables $u_{\vertname}$ for $\vertname\in\vertset$ in that model. Thus, we introduce the notion of average number of solutions of a functional model which takes into account these distributions.

\begin{definition}[Average number of solutions]
\label{def:average number of sol}
Consider a functional model on a directed graph $\graphname=\graphexpl$, $\fcm_{\graphname}$, the average number of solutions of the model is defined as

\begin{equation}
    \avnumsol = \sum_{u} 
    \numsol{u}
    \prod_{\vertname\in\vertset}\probex{\vertname}(u_{\vertname}),
\end{equation}
where the sum runs over $u=\{u_{\vertname}\in\errormaparg{\vertname}\}_{\vertname\in\vertset}$.
\end{definition}

The following proposition connects the probability of successful post-selection (\cref{def: c_success_probability}) of a teleportation causal model in $\graphfamilysn{\graphname}$ to the average number of solutions of the underlying causal model $\fcm_{\graphname}$.

\begin{restatable}[Average number of solutions and success probability]{prop}{austopsucc}
\label{prop:ans to psuccess}
    Consider a functional causal model over a directed graph $\graphname=\graphexpl$, $\fcm_{\graphname}$. Let $\graphname\sn\in\graphfamilysn{\graphname}$ be a teleportation graph and $\fcm_{\graphname\sn}$ the corresponding teleportation causal model constructed from $\fcm_{\graphname}$. The post-selection success probability $\csuccessprob$ (\cref{def: c_success_probability}) of $\fcm_{\graphname\sn}$ satisfies
    \begin{equation}
        \csuccessprob =
    \left(\prod_{\vertname\in\splitvert{\graphname\sn}} \teleprob^{(\vertname)}\right) \avnumsol,
\end{equation}
where $\splitvert{\graphname\sn}$ is the set of split vertices of $\graphname\sn$ and $\teleprob^{(\vertname)}$ is the teleportation probability associated to the classical post-selected teleportation protocol implemented for the split vertex $\vertname\in\splitvert{\graphname\sn}$.
\end{restatable}

The previous proposition allows us to characterize inconsistent functional models, i.e., models where $\csuccessprob =0$, in terms of the average number of solutions. This follows immediately from~\cref{prop:ans to psuccess} and from the fact that $\teleprob >0$ by definition of teleportation protocol. 

\begin{restatable}[Characterization of inconsistent models]{corollary}{austoinconsistent}
    A functional causal model on a directed graph $\graphname$, $\fcm_{\graphname}$, is inconsistent (\cref{def: probability distribution cyclic fcm}) if and only if $\avnumsol = 0$.
\end{restatable}

An interesting subset of functional models is the set of so-called \textit{uniquely solvable} functional models (see also~\cite{Forre_2017, Bongers_2021}.
\begin{definition}[Uniquely solvable functional models]
A functional model $\fcm_\graphname$ on a directed graph $\graphname=\graphexpl$ is uniquely solvable if for all value assignments on error variables, $u=\{u_{\vertname}\in\errormaparg{\vertname}\}_{\vertname\in\vertset}$, it holds that
\begin{equation}
    \numsol{u} =1.
\end{equation}
\end{definition}
In analogy to uniquely solvable models, we can define averagely uniquely solvable models as follows.
\begin{definition}[Averagely uniquely solvable models]
\label{def:average uniquely solvable}
A functional model $\fcm_\graphname$ on a directed graph $\graphname=\graphexpl$ as in~\cref{def:functional_CM} is averagely uniquely solvable if it holds
\begin{equation}
    \avnumsol =1.
\end{equation}
\end{definition}

Clearly, every uniquely solvable model is also averagely uniquely solvable. This is proven in the following proposition.

\begin{restatable}[Unique solvability implies average unique solvability]{prop}{austous}
    A uniquely solvable functional model is also averagely uniquely solvable.
\end{restatable}

The converse, however, is not true. The examples of the next section show that there are functional models that are only averagely uniquely solvable.

\subsection{Examples on solvability}
In this section, we first consider acyclic models, then discuss two examples of cyclic functional models and their solvability.

\paragraph{Acyclic functional models.}
It is easy to see that acyclic models are uniquely solvable, indeed, given~\cref{def: distribution_functional_cm} we have
\begin{equation}
    \numsol{u} = \sum_{\outcome}\probfacyc(\outcome|u) = 1,
\end{equation}
where $u=\{u_{\vertname}\in\errormaparg{\vertname}\}_{\vertname\in\vertset}$ and the sum runs over $\outcome=\{\outcome_{\vertname}\in\outcomemaparg{\vertname}\}_{\vertname\in\vertset}$. 

\paragraph{Uniquely solvable cyclic functional model.}

Consider the graph 
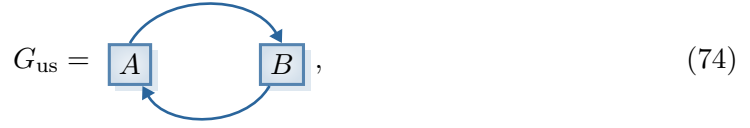
\begin{equation}
\label{eq:example_us_cycle}
    \graphname_{\textup{us}}=\centertikz{
        \node[onode] (A) at (-1, 0) {$A$};
        \node[onode] (B) at (1, 0) {$B$};
        \draw[cleg] (A.north) to[in = 120, out = 60] (B.north);
        \draw[cleg] (B.240) to[in = 300, out = 240] (A.300);
    },
\end{equation}
associate the sets $\outcomemaparg{A}=\outcomemaparg{B}=\{0,1,2\}$ to each vertex as well as the deterministic functions 

\begin{equation}
\begin{aligned}
    a&=\func^A(b):=2b \mod{3}, \\
    b&=\func^B(a):=a.
\end{aligned}
\end{equation}
Since the functions are deterministic, the error variables are irrelevant, and we obtain
\begin{equation}
    \textup{N}_{\fcm}=\sum_{a,b\in\{0,1,2\}} \delta_{a,\funcarg{A}(b)}\delta_{b,\funcarg{B}(a)} = \delta_{0,0}\delta_{0,0} =1,
\end{equation}
since the only consistent solution is $a=b=0$.

\paragraph{Averagely uniquely solvable cyclic functional model.}
An averagely uniquely solvable model is not necessarily uniquely solvable. 
For instance, consider the following graph:
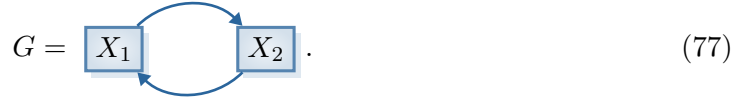
\begin{equation}
    \graphname = \centertikz{
        \node[onode] (x1) at (0,0) {$X_1$};
        \node[onode] (x2) at (2,0) {$X_2$};
        \draw[cleg] (x1) to [out=45,in=135]  (x2);
        \draw[cleg] (x2) to [out=-135,in=-45] (x1);
          
    }.
\end{equation}
Each random variable is taken to be binary, so that $ \outcomemaparg{X_1}=\outcomemaparg{X_2}=\{0,1\}$.
We choose a uniform probability distribution associated to the binary error variable of $X_2$, $U_2$, $\probex{2}(0)=\probex{2}(1)=1/2$, and for all $i=1,2$ we let the function $\funcarg{i}$ associated to the vertex $X_i$ be,
for all $u_2,x_i \in \{0,1\}$.,
\begin{equation}
\begin{aligned}
    x_1&=\funcarg{1}(x_2) := x_2, \\
    x_2&=\funcarg{2}(u_2,x_1) := x_1 \oplus u_2.
\end{aligned}
\end{equation}
We have, for $u_2=0$,
\begin{equation}
    \left.
    \begin{aligned}
       x_1&=\funcarg{1}(x_2) = x_2 \\
       x_2&=\funcarg{2}(0,x_1) = x_1
    \end{aligned}
    \right\}
   \implies x_1 = x_2.
\end{equation}
Hence, $\numsol{u_2=0} = 2$. 
For $u_2=1$,
\begin{equation}
    \left.
    \begin{aligned}
       x_1&=\funcarg{1}(x_2) = x_2\\ 
       x_2&=\funcarg{2}(1,x_1) = x_1\oplus 1
    \end{aligned}
    \right\}
    \implies x_1 = x_2 \neq x_1.
\end{equation}
Hence, $\numsol{u_2=1} = 0$.
The model is not uniquely solvable,
but it is averagely uniquely solvable, since using the given error distribution we obtain
\begin{equation}
    \avnumsol = \frac{1}{2}\numsol{u_2=0}+\frac{1}{2}\numsol{u_2=1} = 1.
\end{equation}

\subsection{Linking Markov factorization and average unique solvability}
\label{sec:linksmarkovianity}
An important property of probability distributions arising in causal models is the Markov property, which captures that the distribution factorizes according to the underlying directed graph. This is defined as follows.

\begin{definition}[Markov property in functional causal models]
\label{def:Markov}
    Consider a functional causal model $\fcm_\graphname$ over a directed graph $\graphname=\graphexpl$. If the joint distribution over observed events $\outcome=\{\outcome_{\vertname}\in\outcomemaparg{\vertname}\}_{\vertname\in\vertset}$ factorizes as in \cref{eq: Markov1}  i.e.,
    \begin{equation}
        \probf(\outcome)_{\graphname} = \prod_{\vertname\in\vertset} \prob^{(\vertname)}\left(\outcome_\vertname|\parnodes{\outcome_\vertname}\right)_{\graphname},
    \end{equation}
    where $\prob^{(\vertname)}$ is the conditional probability distribution specified by $\fcm_\graphname$ for all $\vertname\in\vertset$,
    then we say that the model satisfies the Markov property or is probabilistically Markovian.
\end{definition}
With this definition, it is immediate to see that all acyclic functional causal models satisfy this Markov property by \cref{def: distribution_functional_cm} of their probabilities. Moreover as the probability rule for uniquely solvable models is identical to the acyclic case,\footnote{
    One can see this explicitly in the upcoming \cref{eq:avg unique prob}.
} they also respect this property~\cite{Forre_2017,Bongers_2021}.
However, generic functional models on cyclic graphs are not necessarily Markovian, as there exist examples of non-uniquely solvable models  \cite{VilasiniColbeckPRA} whose distribution does not factorize as in~\cref{def:Markov}. Using our framework, we are able to characterize the set of functional models which satisfy the Markov property. The following propositions shows that in the finite cardinality case, Markovianity is satisfied in a strictly larger set of models than the uniquely solvable ones, namely, it is satisfied for all average unique solvable models.
Furthermore, this is the largest set of models satisfying Markovianity.

\begin{corollary}
Given a functional causal model on a directed graph $\graphname$, $\fcm_{\graphname}$, the following statements are equivalent:
\begin{myitem}
    \item $\fcm_{\graphname}$ is probabilistically Markovian;
    \item $\fcm_{\graphname}$ is averagely uniquely solvable, $\avnumsol=1$;
    \item for every choice of $\graphname\sn\in \graphfamilysn\graphname$ used in the definition of $\csuccessprob$ (\cref{def: c_success_probability}), the post-selection success probability is 
\begin{equation}
     \csuccessprob = \left(\prod_{\vertname\in\splitvert{\graphname\sn}} \teleprob^{(\vertname)}\right),
\end{equation}
where $\splitvert{\graphname\sn}$ is the set of split vertices of $\graphname\sn$ and $\teleprob^{(\vertname)}$ is the teleportation probability associated to the classical post-selected teleportation protocol implemented for the split vertex $\vertname\in\splitvert{\graphname\sn}$.
\end{myitem} 
\end{corollary}
The above corollary follows directly from~\cref{prop:ans to psuccess,prop:connecting channels}. Indeed, it holds that
\begin{align}
\label{eq:avg unique prob}
    \probf(\outcome) 
    &= \frac{
    \sum_{u}\prod_{\vertname\in\vertset}\probex{\vertname}(u_\vertname) \delta_{\outcome_{\vertname}, \funcarg{\vertname}\left(\outcome_{\parnodes{\vertname}},u_{\vertname}\right)}
    }{
    \sum_{y}\sum_{u}\prod_{\vertname\in\vertset}\probex{\vertname}(u_\vertname) \delta_{y_{\vertname}, \funcarg{\vertname}\left(y_{\parnodes{\vertname}},u_{\vertname}\right)}
    } \nonumber\\
    &= \frac{  \sum_{u}\prod_{\vertname\in\vertset}\probex{\vertname}(u_\vertname) \delta_{\outcome_{\vertname}, \funcarg{\vertname}\left(\outcome_{\parnodes{\vertname}},u_{\vertname}\right)}}{\avnumsol},
\end{align}
where the sum runs over $u=\{u_\vertname\in\errormaparg{\vertname}\}_{\vertname\in\vertset}$.
The numerator of the above equality yields the Markov property if and only if the constant denominator, $\avnumsol$, is equal to $1$, i.e., if the model is averagely uniquely solvable.

\paragraph{Links between $d$-separation and Markovianity.} In acyclic functional causal models, the probability distribution is given by~\cref{def: distribution_functional_cm}, which trivially satisfies the Markov property of \cref{def:Markov}. For these models we also know that $d$-separation is sound and complete (\cref{theorem: dsep theorem}). This is not a coincidence, as it is known that in directed acyclic graphs $d$-separation and Markovianity are equivalent to each other, and to another property (called the local Markov property) which states that the variable at each vertex is conditionally independent of its non-descendants given the parents~\cite{Lauritzen_1990,Geiger_1990}. However, in cyclic models this is no longer the case~\cite{Verma_1990}. Firstly, while we have shown that Markovianity in the sense of \cref{def:Markov} holds for all averagely uniquely solvable models, $d$-separation soundness fails in this set as shown by the example of the uniquely solvable model given in  \cite{Neal_2000} (reviewed in \cref{sec:d-sep_failure}).

Furthermore, consider the graph
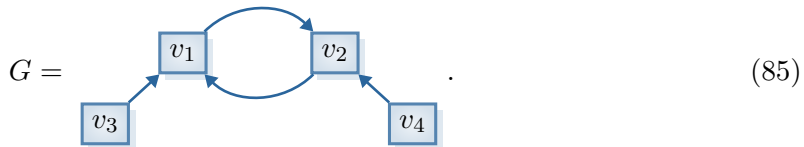
\begin{equation}
    \graphname=\centertikz{
        \node[onode] (v1) at (0,0) {$\vertname_1$};
        \node[onode] (v2) at (2,0) {$\vertname_2$};
        \node[onode] (v3) [below left  =\chanvspace of v1] {$\vertname_3$};
        \node[onode] (v4) [below right  =\chanvspace of v2] {$\vertname_4$};
        \draw[cleg] (v1) to [out=45,in=135]  (v2);
        \draw[cleg] (v2) to [out=-135,in=-45] (v1);
        \draw[cleg] (v3) -- (v1);  \draw[cleg] (v4) -- (v2);
    }.
\end{equation}
The local Markov property implies that $\vertname_2$ is independent of $\vertname_3$ (which is the only non-descendent which is not a parent of $\vertname_2$), conditioned on $\parnodes{\vertname_2}=\{\vertname_1,\vertname_4\}$, that is \linebreak${(X_2 \indep X_3 | X_1, X_4)_{\graphname}}$, where $X_i$ is the random variable associated to the vertex $\vertname_i$. 
However, it holds that $(\vertname_2\not\perp^d \vertname_3 |\vertname_1,\vertname_4)_{\graphname}$, since the path $\vertname_2 \rightarrow \vertname_1\leftarrow \vertname_3$ is a collider, hence conditioning on $\vertname_1$ $d$-connects $\vertname_2$ and $\vertname_3$ allowing the associated variables to be conditionally dependent in a general functional model on this graph. This highlights that in cyclic graphs, $d$-separation, Markovianity (as in \cref{def:Markov}) and the local Markov property are not equivalent.

\section{Discussion and outlook}
\label{sec: discussion}

We have developed a general framework applicable to all consistent cyclic 
functional causal models defined on finite-cardinality variables. The formalism includes both uniquely and non-uniquely solvable models and only excludes pathological models which admit no solutions for all possible valuations of the exogenous vertices\footnote{In case the graph contains no exogenous vertices, these would be models admitting no solutions.}.
We applied this framework to solve two open problems, for all (possibly cyclic) causal models describable within this formalism: (1) providing a general and robust method to uniquely fix the probability distribution, while recovering the known probability rule in the uniquely-solvable (and hence acyclic) case, and (2) introducing a new sound and complete graph-separation property, called $p$-separation, for cyclic models, which is equivalent to $d$-separation in the acyclic case. This was achieved by introducing the concept of a classical post-selected teleportation protocol (inspired by an analogous quantum information protocol), which enabled us to map cyclic functional causal models to acyclic ones with post-selection. Moreover, we defined a new class of models, averagely uniquely solvable functional causal models, finding this to be a strict superset of uniquely solvable ones, and the largest set where the previously known probability rule based on Markov factorization relative to the underlying graph is recovered. 

There are several interesting avenues for future investigations and we discuss some of them below.

\begin{myitem}
    \item \textbf{Causal discovery algorithms:} These algorithms aim to ascertain the underlying causal structure (graph) from observable data, or a set of possible graphs compatible with the data. In acyclic graphs, the $d$-separation theorem plays a crucial role in such algorithms (see \cite{Spirtes2016}).
    Causal discovery algorithms in the cyclic case have been proposed for different subsets of models, e.g., for modular functional causal models based on $\sigma$-separation theorem \cite{Mooij2020}. As we discussed in \cref{sec:pseparation}, in the finite-cardinality case, our framework and the $p$-separation theorem apply to more general models which include non-uniquely solvable ones.     
A natural next step would be to study whether robust causal discovery algorithms based on $p$-separation can be developed, as this would be applicable to all consistent finite-cardinality (possibly cyclic) functional causal models. 

\item \textbf{Causal compatibility problems:} This involves determining whether a given distribution could arise from any classical causal model (here, a functional causal model) on a given causal structure. Again, as this involves relating the causal structure and observable distributions, it is also intimately linked to the $d$-separation theorem. This problem has applications for classical causal inference on the one hand, because focussing on classical causal explanations, incompatibility with a graph would allow to rule out that graph as an explanation \cite{Pearl_2009}. On the other hand, if a distribution generated by a quantum causal model on the same graph\footnote{That is, using quantum states and measurements wired together as suggested by the graph. See \cite{Quantum_paper} for details.} is classically incompatible as above, this would certify a quantum-classical gap in that causal structure. Bell's seminal theorem~\cite{Bell_1964} provides the first instance of an acyclic graph with such a gap, which is verifiable experimentally and has lead to several quantum information processing applications (see \cite{Brunner_2014} for a review on Bell nonlocality, including relevant experiments and applications).
Our contributions enable to extend such studies to cyclic graphs.
In particular, since our framework maps cyclic causal modeling problems to acyclic causal modeling with post-selection, it is likely that existing techniques developed for solving acyclic causal compatibility problems can be ported to cyclic causal compatibility problems. For instance, it would be interesting to consider if the inflation technique for causal compatibility~\cite{Wolfe_2019} would generalize to cyclic causal models, using these ideas. 
\item 
\textbf{Correlation gaps between $d$-, $\sigma$- and $p$-separation and unfaithful models:} 
In \cref{sec:pseparation}, we demonstrated that in the graph of \cref{eq: example_sigmasep}, there exists a $d$-separation relation that aligns with $p$-separation but not with $\sigma$-separation. Combined with our $p$-separation theorem, this implies that for all functional causal models within our framework, the resulting correlations would be unfaithful or fine-tuned relative to $\sigma$-separation (as conditional independence holds, yet there is $\sigma$-connection).\footnote{A causal model is unfaithful (or fine-tuned) relative to a graph-separation property if the model entails a conditional independence between three sets of variables whose vertices are connected in the graph according to the property. If this is never the case, the model is faithful relative to the given graph-separation property.} This finding further suggests that within this graph, a correlation gap between $d$- and $\sigma$-separation cannot be established using functional models in our framework. However, there do exist continuous variable models, not captured within our framework, that are $\sigma$-faithful for this graph and certify such a gap \cite{Forre_2017, forre_2018}. Consequently, our results, together with these prior examples, also point to a correlation gap between finite-cardinality and infinite-cardinality or continuous functional causal models for the graph in \cref{eq: example_sigmasep}\footnote{In conjunction with the results of \cite{Quantum_paper}, this also highlights a gap between all finite-dimensional quantum causal models and continuous variable functional causal models in the graph of \cref{eq: example_sigmasep}.}. 

Similarly, we identified another situation—the graph in \cref{eq: dsep_cycle_example}—where the vertices $\vertname_3$ and $\vertname_4$ are both $d$- and $\sigma$-separated but not $p$-separated. The soundness of $\sigma$-separation in all modular functional causal models \cite{Forre_2017} implies that such models will be unfaithful relative to $p$-separation. Furthermore, it indicates that a correlation gap between $d$- and $p$-separation may not be certifiable through such models for this particular graph.

Whether these curious observations generalize to arbitrary directed graphs remains an open question. This opens new directions for future exploration into correlation gaps between constraints imposed by different graph separation properties and among various classes of causal models on a given directed graph. Additionally, in the acyclic case, it is well-established that the set of unfaithful causal models on a graph (i.e., where there is $d$-connection and yet conditional independence) forms a measure-zero set \cite{Pearl_2009, Spirtes_2005}. The above discussion sheds light on the open problem regarding the generalization of this result for cyclic graphs, for (un)faithfulness defined relative to the distinct properties of $\sigma$- or $p$-separation (both of which reduce to $d$-separation in the acyclic case).

\item \textbf{Links between graph separation properties and Markovianity:}  
As discussed in~\cref{sec:linksmarkovianity}, the notion of Markovianity and the soundness of $d$-separation are equivalent for functional models on acyclic graphs. However, this equivalence no longer holds for models on cyclic graphs. It would be interesting to characterize the class of cyclic models in our framework for which the equivalence is recovered. In particular, although we found that the Markov factorization is recovered for all averagely uniquely solvable models, the example of \cite{Neal_2000} (a uniquely solvable model) demonstrates that the $d$-separation theorem cannot be recovered for all such models. Therefore, tighter conditions would be required for the equivalence, and it would also be fruitful to consider other types of Markov properties such as the local Markov property, corresponding to conditional independence between a vertex and its non-descendants given its parents (which is also equivalent to $d$-separation properties in the acyclic case). In addition, relations between Markov properties and $p$-separation are yet to be explored as well as possible uses of $p$-separation in extending previous studies on Markov equivalence classes~\cite{Richardson_1997,claassen_2023}.

\item \textbf{Characterising classical processes with ``indefinite causal order'':} We note a striking similarity between our results of \cref{sec:linksmarkovianity} (regarding average unique solvability) and those in the literature on so-called indefinite causal order or higher-order processes \cite{Oreshkov_2012, Chiribella_2013, Baumeler_2016} studied in the quantum information community. Specifically, classical processes in such formalisms are higher-order transformations from 
functions to functions, where the set of input functions could be composed together in a manner where one cannot identify a definite acyclic ordering between them \cite{Baumeler_2016}. It was shown in \cite{Baumeler_2016_fixedpoints} that valid classical processes are characterized by having exactly one fixed point, and, by allowing probabilistic mixtures of such processes, the requirement becomes that of having one fixed point of average. There as well, one generally considers functions on finite alphabets, as in our functional causal modeling framework.

It is not difficult to see that each such classical process together with the input functions on which it acts specifies a (possibly cyclic) functional causal model that is valid in our framework\footnote{This is also implied by the our results in the accompanying paper \cite{Quantum_paper} for general quantum processes.}. This highlights a general link between average uniquely solvable functional causal models defined here and the class of classical higher-order processes. Thus, it indicates future avenues to investigate open problems on the characterization of higher-order processes by linking them to problems regarding solvability and Markov properties of cyclic functional models, and to study these in mutual synergy within a common formalism.

\item \textbf{Incorporating interventions:} We have focussed on properties of correlations obtainable through passive observation in cyclic fCMs. Another natural extension of this work would be to incorporate active interventions within this framework and study the consistency and solvability properties of fCMs under interventions. The class of allowed interventions is also expected to play a role. Commonly, do-interventions are considered which fix a variable to a certain value independently of its parents \cite{Pearl_2009, Spirtes_2005}. The above-mentioned links to higher-order processes also suggests more general interventions where the functional dependences may be replaced without ``cutting off'' a vertex from its parents.

\item \textbf{Generalization to infinite-cardinality random variables:} Our framework is applicable to functional causal models associated with finite-cardinality random variables. The generalization to infinite-cardinality models (including both discrete and continuous random variables) is not immediate because it would require post-selecting on a measure-zero event. Therefore, the question on whether and how the probability rule and $p$-separation can be generalized to (possibly non-uniquely solvable) infinite-cardinality models remains open. Whether the discrete with infinite-cardinality vs continuous case impacts this generalization is also open. In particular, the discussions comparing $p$ and $\sigma$ separation \cite{Forre_2017} above and in \cref{sec:pseparation} suggest that for $p$-separation as defined here, soundness could fail in continuous variable functional causal models. It would be interesting to investigate possible extensions of the classical post-selected teleportation protocol introduced here to the infinite-cardinality case, to explore modifications to $p$-separation that may arise in that case. 

\end{myitem}

\bigskip

\paragraph{Acknowledgements.} 

We are grateful to Elie Wolfe for valuable feedback on our framework and its relation to results in classical cyclic causal models. We are thankful to Y\'{i}l\'{e} Y\={\i}ng for suggesting the name $p$-separation for the new graph separation property that we propose here. VV's research has been supported by an ETH Postdoctoral Fellowship. VV acknowledges support from ETH Zurich Quantum Center, the Swiss National Science Foundation via project No.\ 200021\_188541 and the QuantERA programme via project No.\ 20QT21\_187724. VG and CF acknowledge support from NCCR SwissMAP, the ETH Zurich Quantum Center and the Swiss National Science Foundation via project No.\ 20QU-1\_225171. CF acknowledges support from the ETH Foundation.

\newpage
\bibliographystyle{mybibstyle}
%\bibliography{gencyclic_without_urls}

\newcommand{\etalchar}[1]{$^{#1}$}

\newpage
\appendix
\section{Proofs of~\cref{sec:fCM_to_afCM}}
\label{app:proofs map}
In this section, we provide proofs of the results in~\cref{sec:fCM_to_afCM}.

\indepprob*
\begin{proof}
    Assume that $\teleprob$ depends on the probability being teleported, i.e., $\teleprob=\teleprob\big(\probex{A}\big)$, and for all $c\in\outcomemaparg{C}$,
    \begin{equation}
        \sum_{\substack{a\in\outcomemaparg{A}\\b\in\outcomemaparg{B}}} \probex{A}(a)\delta_{\ctelefunc(a,b,c), 1} \cteleprob_B(b)\cteleprob_C(c)=\teleprob\big(\probex{A}\big) \probex{A}(c),
    \end{equation}
    and assume that there exist two distinct probabilities $\probex{A}$ and ${p'}^{A}$ such that $\teleprob(\probex{A})\neq\teleprob({p'}^{A})$. Consider a mixture of the two, i.e., $\lambda\probex{A}+(1-\lambda){p'}^{A}$ for some $\lambda\in[0,1]$, and teleport it:
    \begin{gather}
         \sum_{a,b} \left[\lambda\probex{A}(a)+(1-\lambda){p'}^{A}(a)\right]\delta_{\ctelefunc(a,b,c), 1} \cteleprob_B(b)\cteleprob_C(c)\\\label{eq:cptele e1}
         =\teleprob\left(\lambda\probex{A}+(1-\lambda){p'}^{A}\right) \left[\lambda\probex{A}(c)+(1-\lambda){p'}^{A}(c)\right].
    \end{gather}
    By linearity on the left-hand side of the equation, we have
     \begin{align}
         &\sum_{a, b} \left[\lambda\probex{A}(a)+(1-\lambda)\probex{A'}(a)\right]\delta_{\ctelefunc(a,b,c), 1} \cteleprob_B(b)\cteleprob_C(c)\nonumber\\
         &= \lambda\sum_{a, b} \probex{A}(a) \delta_{\ctelefunc(a,b,c), 1} \cteleprob_B(b)\cteleprob_C(c)+(1-\lambda)\sum_{a, b}\probex{A'}(a)\delta_{\ctelefunc(a,b,c), 1} \cteleprob_B(b)\cteleprob_C(c) \nonumber\\
         \label{eq:cptele e2}
         &= \lambda\teleprob\big(\probex{A}\big)\probex{A}(c)+(1-\lambda)\teleprob\big(\probex{A'}\big)\probex{A'}(c).
    \end{align}
    Combining~\cref{eq:cptele e1,eq:cptele e2} we have
     \begin{align}
     \label{eq:cptele e3}
        &\lambda\teleprob\left(\probex{A}\right)\probex{A}(c)+(1-\lambda)\teleprob\left({p'}^{A}\right)\probex{A'}(c)\nonumber\\
        &=\teleprob\left(\lambda\probex{A}+(1-\lambda){p'}^{A}\right) \left[\lambda\probex{A}(c)+(1-\lambda){p'}^{A}(c)\right],
    \end{align}
    and marginalising over $c$, 
    \begin{equation}
        \lambda\teleprob\big(\probex{A}\big) + (1-\lambda) \teleprob\big({p'}^{A}\big)
        = \teleprob\left(\lambda\probex{A}+(1-\lambda){p'}^{A}\right).
    \end{equation}
    Inserting the last equation in~\cref{eq:cptele e3}, we get a quadratic expression in $\lambda$ which has to hold for all $\lambda\in[0,1]$:
    \begin{equation}
        \left(\teleprob\big(\probex{A}\big)-\teleprob\big({p'}^{A}\big)\right) \left(\probex{A}(c)-{p'}^{A}(c)\right)\lambda^2 + (\dots)\lambda+(\dots) = 0.
    \end{equation}
    This implies that that the term before $\lambda^2$ should vanish for all $c\in\outcomemaparg{C}$. This contradicts our assumption that $\probex{A}\neq{p'}^{A}$ and $\teleprob(\probex{A})\neq\teleprob({p'}^{A})$.
\end{proof}

\copyprop*
\begin{proof}
    Consider a deterministic distribution $\probex{A}(a)=\delta_{a,\bar{a}}$, we have
    \begin{gather}
        \teleprob\delta_{\bar{a},c}=\sum_{\substack{a\in\outcomemaparg{A}\\b\in\outcomemaparg{B}}} \delta_{a,\bar{a}}\delta_{\ctelefunc(a,b,c), 1} \cteleprob_B(b)\cteleprob_C(c)=\sum_{b}\delta_{\ctelefunc(\bar{a},b,c), 1} \cteleprob_B(b)\cteleprob_C(c),
    \end{gather}
    which gives the result up changing the name of $\bar{a}$ to $a$.
\end{proof}

\begin{lemma}
    The uniform prior post-selected teleportation protocol (\cref{def:uniform prior tele}) is a valid classical post-selected teleportation protocol with $\teleprob=1/|\outcomemaparg{A}|$.
\end{lemma}
\begin{proof}
    We have, for all distribution $\probex{A}$ over $\outcomemaparg{A}$,
    \begin{equation}
        \sum_{a\in\outcomemaparg{A}} \probex{A}(a)\delta_{\delta_{a,c}, 1} \frac{1}{|\outcomemaparg{A}|}= \sum_{a\in\outcomemaparg{A}} \probex{A}(a)\delta_{a,c}\frac{1}{|\outcomemaparg{A}|}=\frac{\probex{A}(c)}{|\outcomemaparg{A}|}. \qedhere
    \end{equation}
\end{proof}

\acyclicity*
\begin{proof}
     In the first step of \cref{def:sn_graph_family}, the graph $\graphname'$ is acyclic by definition.
    The second step introduces pre- and post-selection vertices, $\prevertname_\vertname$ and $\postvertname_\vertname$ for each $\vertname\in\splitvert{\graphname\sn}$, and associated edges according to \cref{eq:split_node}. 
    To show that this step preserves acyclicity, notice that a directed graph is acyclic if and only if it can be drawn on a page with all directed edges oriented from bottom to top of the page\footnote{This representation is equivalent to a Hasse diagram for partially ordered sets, recognising the fact that every directed acyclic graph induces a partial order and vice versa.}. 
    Since $\graphname'$ is acyclic, we can represent it on a page through such a diagram. 
    We now add the pre- and post-selection vertices to this diagram by drawing all the $\prevertname_\vertname$ at the bottom of the page, below all other vertices, and all the $\postvertname_\vertname$ at the top of the page, above all the other vertices. 
    Since all preselection vertices $\prevertname_\vertname$ only have outgoing edges and all post-selection vertices $\postvertname_\vertname$ only have incoming arrows, both of which will be oriented from bottom to top in our diagram, it follows that $\graphname\sn$ is acyclic.
\end{proof}

\differentgraphsequiv*

\begin{proof}
In what follows we will use the following result: given a functional model (\cref{def:functional_CM}) on an acyclic graph $\tilde{\graphname} = (\tilde V,\tilde E)$ and a global observed event $\outcome=\{\outcome_{\vertname}\in\outcomemaparg{\vertname}\}_{\vertname\in\tilde{\vertset}}$, the acyclic probability rule~\cref{def: distribution_functional_cm} can be factorized as follows for $\outcome_0\in\outcome$ associated with the vertex $\vertname_0\in\tilde{\vertset}$:
\begin{equation}
    \probfacyc(x)_{\tilde{\graphname}} = \sum_{u} p(u,\bar{x})\prod_{\vertname'\in\childnodes{\vertname_0}} \delta_{\funcarg{\vertname}(\outcome_{\parnodes{\vertname}}, u_{\vertname}),x_{\vertname}}\probex{\vertname}(u_{\vertname})
\end{equation}
where we denoted $\bar{x}:=\{x_v\}_{v\in \tilde V \setminus \childnodes{v_0}}$ and defined 
\begin{equation}
\label{eq:p and f defs1}
    p(u,\bar{x})
    = \prod_{\substack{
        \vertname\in\tilde{V}\\
        \vertname\notin\childnodes{\vertname_0}
    }} 
    \probex{\vertname}(x_{\vertname}) 
    \delta_{\funcarg{\vertname}(x_{\parnodes{\vertname}},u_{\vertname}),\outcome_\vertname}.
\end{equation}
Notice that $\sum_{u}p(u,\bar{x})$ is just the marginal $\probfacyc(\bar x)_{\tilde \graphname}$ where $\bar x = \{x_v\}_{v \in \tilde V \setminus \childnodes{v_0}}$.
With this in mind, we first proceed with proving the result in a special case.
    \begin{itemize}
        \item[]\textbf{Proof in a special case.}  
    Consider the special case where $\graphname_1$ and $\graphname_2$ were constructed by choosing
    \begin{align}
        \splitvert{\graphname_2} = \splitvert{\graphname_1} \cup \{\vertname_0\}, 
    \end{align}
    i.e., the vertices which are split in $\graphname_2$ include all vertices that are split in $\graphname_1$ plus the vertex $\vertname_0$. Let $\psvertset^1$ denote the set of post-selection vertices of $\graphname_1$ and $\postoutcome=\{\postoutcome_{\postvertname}\in \{0,1\}\}_{\postvertname\in\psvertset^1}$ the set of outcomes associated to them. Then, taking $\postoutcome_0$ to be the outcome of the post-selection vertex associated to the vertex $\vertname_0$, we have $\postoutcome \cup \postoutcome_0:=\{\postoutcome_{\postvertname}\in \{0,1\}\}_{\postvertname\in\psvertset^2}$ as outcomes associated to the post-selection vertices of $\graphname_2$.

    Denoting with $(\outcome = \{\outcome_\vertname\in\outcomemaparg\vertname\}_{\vertname\in\vertset}, t)$ an event on $\graphname_1$, consider the probability $\probfacyc(\outcome,\postoutcome)_{\graphname_1}$ associated to the causal model $\fcm_{\graphname_1}$. 
    As this is an acyclic causal model by construction, the probability is immediately given by applying the acyclic probability rule of \cref{def: distribution_functional_cm} to this functional model. 
    As argued above, this can be factorized\footnote{This factorization corresponds to considering the product of all probabilities and delta functions of the functional model $\fcm_{\graphname_1}$ but leaving the multiplication along the split vertex $\vertname_0$ as the last step.} as
    \begin{align}
        \label{eq:proof graphfamily cl 1}
      \probfacyc(\outcome,\postoutcome)_{\graphname_1} = \sum_u p(u,\bar{x},t)\prod_{\vertname\in\childnodes{\vertname_0}} \delta_{\funcarg{\vertname}(\outcome_{\parnodes{\vertname}}, u_{\vertname}),x_{\vertname}}\probex{\vertname}(u_{\vertname})
    \end{align}
    where $p(u,\bar{x},t)$ was defined as prescribed in~\cref{eq:p and f defs1}\footnote{Note that for $\graphname_1$, the joint observed event includes the post-selections, i.e., it is $x\cup t$. Since $\vertname_0$ is not split in $\graphname_1$, $T\not\in\childnodes{\graphname_1}$ for all $T\in\postvertname$.}.
    Let $(\ctelefunc,\cteleprob_B,\cteleprob_C)$ be the classical post-selected teleportation protocol which implements the splitting of $\vertname_0$ in the functional model $\fcm_{\graphname_2}$, with associated pre- and post-selection vertices $\prevertname_0$ and $\postvertname_0$, and teleportation probability $\teleprob\in(0,1]$.
    For all, $a,c\in\outcomemaparg{\vertname_0}$, it holds that (\cref{lem: copy property aug})
    \begin{equation}
         \sum_{b\in\outcomemaparg{B}}\delta_{\ctelefunc(a,b,c), 1} \cteleprob_B(b)\cteleprob_C(c)=\teleprob \delta_{a,c}.
    \end{equation}
     Using this equality, we can rewrite
    \begin{equation}
         \begin{split}
             \probfacyc(\outcome,\postoutcome)_{\graphname_1} &= \sum_u p(u,\bar{x},t) \prod_{\vertname\in\childnodes{\vertname_0}} \delta_{\funcarg{\vertname}(\outcome_{\parnodes{\vertname}}, u_{\vertname}),x_{\vertname}} \probex{\vertname}(u_{\vertname})\\
             &=\sum_up(u,\bar{x},t) \prod_{\vertname\in\childnodes{\vertname_0}}\delta_{\funcarg{\vertname}(\outcome_{\parnodes{\vertname}\setminus \vertname_0}, x_0, u_{\vertname}),x_{\vertname}}\probex{\vertname}(u_{\vertname}) \sum_{r_0\in\outcomemaparg{\vertname_0}} \delta_{x_0,r_0} \\
             &=\sum_up(u,\bar{x},t) \prod_{\vertname\in\childnodes{\vertname_0}}\probex{\vertname}(u_{\vertname})\sum_{r_0} \delta_{\funcarg{\vertname}(\outcome_{\parnodes{\vertname}\setminus \vertname_0}, r_0, u_{\vertname}),x_{\vertname}}\delta_{x_0,r_0} \\
            &=\frac{1}{\teleprob} \sum_u p(u,\bar{x},t)  \prod_{\vertname\in\childnodes{\vertname_0}} \probex{\vertname}(u_{\vertname})\\
            &\hspace{1cm}\sum_{r_0,b}\delta_{\funcarg{\vertname}(\outcome_{\parnodes{\vertname}\setminus \vertname_0}, r_0, u_{\vertname}),x_{\vertname}}\delta_{\ctelefunc(x_0,b,r_0), 1}\cteleprob_B(b) \cteleprob_C(r_{0})\\
            &=\frac{1}{\teleprob} \probfacyc(\outcome,\postoutcome,\postoutcome_0 = 1)_{\graphname_2},
         \end{split}
    \end{equation}
    where in $\funcarg{\vertname}$ we separated the dependency on the parent node $\vertname_0$ as $\funcarg{\vertname}(\outcome_{\parnodes{\vertname}\setminus \vertname_0}, x_0, u_{\vertname})$.
    In the above steps, we have used that dividing by $\teleprob$ is allowed as $\teleprob>0$ by definition and \cref{lem: copy property aug} between the third to fourth step above.
    In the last equation, we used that $\fcm_{\graphname_1}$ and $\fcm_{\graphname_2}$ can be chosen to have the same associated $p(u,\bar{x},t)$ (to see this, one may restrict all the post-selected teleportation protocols to be implemented as in \cref{def:uniform prior tele} --- if the equality holds in this case, the general result follows from \cref{corollary:probs indep of tele implementation classic}).
    
    We can now relate the success probabilities $\successprob^{(1)}$ and $\successprob^{(2)}$ respectively of the functional model $\fcm_{\graphname_1}$ on the classical teleportation graph $\graphname_1$ and on $\fcm_{\graphname_2}$ on $\graphname_2$.
    By \cref{def: c_success_probability} and using that $\teleprob>0$, we have
    \begin{align}
    \label{eq:relating prob succ}
        \successprob^{(1)}
        &= \sum_\outcome \probfacyc(\outcome,\postoutcome = 1)_{\graphname_1} \\
        &= \frac{1}{\teleprob} \sum_\outcome  \probfacyc(\outcome,\postoutcome = 1,\postoutcome_0=1)_{\graphname_2} \\
        &= \frac{1}{\teleprob} \successprob^{(2)},
    \end{align}
    where we denoted with $\postoutcome = 1$ the event $\{\postoutcome_{\postvertname} =1\}_{\postvertname\in\psvertset^1}$, and the sum $\sum_\outcome$ runs over all $\outcome = \{\outcome_\vertname\in\outcomemaparg\vertname\}_{\vertname\in\vertset}$.
    
    We see that the two success probabilities differ by a multiplicative constant (the success probability $\teleprob$ of the post-selected teleportation protocol defined by $\fcm_{\graphname_1}$ for the split vertex $\vertname_0$), and hence, $\successprob^{(1)}$ is zero if and only if $\successprob^{(2)}$ is zero.
    Thus, the probabilities $\probfacyc(x|t=1)_{\graphname_1}$ are defined if and only if the probabilities $\probfacyc(x|t=1,t_0=1)_{\graphname_2}$ are defined. In this case, we can relate them as follows:
    \begin{equation}
        \begin{split}
            \probfacyc(x|t=1)_{\graphname_1}
            &= \frac{\probfacyc(\outcome,\postoutcome=1)_{\graphname_1}}{\successprob^{(1)}}\\
            &= \frac{\teleprob}{\teleprob} \frac{\probfacyc(\outcome,\postoutcome=1,\postoutcome_0=1)_{\graphname_2}}{\successprob^{(2)}} \\
            &= \probfacyc(x|t=1,\postoutcome_0=1)_{\graphname_2},
        \end{split}
    \end{equation}
    which completes the proof in this case.
    \item[]\textbf{Proof in the general case.} Consider two general elements $\graphname_1, \graphname_2$ of $\graphfamilysn{\graphname}$.
    We can prove the general statement by first noticing that the repeated application of the above argument can be used to prove that the probabilities of both $\graphname_1$ and $\graphname_2$ are equivalent to those of $\graphname_0$, where the latter corresponds to the graph where all the vertices have been split (i.e., $\splitvert{\graphname_0} = \vertset$, where $\graphname = (\vertset,\edgeset)$).
    The result follows by transitivity. \qedhere
    \end{itemize}
\end{proof}

\selfcyclecl*
\begin{proof}
    We have shown that the probability rule is independent on the set of split vertices $\splitvert{\graphname\sn}$, hence we can consider the element $\graphname_0=(\vertset_0,\edgeset_0)$ of $\graphfamilysn{\graphname}$ where $\splitvert{\graphname\sn} = \vertset$, i.e., all vertices are split and associated with pre- and post-selection vertices and edges. 
    Therefore, the set of vertices of $\graphname_0$ is $\vertset_0=\vertset\cup \{\prevertname_{\vertname}, \postvertname_\vertname\}_{\vertname\in\vertset}$, i.e., for each vertex in $\graphname$ there is a pre- and post-selection pair of vertices in $\graphname_0$. 
    For $\vertname\in\vertset$ we denote with $(\ctelefunc^\vertname,\cteleprob^\vertname_B,\cteleprob^\vertname_C)$ the post-selected teleportation protocol associated to the corresponding pre- and post-selection pair.
    
    Denoting with $t_{\vertname}\in\{0,1\}$ the outcome associated to $\postvertname_{\vertname}$ and with $r:=\{r_\vertname\in\outcomemaparg{\vertname}\}_{\vertname\in\vertset}$ the outcomes associated to the preselection vertices $\{\prevertname_{\vertname}\}_{\vertname\in\vertset}$, we have (\cref{def: distribution_functional_cm})
    \begingroup
      \crefname{lemma}{\textup{lem.}}{\textup{Lem.}}
    \begin{align}
        \probfacyc(\outcome, \{\postoutcome_\vertname = 1\}_{\vertname\in\vertset})_{\graphname_0}
        &= \sum_{r}\probfacyc(\outcome, r, \{\postoutcome_\vertname = 1\}_{\vertname\in\vertset})_{\graphname_0} \nonumber\\
        &= \sum_{r,b,u} \prod_{\vertname\in\vertset} \probex{\vertname} (u_{\vertname})\delta_{\outcome_{\vertname}, \funcarg{\vertname}(r_{\parnodes{\vertname}},u_\vertname)} \delta_{\ctelefunc^{\vertname}(\outcome_\vertname,b_\vertname,r_{\vertname}),1} \cteleprob_B^{\vertname}(b_\vertname) \cteleprob_C^{\vertname}(r_\vertname) \nonumber\\
        &\hspace{-0.33cm}\stackrel{\cref{lem: copy property aug}}{=} \sum_{r,u} \prod_{\vertname\in\vertset} \probex{\vertname} (u_{\vertname})\delta_{\outcome_{\vertname}, \funcarg{\vertname}(r_{\parnodes{\vertname}},u_\vertname)} \teleprob^{(\vertname)} \; \delta_{\outcome_{\vertname},r_{\vertname}}  \nonumber\\
        &= \teleprob^{\text{tot}}\sum_{u} \prod_{\vertname\in\vertset} \probex{\vertname} (u_{\vertname})\delta_{\outcome_{\vertname}, \funcarg{\vertname}(\outcome_{\parnodes{\vertname}},u_\vertname)} \nonumber\\
    \end{align}
    \endgroup
    where $\teleprob^{\text{tot}}=\prod_{\vertname\in\vertset}\teleprob^{(\vertname)}$ and $\teleprob^{(\vertname)}$ is the probability of successful post-selection of the implementation chosen for the vertex $\vertname$. Therefore,
    \begin{align}
        \probf(\outcome)_{\graphname} 
        &=\frac{\probfacyc(\outcome, \{\postoutcome_\vertname = 1\}_{\vertname\in\vertset})_{\graphname_0}
        }{
        \sum_{\outcomealt}\probfacyc(\outcomealt, \{\postoutcome_\vertname = 1\}_{\vertname\in\vertset})_{\graphname_0}
        } \nonumber\\
        &= \frac{
        \teleprob^{\text{tot}}\sum_{u} \prod_{\vertname\in\vertset} \probex{\vertname} (u_{\vertname})\delta_{\outcome_{\vertname}, \funcarg{\vertname}(\outcome_{\parnodes{\vertname}},u_\vertname)}
        }{
        \teleprob^{\text{tot}}\sum_{y,u} \prod_{\vertname\in\vertset} \probex{\vertname} (u_{\vertname})\delta_{\outcomealt_{\vertname}, \funcarg{\vertname}(\outcomealt_{\parnodes{\vertname}},u_\vertname)}
        } \nonumber\\
        &= \frac{
        \sum_{u} \prod_{\vertname\in\vertset} \probex{\vertname} (u_{\vertname})\delta_{\outcome_{\vertname}, \funcarg{\vertname}(\outcome_{\parnodes{\vertname}},u_\vertname)}
        }{
        \sum_{y,u} \prod_{\vertname\in\vertset} \probex{\vertname} (u_{\vertname})\delta_{\outcomealt_{\vertname}, \funcarg{\vertname}(\outcomealt_{\parnodes{\vertname}},u_\vertname)}
        }. \qedhere
    \end{align}
\end{proof}

\section{Proofs of~\cref{sec:pseparation}}
\label{app:pseparation}
We provide the proof of~\cref{theorem: psep_theorem}. 
\pseptheorem*
\begin{proof}
\textbf{(Soundness)}
    From \cref{def: p-separation}, we have $(V_1\perp^p V_2|V_3)_{\graphname} \equiva \exists  \graphname\sn\in \graphfamilysn\graphname$ such that $(V_1\perp^d V_2|V_3\cup\psvertset)_{\graphname\sn}$, where $\psvertset$ is the set of all post-selection vertices in the chosen classical teleportation graph $\graphname\sn$. 
    From \cref{theorem: dsep theorem}, it then follows that the conditional independence $(X_1\indep X_2|X_3\cup\psvertset)_{(\probfacyc)_{\graphname\sn}}$ holds since $\graphname\sn$ is a directed acyclic graph (\cref{lemma: acyclicity_cl_telegraphs}). Finally, by \cref{def: probability distribution cyclic fcm}, the probability distribution associated with a functional model on the possibly cyclic graph $\graphname$ is given by the corresponding probability computed in any representative teleportation graph $\graphname\sn\in \graphfamilysn\graphname$ with an additional conditioning on the associated post-selection vertices $\psvertset$. Recall from \cref{lemma: acyclic_prob_same_func} that this distribution is the same, independently of which representative $\graphname\sn\in \graphfamilysn\graphname$ is chosen.  Therefore, using \cref{def: probability distribution cyclic fcm}, we know that $(X_1\indep X_2|X_3\cup \psvertset)_{(\probfacyc)_{\graphname\sn}}$ is equivalent to $(X_1\indep X_2|X_3)_{\probf_{\graphname}}$ by \cref{def: p-separation}, which completes the proof.\\

\noindent
\textbf{(Completeness)} We begin by making an important observation.

\hypertarget{obs1}{}\hyperlink{obs1}{\it Observation 1:} Consider a directed graph $\graphname:=(\vertset,\edgeset)$ and an arbitrary subgraph $\graphname':=(\vertset',\edgeset')$ of $\graphname$ with $\vertset'\subseteq \vertset$ and $\edgeset'\subseteq\edgeset$. 
Then any functional causal model $\fcm_{\graphname'}$ (\cref{def:functional_CM}) on the subgraph trivially induces a functional causal model $\fcm_\graphname$ on the original graph $\graphname$ where both models assign the same probabilities to the values $x':=\{x_\vertname\}_{\vertname\in\vertset'}$ of variables associated to $\vertset'$, i.e., $\prob_{\graphname'}(\outcome')=\prob_\graphname(\outcome')$. Explicitly, $\fcm_\graphname$ is constructed from the given $\fcm_{\graphname'}$ by including the same error distribution $\probex{\vertname}(u_\vertname)$ and same function $\funcarg{\vertname}$ to each $\vertname\in \vertset'$ in both models, in $\graphname$, these functions ignore any dependencies on parents of $\vertname$ associated with edges in $\edgeset\backslash \edgeset'$. For all remaining vertices, $\vertname\in \vertset\backslash \vertset'$, the error distributions $\probex{\vertname}(u_\vertname)$ can be arbitrary and the functions set $\outcome_\vertname=u_\vertname$. As the models $\fcm_\graphname$ and $\fcm_{\graphname'}$ are exactly the same for the common vertices $\vertset'$, it is immediate that the probabilities for $\outcome'$ are identical.

Now suppose we have a $p$-connection $ (V_1\not\perp^p V_2|V_3)_{\graphname}$, and consider two cases based on whether or not the corresponding $d$-connection holds. We will establish the existence of a causal model with the corresponding conditional dependence separately in each case.
    \begin{itemize}
        \item[] \textbf{Proof when $d$- and $p$-connections coincide}, i.e., $ (V_1\not\perp^d V_2|V_3)_{\graphname} $: 
        
        By \cref{def: d-sep} of $d$-separation, this $d$-connection implies the existence of a path between $V_1$ and $V_2$ that is not blocked by $V_3$. Since $V_1$ and $V_2$ are disjoint, such a path (all the vertices and edges included in the path) define an acyclic\footnote{This is the case because by definition a path involves distinct vertices.} subgraph $\graphname'$ of $\graphname$, formed by all the vertices of $\graphname$ but only the edges that contribute to the $d$-connecting path. The completeness of $d$-separation for classical acyclic causal models that is well-known~(\cref{theorem: dsep theorem}) then implies that there must exist a causal model on $\graphname'$ where $(X_1\not\indep X_2|X_3)_{\probfacyc_{\graphname'}}$. Since $\graphname'$ is a subgraph of $\graphname$, this also constitutes a causal model on $\graphname$ by \hyperlink{obs1}{\it Observation 1},\footnote{This causal model can potentially be fine-tuned relative to other $d$-separations i.e., there can be $d$-connections between other sets of vertices in $\graphname$ that are not accompanied by the corresponding conditional dependence in this causal model. The completeness statement here only requires the model to reproduce the conditional dependence for the given $d$-connection $ (V_1\not\perp^d V_2|V_3)_{\graphname} $. } which gives the required conclusion for the outcome sets of interest, $(X_1\not\indep  X_2|X_3)_{\probf_{\graphname}}$.
  
        \item[] \textbf{Proof when $d$- and $p$-connections do not coincide}, i.e., $ (V_1\perp^d V_2|V_3)_{\graphname} $:
  
        Here we have the $d$-separation $(V_1\perp^d V_2|V_3)_{\graphname} $ but the $p$-connection $(V_1\not\perp^p V_2|V_3)_{\graphname}$. The former $d$-separation implies the same in all teleportation graphs $\graphname\sn\in \graphfamilysn\graphname$ of $\graphname$, i.e., $(V_1\perp^d V_2|V_3)_{\graphname\sn} $. This is because teleportation graphs are constructed by removing edges and adding colliders (from $\psvertset$). Since these are not conditioned on in this $d$-separation, they cannot convert the given $d$-separation into a $d$-connection by \cref{def: d-sep}. 
        The $p$-connection $(V_1\not\perp^p V_2|V_3)_\graphname$ is equivalent, by \cref{def: p-separation}, to $d$-connection in some teleportation graph $\graphname\sn\in \graphfamilysn\graphname$ of $\graphname$, including additional conditioning on the set $\psvertset$ of post-selection vertices in $\graphname\sn$. That is, we have the following two conditions satisfied in this case:
        \begin{equation}
        \label{eq:psep_compl_1}
            \begin{aligned}
                (V_1&\perp^d V_2|V_3)_{\graphname\sn} \\
                (V_1&\not\perp^d V_2|V_3\cup\psvertset)_{\graphname\sn}.
            \end{aligned}
        \end{equation}

           This is equivalent to saying that there exist $v_1\in V_1$ and $v_2\in V_2$ such that 
        \begin{equation}
        \label{eq:psep_compl_2}
            \begin{aligned}
            (v_1 &\perp^d v_2|V_3)_{\graphname\sn} \\
            (v_1 &\not\perp^d v_2|V_3\cup\psvertset)_{\graphname\sn}.
            \end{aligned}
        \end{equation}

        Again, as in the previous case, we can use the existence of a path which $d$-connects $v_1$ and $v_2$ conditioned on $V_3$ and $\psvertset$ to construct an acyclic subgraph $\graphname\sn'$ of $\graphname\sn$, but introduce a slight modification to handle post-selection vertices. Here we consider any choice of $d$-connecting (or unblocking) path which entails the $d$-connection of \cref{eq:psep_compl_2} and the subgraph $\graphname\sn'$ is formed by all the vertices of $\graphname\sn$ but only keeping the edges which contribute to this $d$-connecting path, along with the edges $(\vertname,\postvertname_\vertname)$ and $(\prevertname_\vertname,\postvertname_\vertname)$ for each $\vertname \in \splitvert{\graphname\sn}$ which are present in $\graphname\sn$. Therefore by construction \cref{eq:psep_compl_2} also holds for $\graphname\sn'$, i.e.,  
        \begin{equation}
        \label{eq:psep_compl_3}
            \begin{aligned}
                (v_1&\perp^d v_2|V_3)_{\graphname\sn'} \\
                (v_1&\not\perp^d v_2|V_3\cup\psvertset)_{\graphname\sn'}.
            \end{aligned}
        \end{equation}
     
        This tells us that all paths between $v_1$ and $v_2$ in $\graphname\sn'$ are blocked by $V_3$, but there exists at least one such path that becomes unblocked when additionally conditioning on $\psvertset$. Thus, $\psvertset$ (being composed of childless vertices) act as colliders. 
        This implies that in the subgraph $\graphname\sn'$, there must be an unblocked path from $v_1$ to some $\postvertname\in \psvertset$ as well as an unblocked path from $v_2$ to some $\postvertname'\in \psvertset$ (not necessarily the same as $\postvertname$), and an unblocked path between $\postvertname$ and $\postvertname'$. That is $\exists \postvertname,\postvertname'\in \psvertset$ such that
        \begin{equation}
        \label{eq:psep_compl_4}
            \begin{aligned}
                (v_1 &\not\perp^d \postvertname|V_3)_{\graphname\sn'} \\
                (v_2 &\not\perp^d \postvertname'|V_3)_{\graphname\sn'} \\
                (\postvertname &\not\perp^d \postvertname'|V_3)_{\graphname\sn'}.
           \end{aligned}
        \end{equation}        
        We now construct a functional model $\fcm_{\graphname\sn'}$ on $\graphname\sn'$ (an acyclic graph) and use the $d$-connections of \cref{eq:psep_compl_4} to argue that the random variable $\bar X_1$ of $v_1 \in V_1$, with values $\bar x_1 = x_{\vertname_1}$, and the random variable $\bar X_2$ of $v_2 \in V_2$, with values $\bar x_2 = x_{\vertname_2}$, will become correlated given the outcomes of $V_3$ and $\postvertname,\postvertname'\in \psvertset$ in this functional model. 
        Using \hyperlink{obs1}{\it Observation 1}, this will induce a functional causal model on the teleportation graph $\graphname\sn$ with the same distribution on the shared vertices.     
        Through the connection between cyclic functional models on $\graphname$ and acyclic functional models on $\graphname\sn$ with post-selection on $\psvertset$ established in \cref{sec:fCM_to_afCM}, we will show that this model on $\graphname\sn$ implies a functional causal model $\fcm_{\graphname}$ on the original (possibly cyclic) $\graphname$ and imply the necessary conditional independence $(X_1\not \indep  X_2|X_3)_{\probf_{\graphname}}$ between the outcomes of $V_1$, $V_2$ and $V_3$ in $\fcm_{\graphname}$.
        
        Before defining the functional model $\fcm_{\graphname\sn'}$, we make another important observation.

        \hypertarget{obs2}{}\hyperlink{obs2}{\it Observation 2:}
            Recall that $\graphname\sn'$ is a subgraph of $\graphname\sn\in \graphfamilysn\graphname$ which is a classical teleportation graph of $\graphname$ (\cref{def:sn_graph_family}). Thus, the post-selection vertices, $\postvertname$, occurring in the subgraph $\graphname\sn'$ only occur in the following structure where $\vertname$ and $u$ are vertices belonging to the original graph $\graphname$ as well as $\graphname\sn'$. 
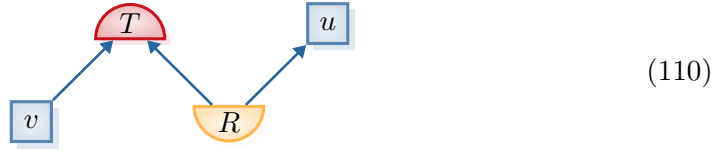
            \begin{equation}
            \label{eq:psep_compl_5}
                \centertikz{
                    \begin{scope}[xscale=2.6,yscale=1.3]
                        \node[onode] (in) at (0,0) {$\vertname$};
                        \node[onode] (out) at (1.5,1) {$u$};
                        \node[prenode] (pre) at (1,0) {$\prevertname$};
                        \node[psnode] (post) at (0.5,1) {$\postvertname$};
                        \draw[cleg] (pre) -- (post);
                        \draw[cleg] (in) --  (post);
                        \draw[cleg] (pre) --  (out);
                    \end{scope}
                }
            \end{equation}
            Indeed, we chose $\graphname\sn'$ to only have the edges present in a $d$-connecting (or unblocked) path associated with the $d$-connection $(X\not\perp^d Y|V_3\cup\psvertset)_{\graphname\sn'}$.
        We now define a generic functional model applicable to any graph as follows. 
        \begin{definition}[A functional model on an arbitrary graph]
            \label{def: proof_cm} 
            Associate a binary outcome $\outcome_{\vertname}\in\outcomemaparg{\vertname}:=\{0,1\}$ to each vertex $\vertname$ of the given graph. Every exogenous vertex $\vertname\in \exnodes$ is associated with a uniform distribution $\probex{\vertname}(\outcome_{\vertname}=0)=\probex{\vertname}(\outcome_{\vertname}=1)=\frac{1}{2}$ and every endogenous vertex $\vertname\in\nexnodes$ is associated with a function $\outcome_{\vertname} = \funcarg{\vertname}\big(\outcome_{\parnodes{\vertname}}\big):=\bigoplus\limits_{u\in \parnodes{\vertname} }\outcome_u$. Here, we have chosen a functional model where the functions deterministically relate each variable to its parental variables in the graph, and have thus applied the simplification allowed by the remark of \cref{remark:errorrv}.
        \end{definition}
        This fully specifies the functional model
        (\cref{def:functional_CM}). In particular, applying the above definition to the acyclic graph $\graphname\sn'$ we will denote the resulting functional model as $\fcm_{\graphname\sn'}$, and the induced model (by \hyperlink{obs1}{\it Observation 1}) on $\graphname\sn$, of which $\graphname\sn'$ is a subgraph, as $\fcm_{\graphname\sn}$\footnote{
            \hyperlink{obs1}{\it Observation 1} here implies that, in case the parent set of a vertex does not coincide between $\graphname\sn$ and $\graphname\sn'$, the association of $\outcome_{\vertname} = \funcarg{\vertname}\big(\outcome_{\parnodes{\vertname}}\big):=\bigoplus_{u\in \parnodes{\vertname} }\outcome_u$ refers to the parent set of $\graphname\sn'$.
        }. 
        
        The functional dependences in the above model encode that the outcome value for each vertex $\vertname$ is given by the sum modulo 2 (denoted by $\bigoplus$) of all the outcomes associated with the parent vertices of $\vertname$. In particular, note that if $u$ has exactly one parent $\vertname$, this causal model imposes $\outcome_u=\outcome_{\vertname}$.
        
        We now argue for the relevant conditional independence using this functional model. 
        We first consider a simple example of a graph $\graphname\sn'$ where \cref{eq:psep_compl_2} and hence \cref{eq:psep_compl_4} are satisfied. We carry out the proof for this simple case  and then argue how the general argument can be reduced to this case.
        
        \begin{itemize}
            \item[] \textbf{Proof for a simple case:} We consider a specific way, using the following three paths, to implement the three $d$-connections of \cref{eq:psep_compl_4} respectively:

        \begin{equation}
        \label{eq: simple_path1}
            \centertikz{\node[onode] (x) at (0,0) {$\vertname_1$};
            \node[onode] (v1) at (1.5,0) {$\tilde\vertname_1$};
             \node[onode] (v2) at (3,0) {$\tilde\vertname_2$};
             \node (d) at (4.5,0) {$\dots$};
              \node[onode] (vk) at (6,0) {$\tilde\vertname_k$};
               \node[psnode] (p) at (6,1.5) {$\postvertname$};
               \draw[cleg] (x)--(v1); \draw[cleg] (v1)--(v2); \draw[cleg] (v2)--(d); \draw[cleg] (d)--(vk); \draw[cleg] (vk)--(p);
             
               }     
        \end{equation}
        
         for some vertices $\{\tilde\vertname_i\}_{i=1}^k$, none of which belong to $V_3$, 
         \begin{equation}
        \label{eq: simple_path2}
            \centertikz{\node[onode] (x) at (0,0) {$\vertname_2$};
            \node[onode] (v1) at (1.5,0) {$\tilde u_1$};
             \node[onode] (v2) at (3,0) {$\tilde u_2$};
             \node (d) at (4.5,0) {$\dots$};
              \node[onode] (vk) at (6,0) {$\tilde u_l$};
               \node[psnode] (p) at (6,1.5) {$\postvertname'$};
               \draw[cleg] (x)--(v1); \draw[cleg] (v1)--(v2); \draw[cleg] (v2)--(d); \draw[cleg] (d)--(vk); \draw[cleg] (vk)--(p);}
        \end{equation}
        
         for some vertices $\{\tilde u_j\}_{j=1}^l$, none of which belong to $V_3$, 
        \begin{equation}
        \label{eq: simple_path3}
           \centertikz{
                    \begin{scope}[xscale=2.6,yscale=1.3]
                        \node[onode] (c) at (1.5,1) {$C$};
                        \node[prenode] (q) at (1,0) {$\prevertname$};
                        \node[psnode] (p) at (0.5,1) {$\postvertname$};
                        \draw[cleg] (q) -- (p);
                        \draw[cleg] (q) --  (c);
        
                         \node[prenode] (q') at (2,0) {$\prevertname'$};
                         \node[psnode] (p') at (2.5,1) {$\postvertname'$};
                           \draw[cleg] (q') -- (c);
                           \draw[cleg] (q') -- (p');
                    \end{scope}
                    }    
        \end{equation}
        
        for some $C\in V_3$.
        Let $\graphname\sn'$ be defined to be formed by these three unblocked paths (i.e. all vertices and edges featuring in \cref{eq: simple_path1}, \cref{eq: simple_path2}, \cref{eq: simple_path3} included) and the remaining vertices of the original graph $\graphname\sn$ included without any of their associated edges.

        Since all the vertices in $\{\tilde v_i\}_{i=1}^k$ and $\{\tilde u_j\}_{j=1}^l$ have exactly one parent in this $\graphname\sn'$, one can immediately see that the functional model $\fcm_{\graphname\sn'}$ for this case imposes $\outcome_{\tilde\vertname_k}=\outcome_{\vertname_1}$ and $\outcome_{\tilde u_l}=\outcome_{\vertname_2}$. $\postvertname$ has two parents, $\tilde\vertname_k$ and $\prevertname$, with the outcomes related as $\outcome_\postvertname=\outcome_{\tilde\vertname_k}\oplus\outcome_\prevertname$ and similarly $\postvertname'$ has two parents, $\prevertname'$ and $\tilde u_l$, with the outcomes related as $\outcome_{\postvertname'}=\outcome_{\prevertname'}\oplus \outcome_{\tilde u_l}$. Finally, we also have $\outcome_C=\outcome_{\prevertname}\oplus\outcome_{\prevertname'}$.

        If we condition on $\outcome_C=0$, then we have $\outcome_{\prevertname}=\outcome_{\prevertname'}$.
        Then notice that $\outcome_\postvertname=\outcome_{\vertname_1}\oplus\outcome_\prevertname$ and $\outcome_{\postvertname'}=\outcome_{\vertname_2}\oplus\outcome_\prevertname$, i.e., given that $\outcome_C=0$, we have $\outcome_\postvertname\oplus \outcome_{\vertname_1}=\outcome_{\postvertname'}\oplus \outcome_{\vertname_2}$. This implies that conditioned on $\outcome_C=\outcome_{\postvertname'}=\outcome_\postvertname=0$, we have perfect correlations $\outcome_{\vertname_1}=\outcome_{\vertname_2}$ between the outcomes of $\vertname_1$ and $\vertname_2$. 
        We can visualize this as follows. Our causal model ensures that when we condition on $\outcome_C=0$, for the purpose of correlations in the model, we can equivalently reason using the following simplified graph\footnote{Even though this does not have the structure of a teleportation graph, the relevant point is that $\postvertname$ and $\postvertname'$ share effectively a common cause $\prevertname$ when conditioned on $\outcome_C=0$.}

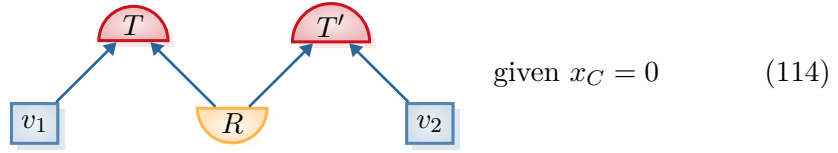
        \begin{equation}
        \label{eq: simple}
           \centertikz{
                    \begin{scope}[xscale=2.6,yscale=1.3]
                        \node[onode] (x) at (0,0) {$\vertname_1$};
                        \node[psnode] (p') at (1.5,1) {$\postvertname'$};
                        \node[prenode] (q) at (1,0) {$\prevertname$};
                        \node[psnode] (p) at (0.5,1) {$\postvertname$};
                        \draw[cleg] (q) -- (p);
                        \draw[cleg] (q) -- (p');
                        \draw[cleg] (x) --  (p);
         
                          \node[onode] (y) at (2,0) {$\vertname_2$};
                           \draw[cleg] (y) -- (p');
                          
                    \end{scope}
                    }    \quad \text{given } \outcome_C=0
        \end{equation}
            Explicitly, recalling that $\graphname\sn'$ is a directed acyclic graph we can directly compute the joint probability distribution of $\fcm_{\graphname\sn'}$ (\cref{def: distribution_functional_cm}). Then, the functional dependence $\outcome_\postvertname\oplus \outcome_{\vertname_1}=\outcome_{\postvertname'}\oplus \outcome_{\vertname_2}$, which holds when $\outcome_C = 0$, translates to the conditional dependence 
        \begin{equation}
           \label{eq:psep_compl_6}
            (\bar X_1\not\indep \bar X_2|\postvertname=0, \postvertname'=0,  C=0)_{\probfacyc_{\graphname\sn'}},
        \end{equation}
        where $\bar X_i$ denotes the random variable associated to the vertex $\vertname_i$ for $i \in \{1,2\}$.
        
        Note that the remaining vertices are causally disconnected in $\graphname\sn$ (in the sense that are not involved in any functional dependencies, except trivially on their error variable), by construction of the way the causal model on the subgraph $\graphname\sn'$ is extended to a causal model on $\graphname\sn$ in \hyperlink{obs1}{\it Observation 1}, and conditioning on the outcomes of such vertices cannot uncorrelate the remaining outcomes in an acyclic functional model. 
        Using this, we see that \cref{eq:psep_compl_6}
        implies that $(\bar X_1\not\indep X_2|V_3 \cup \{\postvertname=0\}_{\postvertname\in\psvertset})_{\probacyc_{\graphname\sn}}$ also holds using \cref{def:conditional independence}, where we recall that $X_i$ is the set of variables associated to the vertex set $V_i$ for $i\in\{1,2\}$.
        This trivially implies, by inclusion, that
        $(X_1 \not \indep X_2|V_3 \cup \{\postvertname=0\}_{\postvertname\in\psvertset})_{\probfacyc_{\graphname\sn}}$ holds.

        Finally, to link the constructed model $\fcm_{\graphname\sn}$ on the teleportation graph $\graphname\sn\in \graphfamilysn\graphname$ of $\graphname$ to a causal model $\fcm_{\graphname}$ on the original (possibly cyclic) graph $\graphname$, it is important to note that our mapping from cyclic functional models to acyclic functional models with post-selection relies on a particular form of post-selection at the post-selection vertices $\psvertset$. The post-selections must form a post-selected teleportation protocol (\cref{def:classical ps tele}) for each pair of pre and post-selection vertices in \cref{eq:psep_compl_5}.

        Notice that the graph $\graphname$ differs from $\graphname\sn$ specifically by replacing each structure of the form \cref{eq:psep_compl_5} with a directed edge $\centertikz{\node[onode] (v) at (0,0) {$\vertname$}; \node[onode] (u) at (1.5,0) {$u$}; \draw[cleg] (v)--(u);}$ and thus, $\graphname$ contains no pre and post-selection vertices, only observed vertices $\centertikz{\node[onode] (v) at (0,0) {$\vertname$}; }$. By post-selecting on $\{t_\postvertname=0\}_{\postvertname\in\psvertset}$ (i.e., $0$ outcome on all post-selection vertices) in the causal model $\fcm_{\graphname\sn}$ defined on $\graphname\sn$, we induce a causal model $\fcm_{\graphname}$ on $\graphname$ where one replaces $r_\prevertname$ (outcome of a pre-selection vertex) appearing in any functional dependence with the outcome $\outcome_{\vertname}$ for the corresponding vertex $\vertname$ (which has an edge to the same $\postvertname$ as $\prevertname$, as in \cref{eq:psep_compl_5}). This way all functional dependences are expressed only in terms of the vertices of $\graphname$, and $\fcm_{\graphname}$ is a fully specified functional model on $\graphname$.

        This establishes that the constructed causal model $\fcm_{\graphname\sn}$ on the acyclic teleportation graph $\graphname\sn$ of $\graphname$ induces a functional model $\fcm_{\graphname}$ on the possibly cyclic graph $\graphname$ where $\probf(\outcome)_{\graphname}=\probfacyc(\outcome|\{t_\postvertname=0\}_{\postvertname\in\psvertset})_{\graphname\sn}$. Here the $t_\postvertname=0$ post-selection plays the same role as the post-selection on $\outcome_\postvertname=1$ in \cref{def: probability distribution cyclic fcm} and $\outcome=\{\outcome_{\vertname}\}_{\vertname\in \vertset}$ is the set of all observed outcomes of the graph. 
        
        Therefore the the conditional dependence $(X_1\not\indep  X_2|X_3\cup \{\postvertname=0\}_{\postvertname\in \psvertset})_{\probfacyc_{\graphname\sn}}$ established above for the acyclic model on the teleportation graph $\graphname\sn$ implies $(X_1\not\indep X_2|X_3)_{\probf_{\graphname}}$, as required.
         \item[] \textbf{Proof for the general case:} We now consider the case of more general paths that can ensure the $d$-connections in \cref{eq:psep_compl_4}, as compared to the paths in \cref{eq: simple_path1}, \cref{eq: simple_path2}, \cref{eq: simple_path3}. Consider the first $d$-connection $(\vertname_1 \not\perp^d \postvertname|V_3)_{\graphname\sn'}$ of \cref{eq:psep_compl_4}. In the simple case of \cref{eq: simple_path1}, we argued that for the correlations between outcomes in our functional model, we could equivalently consider $\vertname_1$ as a direct cause of $\postvertname$ as in \cref{eq: simple}. This used the fact that the $d$-connecting path between $\vertname_1$ and $\postvertname$ in \cref{eq: simple_path1} ensures equal outcomes for all intervening vertices in the path that are not in the conditioning set $V_3$. 
         It is easy to see that this property holds generically for our causal model, independently of the particular $d$-connecting path \cref{eq: simple_path1}, if we condition on $\outcome_{\vertname}=0$ for every $\vertname$ that features in this $d$-connecting path where $\vertname$ or a descendant of it belongs to $V_3$.

         We argue this through examples for simplicity, but the argument generalizes in a straightforward manner. 
         Consider the direct cause graph $\centertikz{\node[onode] (x) at (0,0) {$v_1$}; \node[onode] (v) at (1.5,0) {$\tilde\vertname$};\draw[cleg] (x)--(v);}$, and the common cause graph $\centertikz{\node[onode] (x) at (0,0) {$v_1$}; \node[onode] (u) at (1.5,0) {$\tilde u$}; \node[onode] (v) at (3,0) {$\vertname$};\draw[cleg] (u)--(x); \draw[cleg] (u)--(v); }$ where $\tilde u$ in the latter case does not belong to $V_3$ (otherwise, $v_1$ and $\tilde v$ would be $d$-separated by conditioning on $V_3$). 
         In both these cases, the causal model from \cref{def: proof_cm} ensures $\outcome_{v_1}=\outcome_{\tilde\vertname}$ ($=\outcome_{\tilde u}$). 
         Now consider the collider graph $\centertikz{\node[onode] (x) at (0,0) {$v_1$}; \node[onode] (u) at (1.5,0) {$\tilde u$}; \node[onode] (v) at (3,0) {$\tilde\vertname$};\draw[cleg] (x)--(u); \draw[cleg] (v)--(u); }$ where $\tilde u\in V_3$ (so that $v_1$ and $\tilde\vertname$ are $d$-connected given $V_3$). 
         Our causal model ensures that $\outcome_{\tilde u}=\outcome_{v_1}\oplus\outcome_{\tilde\vertname}$ i.e., if we post-select on $\outcome_{\tilde u}=0$, we have equality $\outcome_{v_1} =\outcome_{\tilde\vertname}$. 
         Even if $\tilde u\not\in V_3$ but $\tilde u$ is a descendant of $u'\in V_3$, this means we have a directed path from $\tilde u$ to $u'$ and conditioning on $\outcome_{u'}=0$ would give $\outcome_{\tilde u}=0$ and thereby  $\outcome_{v_1} =\outcome_{\tilde\vertname}$. Note that as we apply the functional model defined in \cref{def: proof_cm} to the subgraph $\graphname\sn'$ comprised only of edges in a relevant $d$-connecting path, any such colliders $\tilde{u}$ will only have at most two children (two children if it is itself a collider and one child if it is a descendant of a collider) in $\graphname\sn'$ even if this vertex has more parents in $\graphname\sn$.
        
         In a similar manner, we can consider the second $d$-connection $(v_2 \not\perp^d \postvertname'|V_3)_{\graphname\sn'}$ of \cref{eq:psep_compl_4} and obtain that for the purpose of our arguments, we can equivalently consider $v_2$ as a direct cause of $\postvertname'$ as in \cref{eq: simple}.

        Finally for the third $d$-connection $(\postvertname \not\perp^d \postvertname'|V_3)_{\graphname\sn'}$ of \cref{eq:psep_compl_4}, we can apply a similar argument as well. In \cref{eq: simple_path3}, the central observed vertex $C$ acts as a collider. We may also have a situation where it acts as part of a chain.

        \begin{equation}
           \centertikz
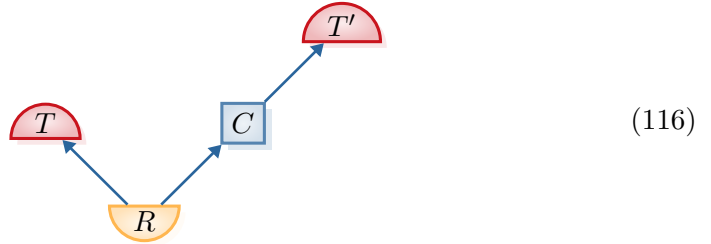
{
                    \begin{scope}[xscale=2.6,yscale=1.3]
        
                        \node[onode] (c) at (1.5,1) {$C$};
                        \node[prenode] (q) at (1,0) {$\prevertname$};
                        \node[psnode] (p) at (0.5,1) {$\postvertname$};
                        \draw[cleg] (q) -- (p);
         
                        \draw[cleg] (q) --  (c);

                         \node[psnode] (p') at (2,2) {$\postvertname'$};
                   
                           \draw[cleg] (c) -- (p');
                    
                    \end{scope}
                    }    
        \end{equation}
        
        In this case, for this to be the $d$-connecting path for the third condition of \cref{eq:psep_compl_4}, we must have $C\not\in V_3$.
        The causal model of \cref{def: proof_cm} for this case entails $\outcome_C=\outcome_\prevertname$, i.e., this is again equivalent to $\postvertname$ and $\postvertname'$ sharing the same common cause $\prevertname$ as in \cref{eq: simple}.
        
        Another case is where $C$ is a common cause of $\postvertname$ and $\postvertname'$ (then $C\not\in V_3$), here we already have the structure of \cref{eq: simple} (it does not matter for our arguments whether the common cause is $C$ or $\prevertname$, the causal model behaves in the same way). More generally, $\postvertname$ and $\postvertname'$ could be $d$-connected through paths that have additional intervening vertices $\postvertname''\in \psvertset$ which are conditioned on. These arguments straightforwardly generalize to that case, post-selecting on the $0$ value of all such vertices again reduces our arguments to a situation of the form of \cref{eq: simple}.

        Therefore all the steps following \cref{eq: simple} also apply to the general case, allowing us to establish the required conditional independence  $(X_1\not\indep  X_2|X_3)_{\probf_{\graphname}}$. 
        
        \end{itemize}
 This completes the proof. \qedhere
  \end{itemize}

\end{proof}

\section{Proofs of~\cref{sec:solvability}}
\label{app:proofs solv}
\austopsucc*
\begin{proof}
    Let us denote with
   \begin{equation}
       \csuccessprob^{\graphname\sn}=\probfacyc\left(\{\postoutcome_\postvertname = 1\}_{\postvertname\in\psvertset}\right)_{\graphname\sn},
   \end{equation}
   the success probability relative to $\graphname\sn\in\graphfamilysn{\graphname}$, with post-selection set $\psvertset$. We first prove the statement for a specific teleportation graph.
   
   Consider the teleportation graph $\graphname_0\in\graphfamilysn{\graphname}$ where $\splitvert{\graphname_0}=\vertset$, i.e., all vertices are split and associated with pre- and post-selection vertices and edges as in~\cref{def: family functional cm}.
   Therefore, the set of vertices of $\graphname_0$ is $\vertset_0=\vertset\cup \{\prevertname_{\vertname}, \postvertname_\vertname\}_{\vertname\in\vertset}$.
   For each $\vertname\in\vertset$ we denote with $(\ctelefunc^\vertname,\cteleprob^\vertname_B,\cteleprob^\vertname_C)$ the augmented post-selected teleportation protocol associated to the corresponding pre and post-selection pair.

   Since all vertices are split, for each vertex $\vertname\in\vertset$ the variable associated to $\vertname$, $\outcome_{\vertname}$ is the value associated to such vertex, obtained through $\funcarg{\vertname}$, and the variable associated to the reconsigning pre-selection vertex $\prevertname_\vertname$, $r_{\vertname}$, acts as input of the functions $\funcarg{\vertname'}$ where $\vertname'\in\childnodes{\vertname}$. Thus, denoting with $t_{\vertname}\in\{0,1\}$ the values associated to $\postvertname_{\vertname}$ and with $r:=\{r_\vertname\in\outcomemaparg{\vertname}\}_{\vertname\in\vertset}$ the values associated to the preselection vertices $\{\prevertname_{\vertname}\}_{\vertname\in\vertset}$, we have (\cref{def: distribution_functional_cm})
   \begingroup
   \crefname{lemma}{\textup{lem.}}{\textup{Lem.}}
   \begin{align}
        \probfacyc(\{\postoutcome_\vertname = 1\}_{\vertname\in\vertset})_{\graphname_0} 
        &= \sum_{r,x}\probfacyc(\outcome, r, \{\postoutcome_\vertname = 1\}_{\vertname\in\vertset})_{\graphname_0} \nonumber\\
        &= \sum_{r,b,x,u}
        \prod_{\vertname\in\vertset} \probex{\vertname} (u_{\vertname})\delta_{\outcome_{\vertname}, \funcarg{\vertname}(r_{\parnodes{\vertname}},u_\vertname)} \delta_{\ctelefunc^{\vertname}(\outcome_\vertname,b_\vertname,r_{\vertname}),1} \cteleprob_B^{\vertname}(b_\vertname) \cteleprob_C^{\vertname}(r_\vertname)  \nonumber\\
        &\hspace{-0.27cm}\stackrel{\cref{lem: copy property aug}}{=} \sum_{r,x,u}
        \prod_{\vertname\in\vertset} \probex{\vertname} (u_{\vertname})\delta_{\outcome_{\vertname}, \funcarg{\vertname}(r_{\parnodes{\vertname}},u_\vertname)} \teleprob^{(\vertname)} \delta_{\outcome_{\vertname},r_{\vertname}} \nonumber\\
    &=  \teleprob^{\text{tot}}\sum_{x,u}
        \prod_{\vertname\in\vertset} \probex{\vertname} (u_{\vertname})\delta_{\outcome_{\vertname}, \funcarg{\vertname}(\outcome_{\parnodes{\vertname}},u_\vertname)}   \nonumber\\
        &= \teleprob^{\text{tot}} \sum_{u}\prod_{\vertname\in\vertset}\probex{\vertname}(u_{\vertname}) \numsol{u} \nonumber\\
    &= \teleprob^{\text{tot}} \avnumsol,
   \end{align}
   \endgroup
   where $\teleprob^{\text{tot}}=\prod_{\vertname\in\vertset}\teleprob^{(\vertname)}$ and $\teleprob^{(\vertname)}$ is the probability of successful post-selection of the implementation chosen for the vertex $\vertname$. 

   Consider an arbitrary teleportation graph $\graphname\sn\in\graphfamilysn{\graphname}$ with $\splitvert{\graphname\sn}\subset \vertset$. In the proof of~\cref{lemma: acyclic_prob_same_func} we related the success probabilities of two teleportation graphs. By applying recursively~\cref{eq:relating prob succ} we get 
   \begin{equation}
       \csuccessprob^{(\graphname\sn)} = \left(\frac{\teleprob^{\text{tot}}}{\prod_{\vertname\notin\splitvert{\graphname\sn}}\teleprob^{(\vertname)}}\right)  \avnumsol = \left(\prod_{\vertname\in\splitvert{\graphname\sn}}\teleprob^{(\vertname)}\right) \avnumsol,
   \end{equation}
   which completes the proof.
\end{proof}

\austous*

\begin{proof}
    Consider a uniquely solvable functional model: it holds that
    \begin{equation}
        \avnumsol =  \sum_{x} \prod_{\vertname\in\vertset}\probex{\vertname}(u_{\vertname})\numsol{u} =  \sum_{x} \prod_{\vertname\in\vertset}\probex{\vertname}(u_{\vertname}) = 1,
    \end{equation}
    where the sum runs over $x = \{x_\vertname \in \outcomemaparg\vertname\}_{\vertname\in\vertset}$.
\end{proof}

\end{document}